\def \eps {\varepsilon}

\documentclass[11pt]{article}

\usepackage{fullpage}
\usepackage{amsfonts}
\usepackage{bbding}
\usepackage{graphicx,float}
\usepackage{amsfonts,amssymb,amsmath,amsthm,amscd}
\usepackage{mathrsfs}
\usepackage{xcolor}
\usepackage{empheq}
\usepackage{adforn}
\usepackage{cancel}
\usepackage{xr-hyper}
\usepackage{xr}
\usepackage{mdframed}
\usepackage{mathdots}
\usepackage{enumerate}
\usepackage{lmodern}
\usepackage{mdframed}
\usepackage{stmaryrd}
\usepackage{blindtext}
\usepackage[all, knot]{xy}
\usepackage{leftidx}
\usepackage{accents}

\definecolor{slightblue}{rgb}{.8, .8, 1}
\definecolor{hair}{RGB}{100,225,190}
\definecolor{ruby}{RGB}{220,50,120}
\definecolor{grass}{RGB}{150,220,110}
\definecolor{ceruleanblue}{rgb}{0.16, 0.32, 0.75}
\definecolor{deepcarmine}{rgb}{0.66, 0.13, 0.24}
\definecolor{otterbrown}{rgb}{0.4, 0.26, 0.13}
\definecolor{sapphire}{rgb}{0.03, 0.15, 0.4}

\usepackage[colorlinks=true,  linkcolor=otterbrown, citecolor=sapphire,  urlcolor=hair!80!black]{hyperref}

\usepackage{soul}

\usepackage[T1]{fontenc}

\newtheorem{theorem}{Theorem}[section] \newtheorem{lemma}[theorem]{Lemma}
\newtheorem{proposition}[theorem]{Proposition} \newtheorem{corollary}[theorem]{Corollary}

\theoremstyle{definition} 
\newtheorem{definition}[theorem]{Definition}

\newtheorem{remark}[theorem]{Remark} \numberwithin{equation}{section}
\numberwithin{figure}{section}

\newcommand{\Cb}{\mathbb{C}}

\newcommand{\Eb}{\mathbb{E}}

\newcommand{\Hb}{\mathbb{H}}

\newcommand{\Nb}{\mathbb{N}}

\newcommand{\Pb}{\mathbb{P}}
\newcommand{\Qb}{\mathbb{Q}}
\newcommand{\Rb}{\mathbb{R}}

\newcommand{\Ub}{\mathbb{U}}

\newcommand{\Zb}{\mathbb{Z}}

\newcommand{\Pf}{\mathbf{P}}

\newcommand{\Ac}{\mathcal{A}}
\newcommand{\Bc}{\mathcal{B}}
\newcommand{\Cc}{\mathcal{C}}
\newcommand{\Dc}{\mathcal{D}}
\newcommand{\Ec}{\mathcal{E}}
\newcommand{\Fc}{\mathcal{F}}
\newcommand{\Gc}{\mathcal{G}}

\newcommand{\Ic}{\mathcal{I}}

\newcommand{\Lc}{\mathcal{L}}
\newcommand{\Nc}{\mathcal{N}}

\newcommand{\Rc}{\mathcal{R}}
\newcommand{\Sc}{\mathcal{S}}
\newcommand{\Tc}{\mathcal{T}}

\newcommand{\Cs}{\mathscr{C}}
\newcommand{\Es}{\mathscr{E}}

\newcommand{\Sfr}{\mathfrak{S}}

\usepackage{sidecap,subcaption}
\usepackage[section]{placeins}
\usepackage{wrapfig}
\usepackage{lpic}

\sidecaptionvpos{figure}{c}

\newcommand{\wt}{\widetilde}
\newcommand{\wh}{\widehat}
\newcommand{\ol}{\overline}

\newcommand{\dimh}{\mathrm{dim}_\mathrm{H}}
\newcommand{\mub}{\mu^{\mathrm{bub}}}
\newcommand{\fr}{\mathrm{fr}}

\newcommand{\area}{\mathrm{Area}}
\newcommand{\dist}{\mathrm{dist}}
\newcommand{\one}{\mathbf{1}}

\begin{document}

\title{Multiple points on the boundaries of Brownian loop-soup clusters}

\author{Yifan Gao
	\thanks{City University of Hong Kong} 
	\and Xinyi Li\thanks{Peking University}
	\and Wei Qian	\footnotemark[1] \thanks{CNRS and Laboratoire de Math\'{e}matiques d'Orsay, Universit\'{e} Paris-Saclay (on leave)}
	}
	\date{}
	
	\maketitle

\begin{abstract}
For a Brownian loop soup with intensity $c\in(0,1]$ in the unit disk, we show that almost surely, the set of simple (resp.\ double) points on any portion of boundary of any of its clusters has Hausdorff dimension $2-\xi_c(2)$ (resp.\ $2-\xi_c(4)$), where $\xi_c(k)$ is the generalized disconnection exponent computed in \cite{MR4221655}. 
As a consequence, when the dimension is positive, such points are a.s.\ dense on every boundary of every cluster. There are a.s.\ no triple points on the cluster boundaries.

 As an intermediate result, we establish a \emph{separation lemma} for Brownian loop soups, which is a powerful tool for obtaining sharp estimates on non-intersection and non-disconnection probabilities in the setting of loop soups.
In particular, it allows us to define a family of \emph{generalized intersection exponents} $\xi_c(k, \lambda)$, and
show that $\xi_c(k)$ is the limit  as $\lambda\searrow 0$ of  $\xi_c(k, \lambda)$.
\end{abstract}

\tableofcontents

%%%%%%%%%%%%%%%%%%%%%%%%%%%%%%%%%%%%%%%%%%%%%%
%%%% Main text entry area:

	\section{Introduction}
	\subsection{Background and overview}
 The Brownian loop soup was introduced by Lawler and Werner in \cite{MR2045953}, and has deep connections with many conformally invariant objects in two-dimensional random geometry. Moreover, the Brownian loop soup itself possesses a complicated structure which is not yet completely understood. 
Sheffield and Werner showed in \cite{MR2979861} that the Brownian loop soup undergoes a phase transition at $c=1$. At a supercritical intensity $c>1$, a Brownian loop soup in the unit disk a.s.\ has one single cluster; at a subcritical or critical intensity $c\in (0,1]$, it a.s.\ has infinitely many clusters, and the outer boundaries of these clusters are distributed as a CLE$_\kappa$, where $\kappa\in(8/3,4]$ and $c\in(0,1]$ are related by
\begin{align}\label{eq:c_kappa}
c(\kappa) =(6-\kappa)(3\kappa-8) /(2\kappa), \quad \kappa(c)=\left(13 -c -\sqrt{(1-c)(25-c)}\right)/3.
\end{align}
The Hausdorff dimension $d_0 (c)$ of the outer boundary of a cluster in a loop soup with intensity $c\in(0,1]$ is equal to the dimension of a CLE$_\kappa$ loop, which is equal to $1+\kappa(c)/8$ by \cite{MR2435854}. As $c\searrow 0$, $d_0(c)$ decreases continuously to $4/3$, which is the dimension of the \emph{frontier} (i.e., outer boundary) of a Brownian motion. 

It was shown by Lawler, Schramm and Werner in \cite{MR1992830} that the Brownian frontier is an SLE$_{8/3}$-type curve, but the dimension of the Brownian frontier was obtained a bit earlier using so-called \emph{Brownian disconnection exponents} which characterize the decay rate of the non-disconnection probabilities of Brownian motion\footnote{When $\beta$ is an integer, the probability that the union of the trace of $\beta$ independent Brownian motions started from the origin and stopped at first exit of $\Bc_r$, the disk of radius $e^r$, do not disconnect the unit disk from infinity decays like $\exp(-r \xi(\beta))$. }
\begin{equation}\label{eq:discexp}
    \xi(\beta)=\frac{1}{48} \left(\left(\sqrt{24\beta+1} - 1\right)^2 -4 \right),
\end{equation} computed in a celebrated work \cite{MR1879851} by the same authors using SLE techniques.

Disconnection exponents play a central role in the computation of the dimension of fractals related to the Brownian frontier. On a heuristic level, conditioned on a point being on the frontier (resp.\ being a double point on the frontier), the forward and backward paths act as two (resp.\ four) independent Brownian motions whose union does not disconnect this point from infinity. In fact, it was proved by Lawler \cite{MR1386292} that the dimension of the Brownian frontier is equal to $2-\xi(2)$, and by Kiefer and M\"orters \cite{MR2644878} that the dimension of double points on the Brownian frontier is equal to $2-\xi(4)$.

As a matter of fact, every point on the Brownian frontier is also on the Brownian motion itself, because the Brownian motion is a.s.\ a continuous curve. In contrast, most points on the outer boundary of a loop-soup cluster are not visited by any Brownian loop, because the dimension of the Brownian frontier $4/3$ is strictly less than $1+\kappa/8$, the dimension of the CLE$_\kappa$ loops.\footnote{In this paper, we only focus on the $c\in(0,1]$ regime. If a point on the outer boundary of a cluster is visited by a Brownian loop, then it is also on the outer boundary of this loop. Therefore, such points belong to the countable union of the outer boundaries of the loops in a Brownain loop soup, which has dimension $4/3$.} However, as shown by \cite[Theorem 1]{MR3994105}, there do exist points on the outer boundary of a cluster which are visited by at least one Brownian loop.
Back in 1965, L\'evy \cite{MR0190953} already showed the existence and denseness of double points on the Brownian frontier.\footnote{The argument roughly says that if there is a portion of the Brownian frontier that contains no double points, then the Brownian path must follow that portion, which is impossible.} However, this argument fails for the loop soup, and we could not find a simple way to decide whether there exist double points on the outer boundaries of loop-soup clusters. It was shown by Burdzy and Werner \cite{MR1387629} that there are no triple points on the Brownian frontier (using the fact that $\xi(6)>2$), but this does not seem to imply whether there exist triple points on the boundaries of loop-soup clusters.

The third author started to investigate whether there exist double points  on the boundaries of clusters in  \cite{MR4221655}, motivated by a question related to the decomposition of the Brownian loop-soup clusters at the critical intensity $c=1$  \cite{MR3994105}.
The article \cite{MR4221655} tried to generalize the strategy used for the Brownian motion to the loop-soup setting, and to compute explicitly the dimension of multiple points on the boundaries of loop-soup clusters.
As a first step, \cite{MR4221655} extended the Brownian disconnection exponent to some \emph{generalized disconnection exponent} $\eta_\kappa(\beta)$, defined for all $\kappa\in(0,4]$ and
$$\beta\ge \max((6-\kappa)/(2\kappa), (4-\kappa)^2/(2\kappa))$$
in terms of \emph{general restriction measures} which were constructed using radial hypergeometric SLE. The values of $\eta_\kappa(\beta)$ for $\kappa$ and $\beta$ in the previous range were then computed. 
For $c\in [0,1]$, if we let $\kappa\in [8/3,4]$ be determined by $c$ via \eqref{eq:c_kappa}, then for $\beta\ge (6-\kappa)/(2\kappa)$,  we have
	\begin{align}\label{eq:exponent}
	\xi_c(\beta) := \eta_\kappa(\beta)=\frac{1}{48} \left(\left(\sqrt{24\beta+1-c} - \sqrt{1-c}\right)^2 -4(1-c) \right).
	\end{align}
For $c=0$, the value of $\xi_0(\beta)$ equals the Brownian disconnection exponent $\xi(\beta)$ (see \eqref{eq:discexp}, but its definition given in \cite{MR4221655}  was very different from that of the Brownian disconnection exponent. 
Nevertheless,  \cite{MR4221655} gave heuristic indications why the exponents $\eta_\kappa(\beta)$ defined therein should be the correct generalization. In particular, if $\beta$ equals an integer $k\in\Nb$, it was informally argued that  $\xi_c(k)$ should be related to the disconnection probability of $k$ Brownian motions inside an independent Brownian loop soup with intensity $c\in(0,1]$.

In this paper, we rigorously establish the relation between $\xi_c (k)$ for $k\in \Nb$ and the disconnection behavior of $k$ Brownian motions inside a loop soup, and compute the dimensions of multiple points on the cluster boundaries, confirming several predictions in \cite{MR4221655}. 
The actual work of making the statements rigorous is far from obvious, and does not follow from \cite{MR4221655}.
The most  innovative and technical part of this work is the proof of a \emph{separation lemma} for the loop-soup setting. This is a generalization of the separation lemma for Brownian motions established by Lawler \cite{MR1386292}, but its statement is much more intricate, and its proof involves considerably more technicalities which are specific to the loop-soup setting (see Section~\ref{subsec:proof_strategy} for a detailed discussion).
Among other things, our separation lemma implies up-to-constants estimates for the generalized non-intersection and non-disconnection probabilities, which allows us to define  the \emph{generalized intersection exponents} $\xi_c(k, \lambda)$  for $k\in\Nb$ and $\lambda>0$, and to show that 
\begin{align}\label{eq0}
\xi_c(k) = \lim_{\lambda\searrow 0}\xi_c(k, \lambda).
\end{align}

In \cite{MR1879851}, the derivation of the Brownian disconnection exponent \eqref{eq:discexp} involves the computation of a more general family of critical exponents, namely \emph{Brownian intersection exponenents}
\begin{equation}\label{eq:interexp}
\xi(\beta_1,\ldots,\beta_k) =\frac{1}{48} \left(\left(\sum_{i=1}^k\sqrt{24\beta_i+1} - k\right)^2 -4 \right),    
\end{equation}
which characterize the non-intersection probability of (packets of) Brownian motions in a fashion similar to disconnection exponents. In fact, the disconnection exponent $\xi(\beta)$ can be viewed as the limit of $\xi(\beta,\iota)$ as $\iota\to 0$. 
Therefore, \eqref{eq0} generalizes this relation between Brownian disconnection and intersection exponents (which correspond to $c=0$ in our case). 

\subsection{Multiple points on the boundaries of clusters}\label{subsec:main_dimension}
In this section, we state our results on the multiple points on the boundaries of clusters in a Brownian loop soup.

Let $\Gamma_0$ be a Brownian loop soup with intensity $c\in(0,1]$ in the unit disk. A \emph{cluster} in $\Gamma_0$ is a subset of $\Gamma_0$ which is equal to an equivalence class under the equivalence relation $\gamma_1 \sim \gamma_0$ if there exists a finite sequence of loops $\gamma_2, \ldots, \gamma_n$ such that $\gamma_i$ intersects $\gamma_{i+1}$ for $1\le i \le n-1$ and $\gamma_n$ intersects $\gamma_0$.
For every cluster $K$ in $\Gamma_0$, we say that $\ell$ is a boundary of the cluster $K$, if $\ell$ is equal to the boundary of a connected component of $\Cb \setminus \cup_{\gamma\in K} \gamma$.
Let $\partial K$ denote the collection of boundaries of $K$.
Every boundary $\ell\in\partial K$ is a Jordan curve\footnote{This follows from the fact that the loop soup is conformally invariant so one can map via an inversion an inner boundary of a cluster to the outer boundary which is a ${\rm CLE}_\kappa$ ($8/3<\kappa\leq 4$) loop -- a Jordan curve.}, hence can be parametrized by a continuous bijection $f$ from the unit circle to $\ell$. We say that $\ell_0$ is a \emph{portion} of $\ell$, if $\ell$ is the image under $f$ of an arc with positive length on the unit circle.
We say that a point is a simple (resp.\ double or $n$-tuple) point of $\Gamma_0$, if it is visited at least once (resp.\ twice or $n$ times) in total by (the same or different) Brownian loop(s) in $\Gamma_0$. Let $\Sc$ and $\Dc$ be respectively the set of simple and double points in $\Gamma_0$.  Let $\dimh(\cdot)$ denote the Hausdorff dimension of a set.
\begin{theorem}\label{main-thm}
		Let $\Gamma_0$ be a Brownian loop soup with intensity $c \in (0,1]$ in the unit disk. The following holds almost surely.
For every cluster $K$ in $\Gamma_0$,  for any portion $\ell_0$ of any boundary $\ell \in \partial K$,
		we have
		\begin{align*}
		\dimh(\ell_0 \cap \Sc)= 2- \xi_c(2), \quad \dimh(\ell_0 \cap \Dc)= 2- \xi_c(4),
		\end{align*}
		where $\xi_c$ is given by~\eqref{eq:exponent}.
There are no triple points on $\ell$.
	\end{theorem}
	
This immediately implies the following corollary. 
\begin{corollary}\label{cor:dense}
Let $\Gamma_0$ be a Brownian loop soup with intensity $c\in(0,1]$ in the unit disk. The following holds almost surely.
If $c\in(0,1]$ (resp.\ $c\in(0,1)$), then for every cluster $K$ in $\Gamma_0$ and every boundary $\ell\in\partial K$, $\ell \cap \Sc$ (resp.\ $\ell \cap \Dc$) is dense on $\ell$. 
\end{corollary}
The dimensions $d_1(c):= 2-\xi_c(2)$ and $d_2(c):= 2-\xi_c(4)$ are both continuous and decreasing in $c\in [0,1]$, contrary to the dimension $d_0(c)$ of the outer boundaries of clusters which is increasing in $c$. For $c=0$, $d_1(0)=d_0(0)$ and $d_2(0)$ are respectively equal to the dimensions of the Brownian frontier and of double points on the Brownian frontier. For $c=1$, $\dimh(\ell \cap \Dc)=0$, and this is not sufficient to determine whether there exist double points on the cluster boundaries in a critical loop soup. The latter is an open question \cite[Question 1.2]{MR4221655} that will be answered in a separate work \cite{nonexistence}.

The most difficult part of the proof is an up-to-constants estimate that we will present in the next subsection. Once we have this estimate, we will apply similar arguments as in \cite{MR2644878}, which were used to get the dimension of double points on the Brownian frontier. (The article \cite{MR1386292} computed the dimension of the set of times at which the Brownian motion is on the frontier, and then used the dimension doubling theorem. This method is suitable for computing the dimension of simple points, but not of double points.)

Finally, to finish the proof of Theorem~\ref{main-thm}, we need to establish a certain zero-one law. The proof of the zero-one laws in \cite{MR1386292} and \cite{MR2644878} for the Brownian motion is not adapted to the loop-soup setting.
Instead, we will make use of the CLE exploration process introduced by Sheffield and Werner \cite{MR2979861}, and more specifically the partial exploration of Brownian loop soups developed in \cite{MR3901648}. 
	
\subsection{Up-to-constants estimate and separation lemma}\label{subsec:intro_separation}

The proof of Theorem~\ref{main-thm} crucially relies on a certain up-to-constants estimate on some non-disconnection probabilities. It allows us to identify the generalized disconnection exponent defined in  \cite{MR4221655} (see Section~\ref{subsec:GDE} for the definition) with the following one, defined using $k$ independent Brownian motions inside a Brownian loop soup (see Figure~\ref{fig:loop soup}).

\begin{figure}
\centering
\includegraphics[scale=.52]{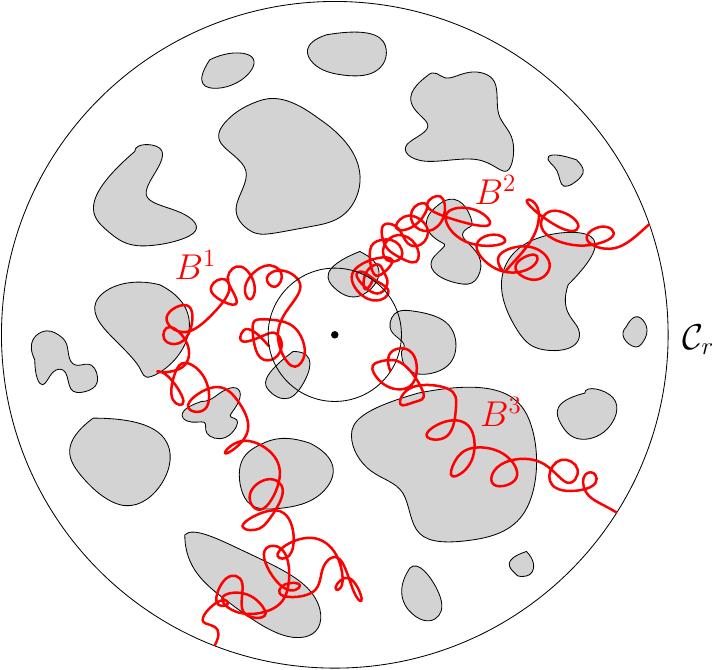}
\caption{We illustrate the event that $ \Theta_r$ does not disconnect $\Cc_0$ from $\infty$ in the case $k=3$. We depict in grey the (filled) clusters of the Brownian loops in $\Gamma_r \setminus \Gamma_0$.}
\label{fig:loop soup}
\end{figure}

Let $\Gamma$ be a Brownian loop soup with intensity $c\in(0,1]$ in the whole plane.
	For all $r>0$, let $\Gamma_r$ be the collection of loops in $\Gamma$ which are contained in the ball $\Bc_r$ of radius $e^r$ centered at $0$.
	Denote $\Cc_r=\partial \Bc_r$.
Let $B^1$, $\cdots$, $B^k$ be $k$ independent Brownian motions started from $k$ uniformly chosen points on $\Cc_0$. For all $r$, define the hitting times 
	\[\tau^i_r=\inf\{ t>0: B^i_t\in \Cc_r\}.\]
	Denote by $\overline B^i_r$ the path $B^i[0,\tau^i_r]$. 
	Let $\Theta_r$ be the union of $\overline B^1_r\cup\cdots\cup \overline B^k_r$ together with the clusters in $\Gamma_r\setminus \Gamma_0$ that it intersects. Let $ p(c,k,r,0)$ be the probability that $ \Theta_r$ does not disconnect $\Cc_0$ from $\infty$.
By the scale-invariance and the Markov property of the Brownian motions and the loop soup, a standard sub-multiplicativity argument (see e.g.\ \cite[Section 2]{MR4221655}) would imply the existence of an exponent $\xi_c(k,0)>0$ such that 
\begin{align}\label{eq:wh-p1}
 p(c,k,r,0) = e^{-r [\xi_c(k,0)+o(1)]}.
\end{align}
In the following, we improve \eqref{eq:wh-p1} to an up-to-constants estimate.
	\begin{theorem}\label{thm:same_exponent}
For all $c\in(0,1]$ and $k\in\Nb$, there exist $\xi_c(k,0)>0$ and $C_2>C_1>0$ such that 
\begin{align}\label{eq:wh-p}
C_1 e^{-r\xi_c(k,0)} \le  p(c,k,r,0) \le C_2 e^{-r\xi_c(k,0)}.
\end{align}
Moreover,  $\xi_c(k,0)$ is the same as  the generalized disconnection exponent $\xi_c(k)$ given by \eqref{eq:exponent}.
\end{theorem}

In order to keep this introduction brief, we will state more related results in Section~\ref{subsec:Statement of the results}. 
For example, we will show that the exponent  $\xi_c(k,0)$ is the limit as $\lambda\searrow 0$ of a certain family of \emph{generalized intersection exponents} $\xi_c(k, \lambda)$. We will also obtain up-to-constants estimates for the \emph{generalized non-intersection probabilities} (Theorem~\ref{thm:xi-k}). This implies that several other natural ways to define the generalized disconnection and intersection exponents are also equivalent.

\smallskip

The main ingredient for proving Theorem~\ref{thm:same_exponent} is a separation lemma for Brownian motions inside an independent  Brownian loop soup. The precise form of this separation lemma is rather technical, so we postpone its statement to Lemma~\ref{lem:separation_lem_2}.
Roughly speaking, in order to obtain the quasi-multiplicativity of certain arm events in successive annuli, one needs to work on an event where these arms are ``well-separated'' at their endpoints (so that there is some space to extend these arms to further scales) and to control the probability of the separation of arms.

An instance of a separation lemma appeared in Kesten's work \cite{MR879034} on 2D percolation. This was a key step that allows one to relate pivotal points for crossing events to four-arm events, and to obtain the quasi-multiplicativity of four-arm events. Later, in a completely different setting, Lawler established a separation lemma for Brownian motions \cite{MR1386292} which implies up-to-constants estimates for disconnection and intersection probabilities of $n$ independent Brownian motions.
This result has multiple implications, including establishing the dimensions and the existence of the Minkowski content of special points inside the Brownian trajectory (e.g.\ \cite{MR1386292,MR1386294,holden2018minkowski}), and proving analyticity of Brownian intersection exponents \cite{MR1961197}.

We mention that it is in fact possible to establish the up-to-constants estimates without proving the separation lemma. 
For Brownian motions, an analogous up-to-constants estimate was first derived in \cite{MR1386292} using the separation lemma, but later another proof was given in \cite{MR1901950} bypassing the use of the separation lemma. The article \cite{MR1901950} also proved another equivalent form of the separation lemma for Brownian motions, formulated in terms of extremal distances, which is arguably the simplest set-up to work with in dimension two.

In this paper, we decide to establish the separation lemma (Lemma~\ref{lem:separation_lem_2}) which is a stronger result than the up-to-constants estimates, because it has numerous other applications. For example, it will allow us to obtain sharper estimates than Theorem~\ref{thm:same_exponent}, the analyticity in $\lambda$ of the generalized intersection exponent $\xi_c(k,\lambda)$, and the existence of Minkowski content of certain special points. We plan to pursue these questions in  future works.
To prove the separation lemma, we continue to use the set-up of extremal distances in \cite{MR1901950}, but have to overcome many additional difficulties due to the presence of the loop soup. 
In Section~\ref{subsec:proof_strategy}, we describe our general proof strategy of the separation lemma for loop soups, and make a detailed account of the main differences with the usual case (without a loop soup).

\subsection{Structure of the paper}
In Section \ref{sec:pre}, we introduce the notation and discuss some basic properties of the extremal distance and Brownian motions, as well as the generalized disconnection exponents defined in \cite{MR4221655}. 
We prove Theorem~\ref{thm:same_exponent} in Section \ref{sec:up-to-constant}, assuming a crucial ``separation lemma'' (Lemma~\ref{lem:separation_lem_2}) whose proof is deferred to the next section. Section~\ref{sec:separation-lemma} is dedicated to the proof of the said separation lemma in the loop-soup setting. 
In Section~\ref{sec:dim}, we show that with positive probability, multiple points on the outer boundaries of outermost clusters in a Brownian loop soup have the predicted Hausdorff dimension.  
In Section~\ref{sec:0-1}, we prove a zero-one law which implies the almost sure dimension and the denseness of such points.
Finally, we complete the proof of Theorem~\ref{main-thm} in Section~\ref{sec:main}.

\section*{Acknowledgement}
YG and XL acknowledge the support by National Key R\&D Program of China (No.\ 2021YFA1002700 and No.\ 2020YFA0712900) and NSFC (No.\ 12071012). 
WQ acknowledges the support by the National Science Foundation under Grant No.\ 1440140, while she was in residence at the Simons Laufer Mathematical Sciences Institute in Berkeley, California, during the spring semester of 2022. We thank Pierre Nolin for useful comments on an earlier version of the paper.

	\section{Preliminaries}\label{sec:pre}
	In this section, we fix some notation and recall a range of preliminary results.
	\subsection{Notation}
	Let $\Nb$ be the set of natural numbers $\{1,2,\cdots\}$ and $\Nb_0:=\{0\}\cup \Nb$.
	Let $\Cb$ be the complex plane. Let $\Ub$ be the unit disk. 
	For a set $A$ in the plane, let $\partial A$ be the boundary of $A$ and $\ol A$ be the closure of $A$. 
For $x\in \Cb$, write $|x|$ for the Euclidean norm of $x$. For two sets $A$ and $B$ in $\Cb$, define $\dist(A,B):=\inf_{x\in A, y\in B} |x-y|$. We abbreviate $\dist(\{x\},B)$ as $\dist(x,B)$.
	
	Let $\Bc(x,R)$ be the open ball of radius $R$ around $x$. We work on exponential scales and introduce the abbreviations $\Bc_r(x)$ for $\Bc(x,e^r)$, $\Cc_r(x)$ for $\partial \Bc(x,e^r)$. For all $s<r$, let $\mathcal A(x,s,r):=\Bc_r(x)\setminus\ol\Bc_s(x)$ be the open annulus bounded between $\Cc_s(x)$ and $\Cc_r(x)$. We omit $x$ if $x=0$. For example, $\Cc_r$ is the abbreviation for $\Cc_r(0)$.

	We say that $A$ is not \emph{disconnected} from infinity by $B$ if there exists a continuous curve $\gamma: [0,\infty)\rightarrow\Cb$ such that $\gamma(0)\in A$, $\gamma[0,\infty)\cap B=\emptyset$ and $\gamma(t)\rightarrow\infty$ as $t\rightarrow\infty$. We say that $B$ does not disconnect (resp.\ disconnects) $A_1$ from $A_2$ in $D$, if there exists (resp.\ does not exist) a continuous curve $\gamma: [0, t]\rightarrow \ol D$ such that $\gamma(0)\in A_1$, $\gamma(t)\in A_2$ and $\gamma[0,t]\cap B=\emptyset$. 

For two positive functions $f(x)$ and $g(x)$, we write
	\begin{itemize}
		\item $f(x)\approx g(x)$ if $g(x)\neq 1$ and $\lim_{x\rightarrow\infty}\log f(x)/\log g(x)=1$;
		\item $f(x)\lesssim g(x)$ if there exists $C>0$ such that $f(x)\le C g(x)$ for all $x$;
		\item $f(x)\asymp g(x)$ if $f(x)\lesssim g(x)$ and $g(x)\lesssim f(x)$.
	\end{itemize}

Since $c$ has been used as the intensity of the Brownian loop soup, we avoid using $c$ to denote a constant. We will use $C, C', C_1, C_2, \ldots$ to denote arbitrary universal positive constants which may change from line to line. If a constant depends on some other quantity, this will be made explicit. For example, if $C$ depends on $\lambda$, we write $C(\lambda)$.

	\subsection{Extremal distance}\label{subsec:ed}
	Let $O$ be an open set and let $V_1,V_2$ be two subsets of $\ol O$. Let $\Upsilon$ be the set of rectifiable curves that disconnect $V_1$ from $V_2$ in $O$. For any Borel measurable function $\rho$, define the $\rho$-area of $O$ as
	\[
	\area_{\rho}(O):=\iint_{O}\rho(x+iy)^2 dx dy,
	\]
	and the $\rho$-length of any rectifiable arc $\gamma$ as
	\[
	\ell_{\rho}(\gamma):=\int_{\gamma}\rho(z)d |z|.
	\]
	We introduce the minimal length 
	\[
	\ell_{\rho}(\Upsilon):=\inf_{\gamma\in\Upsilon}\ell_{\rho}(\gamma).
	\]
	We say $\rho$ is \textit{admissible} if $\ell_{\rho}(\Upsilon)\ge 1$. The \textit{$\pi$-extremal distance} between $V_1$ and $V_2$ in $O$ is defined as $\pi$ times the extremal distance in the following way
	\begin{equation}\label{eq:def-ex}
	L\left(O ; V_{1}, V_{2}\right):=\pi \inf _{\rho} \area_{\rho}(O)=\pi\sup_{\rho}\ell_{\rho}(\Upsilon)^2,
	\end{equation}
	where the infimum is taken over all admissible $\rho$'s and the supremum is taken over all $\rho$'s subject to the condition $0<\area_{\rho}(O)\le 1$. If an admissible $\rho$ achieves the infimum in \eqref{eq:def-ex}, we call the corresonding $\rho$-metric the \emph{extremal metric} finding the extremal distance between $V_1$ and $V_2$ in $O$.
	The $\pi$-extremal distance is \textit{conformally invariant} (see \cite{MR0357743}), i.e., 
	if $O$ is mapped conformally onto $O'$, and the sets $V_1, V_2$ are mapped to the sets $V'_1, V'_2$, then $L\left(O ; V_{1}, V_{2}\right)=L\left(O' ; V'_{1}, V'_{2}\right)$. 
	
	We first give some simple examples of extremal distances and metrics.
If $\partial_1, \partial_2$ are the two vertical sides of the rectangle $\mathcal{R}_{L}:=(0, L) \times(0, \pi)$, then the $\pi$-extremal distance between $\partial_1$ and $\partial_2$ in $\mathcal{R}_{L}$ is $L$ (see \cite[Example 4-1]{MR0357743}). In this case, $\rho_0\equiv 1$ gives the extremal metric. For a general topological quadrangle, one finds the extremal distance and metric by mapping it to a rectangle with appropriate aspect ratio.
For instance, if $W=\Ac(s,r)\cap\{ z=i e^\alpha: |\alpha|<\theta\}$ is a wedge of angle $2\theta\in(0,2\pi)$, then the mapping $z\mapsto \log z$ transforms $W$ into the rectangle $(s,r)\times (-\theta,\theta)$. Therefore
	\begin{equation}\label{eq:wedge-ex}
	L(W; \partial W\cap\Cc_s,\partial W\cap\Cc_r)=\frac{\pi}{2\theta}(r-s).
	\end{equation}

	In the following, we collect some results about extremal distances. 
	
	\begin{lemma}[{The comparison principle, \cite[Corollary of Theorem 4-1]{MR0357743}}]\label{lem:comparison-principle}
		The $\pi$-extremal distance $L\left(O ; V_{1}, V_{2}\right)$ decreases when $O,V_1$ and $V_2$ increase.
	\end{lemma}
	
	\begin{lemma}[{The composition laws, \cite[Theorem 4-2]{MR0357743}}]\label{lem:composition-laws}
		Let $O$ be a topological quadrangle in the plane with four marked sides $\partial_i, 1\le i\le 4$ such that $\partial_1, \partial_2$ are opposite. If $\gamma$ is a simple curve in $O$ connecting $\partial_3$ and $\partial_4$, $O'$ is the connected component of $O\setminus\gamma$ whose boundary contains $\partial_1$, $O''$ is the component of $O\setminus\gamma$ whose boundary contains $\partial_2$, then 
		\[
		L\left(O ; V_{1}, V_{2}\right)\ge L\left(O' ; \partial_{1}, \gamma\right)+L\left(O'' ; \gamma, \partial_{2}\right).
		\]
	\end{lemma}
	
Let us give the definition of \emph{good} domains and their boundaries.
	\begin{definition}[Good domains]\label{def:good-domain}
		Suppose that $\Bc'\subset\Bc''$ are two open balls such that $d(\partial\Bc',\partial\Bc'')>0$. Suppose that $D$ is a simply connected domain in $\Bc''\setminus\Bc'$, and define $\partial_1:=\partial D\cap\Bc'$ and $\partial_2:=\partial D\cap\Bc''$. We say that $D$ is a good domain if both $\partial_1$ and $\partial_2$ are (continuous) arcs of positive length. If $D$ is a good domain, let $\partial_3, \partial_4$ be the other two parts of $\partial D$ such that $\partial_1,\partial_3,\partial_2,\partial_4$ are in counterclockwise order. In this paper, if we label the boundaries of a good domain, we always use subscripts $1,3,2,4$ to order indices in a counterclockwise manner. 
	\end{definition}

	The following is a reformulation of \cite[Lemma 2.3]{MR1901950} for the annulus setting.
	\begin{lemma}[{\cite[Lemma 2.3]{MR1901950}}]\label{lem:reverse-composition}
		Let $0<s-\delta<s+\delta<r$. Let $D$ be a good domain associated with $\Bc_0$ and $\Bc_r$ as above. Let $\partial_1,\partial_3,\partial_2,\partial_4$ be the four parts of $\partial D$. Suppose that $\partial_3\cap\Ac(s-\delta,s+\delta)$ and $\partial_4\cap\Ac(s-\delta,s+\delta)$ are both of diameter smaller than $\delta^{1/6} e^s$ and at distance at least $\delta^{1/7} e^s$ from each other. Let $V$ denote the segment in $D\cap\Cc_s$ that disconnects $\partial_1$ and $\partial_2$ in $D$. Then, there exist some $C(\delta)\ge 0$ such that for all $s,r$,
		\begin{equation}
		L(D;\partial_1,\partial_2)\le L(D\cap\Ac(0,s);\partial_1,V)+L(D\cap\Ac(s,r);V,\partial_2)+C(\delta).
		\end{equation}
	\end{lemma}

	The following lemma is a generalized version of \cite[Lemma 2.1]{MR1901950}. 
    \begin{lemma}\label{lem:D'}
    	Suppose that $D$ is a good domain and $\partial_2$ is a counterclockwise arc from $z_1$ to $z_2$. For $i=1,2$, let $W_i'$ and $W_i$ be two pairs of simply connected sets such that $z_i\in W_i'$, $W_i'\subseteq W_i$, $d_i:=\dist(\partial W_i',\partial W_i)>0$ and $W_1\cup W_2$ does not disconnect $\partial_1$ from $\partial_2$ in $D$. For each $i$, denote $a_i:=\area(\partial W_i\setminus\partial W_i')$. Let $D'=D\setminus(W_1'\cup W_2')$, $\partial'_1=\partial_1$ and $\partial'_2=\partial_2\setminus (W_1'\cup W_2')$. Then,
    	\begin{equation}\label{eq:D'}
    	L(D;\partial_1,\partial_2)\le L(D';\partial_1',\partial_2')\le L(D;\partial_1,\partial_2)+\pi\left( \frac{a_1}{d_1^2}+\frac{a_2}{d_2^2} \right).
    	\end{equation}
    \end{lemma}

    \begin{proof}
    	The first inequality is a direct consequence of Lemma~\ref{lem:comparison-principle}, so we focus on the second inequality. Note that $D'$ is also a good domain and we denote by $\partial_i'$, $1\le i\le 4$ the four parts of $\partial D'$. Suppose $\rho$ is the admissible metric with respect to the curves connecting $\partial_3$ to $\partial_4$ in $D$ such that $L(D;\partial_1,\partial_2)=\pi \area_{\rho}(D)$. From \eqref{eq:def-ex}, we see that every admissible metric $\rho'$ associated with $\partial_3'$ and $\partial_4'$ will give us an upper bound $\pi \area_{\rho'}(D')$. Since $W_i'$ is away from disconnecting $\partial_1$ from $\partial_2$ in the sense that enlarging them to $W_i$ will not cause disconnection, we can construct an admissible metric $\rho'$ from $\rho$ by modifying it around $W_i'$, namely, letting
        \[
        \rho'=\max\{ \rho, d_1^{-1} \one_{W_1\setminus W_1'}+d_2^{-1} \one_{W_2\setminus W_2'} \}.
        \]
        It is readily to see $\rho'$ is admissible. Hence, 
        \[
        L(D';\partial_1',\partial_2')
        \le \pi \area_{\rho'}(D')
        \le L(D;\partial_1,\partial_2)+\pi\left( \frac{a_1}{d_1^2}+\frac{a_2}{d_2^2} \right).
        \]
        This completes the proof.
    \end{proof}

We are now at the point to present the most complicated case, which allows us to write the extremal distance into the sum of two by splitting the domain at an intermediate scale. The statement is quite lengthy since we need to show that it is still true even if we perturb the domain at an intermediate scale. The proof requires a combination of previous lemmas.

We start with notation. Readers are referred to Figure \ref{fig:very-nice} for an illustration of the various geometric objects defined below. Let $s,r\ge 1$. Suppose $D$ is a good domain in $\Ac(0,s+r+2)$ with boundaries $\partial_i$'s such that $\partial_3\cap \Ac(s-1/5,s+11/5)$ is contained in the wedge 
$$W_1':=\{ z: |\arg(z)|\le 1/15\}\cap\Ac(s-1/5,s+11/5),$$ and $\partial_4\cap \Ac(s-1/5,s+11/5)$ is contained in the wedge 
$$
W_2':=\{ z: |\arg(z)-\pi|\le 1/15\}\cap\Ac(s-1/5,s+11/5).
$$
Let 
$$
W_i:=\{ z: |\arg(z)+(1-i)\pi|\le 1/10\}\cap\Ac(s-1/2,s+1/2),\; i=1,2
$$
be two larger wedges opposite to each other. Suppose that $W_1\cup W_2$ does not disconnect $\partial_1$ from $\partial_2$ in $D$. Let $D'$ and $D''$ be two good domains in $\Ac(0,s)$ and $\Ac(s+2,s+r+2)$ respectively such that
$$
D'\cap\Ac(0,s-1/5)=D\cap\Ac(0,s-1/5)
$$
and  $$
D''\cap\Ac(s+11/5,s+r+2)=D\cap\Ac(s+11/5,s+r+2).
$$
Let $\partial'_i$'s and $\partial''_i$'s be the boundaries of $D'$ and $D''$ respectively. We further suppose that $$(\partial'_3\cup \partial''_3)\cap \Ac(s-1/5,s+11/5)\subseteq W'_1\;\mbox{ and }\;(\partial'_4\cup \partial''_4)\cap \Ac(s-1/5,s+11/5)\subseteq W'_2.
$$
\begin{lemma}\label{lem:concatenation}
	There exist universal constants $C_1,C_2>0$ such that the following holds,
	\[
	L(D';\partial'_1,\partial'_2)+L(D'';\partial''_1,\partial''_2)+C_1\le L(D;\partial_1,\partial_2) \le L(D';\partial'_1,\partial'_2)+L(D'';\partial''_1,\partial''_2)+C_2.
	\]
\end{lemma}
\begin{proof}
	Let $V'$ (resp. $V''$) be the segment in $D\cap\Cc_s$ (resp. $D\cap\Cc_{s+2}$) that disconnects $\partial_1$ from $\partial_2$ in $D$. Let $Q'$ (resp. $Q''$) be the inner (resp. outer) connected component of $D\setminus V'$ (resp. $D\setminus V''$). 
	By Lemmas~\ref{lem:comparison-principle} and \ref{lem:composition-laws}, 
	\[
	L(Q';\partial_1,V')+L(Q'';V'',\partial_2)\le L(D;\partial_1,\partial_2).
	\]
	By using Lemma~\ref{lem:reverse-composition} repeatedly, we obtain that for some constant $C$,
	\[
	L(D;\partial_1,\partial_2)\le L(Q';\partial_1,V')+L(Q'';V'',\partial_2)+ C.
	\]
	Therefore, it remains to show $L(Q';\partial_1,V')$ (resp. $L(Q'';V'',\partial_2)$) is equal to $L(D';\partial'_1,\partial'_2)$ (resp. $L(D'';\partial''_1,\partial''_2)$) up to additive constants, which is guaranteed by Lemma~\ref{lem:D'}. More precisely, if we consider the domain $\wt D$ obtained by subtracting $W'_1\cup W'_2$ from $Q'$ (which is the same if $Q'$ is replaced by $D'$), and note that the two larger wedges $W_1$ and $W_2$ considered here play roles in the counterparts in Lemma~\ref{lem:D'}, we immediately get that both $L(Q';\partial_1,V')$ and $L(D';\partial'_1,\partial'_2)$ are equal to $L(\wt D;\partial\wt D\cap \Cc_0,\partial\wt D\cap \Cc_s)$ up to additive constants. The same argument applies also to $L(Q'';V'',\partial_2)$ and $L(D'';\partial''_1,\partial''_2)$, which concludes the proof of Lemma~\ref{lem:concatenation}.
\end{proof}

To be self-contained, we record Lemma 2.2 in \cite{MR1901950} that will be used in the proof of Corollary~\ref{cor:xi-2}.

\begin{lemma}[{\cite[Lemma 2.2]{MR1901950}}]\label{lem:segment}
	Suppose $D$ is a good domain in $\Ac(0,r)$ with boundaries $\partial_i$'s.
	For all $\delta>0$, there exists $C(\delta)$ such that if $V\subseteq \partial_1$ is a segment of length at least $\delta$, if $\dist(V,\partial_3\cup\partial_4)>\delta$ and if the $\delta$-neighborhood of $V$ disconnects $\partial_3$ from $\partial_4$ in $D\cap \Ac(0,\delta)$, then
	\[
	L(D;V,\partial_2)\le L(D;\partial_1,\partial_2)+C(\delta).
	\]
\end{lemma}

	\subsection{Brownian motions and excursions}\label{subsec:facts}

	In this subsection, we recall some facts about planar Brownian motions and excursions that will be used later. 
 
	We first recall the measures on Brownian paths introduced in Section 3.3 of \cite{MR2045953}. Note that in this subsection and subsequent applications involving path measures, we always require domains to have \emph{piece-wise analytic} boundaries. 
 Given a domain $D$ with piece-wise analytic boundaries and $x, y \in D$, one can define the \emph{interior-to-interior measure} $\mu^D_{x,y}$ on paths from $x$ to $y$ in $D$ which has total mass given by the Green's function $G_D(x,y)$. In this paper, the Green's function is normalized so that when $D$ is the unit disk, 
	\begin{equation}\label{eq:Green}
	G_D(0,y)=-\frac{1}{\pi} \log |y|.
	\end{equation}
	Moreover, the measure $\mu^D_{x,y}$ satisfies reversibility, i.e., one can sample $\gamma\sim\mu^D_{x,y}$ by first sampling $\gamma'$ according to $\mu^D_{y,x}$ then reversing $\gamma'$ to get $\gamma$.
	By letting the starting point or both the starting and ending points tend to some fixed boundary points of $D$, and rescaling the interior-to-interior measure $\mu^D_{x,y}$ in a proper way, one can defined the \emph{interior-to-boundary measure} or \emph{boundary-to-boundary measure} respectively. 
	(Whenever we consider measures associated with boundary points, we always assume $\partial D$ is an analytic curve in the neighborhood of these points.) 
	With slight abuse of notation, we still denote these two kinds of measures by $\mu^D_{x,y}$ for $x\in D$, $y \in \partial D$ or for $x, y \in \partial D$ corresponding to interior-to-boundary measure or boundary-to-boundary measure respectively. There is no confusion as soon as the location of the starting and ending points are specified.
	Furthermore, under such a rescaling, the interior-to-boundary measure has total mass 
	\begin{align*}
		|\mu^D_{x,y}| =H_D(x,y) \quad \text{ for } x\in D, y\in \partial D, 
	\end{align*}
	where $H_D(x,y)$ is the Poisson kernel between $x$ and $y$ in $D$, and $H_{\Bc_0}(0,y)=1/(2\pi)$ for all $y\in \Cc_0$. Similarly, the total mass of boundary-to-boundary measure is given by the ``boundary Poisson kernel'' (see \cite{MR2045953} for details).
	
	For $x\in D$, let $\Pb^x$ be the law of a Brownian motion $B$ started from $x$ and stopped at time $\tau$, its first hit on $\partial D$. Then, $\Pb^x$ can be written as the following integral of interior-to-boundary measure 
	\begin{equation}\label{eq:int-itb}
		\Pb^x=\mu_{x,\partial D}:=\int_{\partial D} \mu_{x,y}^D \, dy,
	\end{equation}
where $dy$ denotes the arc length on $\partial D$. Note that  $|\mu_{x,\partial D}|=1$. If  $x_0\in \partial D$, then $\mu_{x_0,\partial D}:=\int_{\partial D} \mu_{x_0,y}^D \, dy$ is the integration of boundary-to-boundary path measures and in this case it still holds that 
\begin{equation}\label{eq:totalbdry}
|\mu_{x_0,\partial D}|=1.
\end{equation}

In the following, we record an important property of path measures, namely the last exit path decomposition (see e.g.\ Section~5.2, in particular (5.5) of \cite{MR2129588}). 
For $x\in D'\subset D$ and $y\in\partial D$, by decomposing the path from $x$ to $y$ in $D$ according to its last exit from $D'$ (using the reversibility of path measures and the strong Markov property for the reversed path), we obtain
\begin{equation}\label{eq:led}
    \mu^D_{x,y}=\int_{\partial D'} \mu^D_{x,z}\oplus \mu^{D\setminus D'}_{z,y} \, dz.
\end{equation}
See \cite[pp.~103--104]{MR2129588} for a proof of this formula. This formula implies that the Brownian path before and after the last exit  from $D'$ are ``independent'' when ``conditioned'' on the location of the last exit.   
If we let $\sigma$ be the last exit time of $B$ from $D'$ before $\tau$ (the first time $B$ hits $\partial D$), then the law of $B[0,\sigma]$ conditioned on the location of $B(\sigma)$ is given by the normalized interior-to-interior measure:
\begin{equation}\label{eq:cled}
	\Pb^x(B[0,\sigma]\in \cdot \mid B(\sigma)=z)=\mu^D_{x,z}/|\mu^D_{x,z}|,
\end{equation}
which we will refer to as the \emph{Brownian link} (sometimes abbreviated as ``links'') from $x$ to $z$ in $D$. Note that this definition does not depend on the choice of $D'$ (as long as $z\in \partial D'$. We can derive \eqref{eq:cled} and this fact from \eqref{eq:led} applied to \eqref{eq:int-itb}: we have
$$
\Pb^x(B[0,\sigma]\in \cdot \mid B(\sigma)=z)  = 
\frac{\mu^D_{x,z}(\cdot)\int_{\partial D}  |\mu^{D\setminus D'}_{z,y}| \, dy}{\int_{\partial D} |\mu^D_{x,z}| \cdot |\mu^{D\setminus D'}_{z,y}| \, dy}
$$
and by applying \eqref{eq:totalbdry} we obtain \eqref{eq:cled} as desired. We also remark that similar last-exit decompositions exist for other types of path measures as well and can be proved in a similar fashion.

\smallskip

Below we give three equivalent definitions of the probability measure $\mu_{0,r}^{\#}$ on Brownian excursions between  $\Cc_0$ and $\Cc_r$ in the annulus $\Ac(0,r)$.

\begin{enumerate}[(i)]
\item We define the following measure
\[
\mu_{0,r}=\int_{\Cc_0}\int_{\Cc_r}\mu^{\Ac(0,r)}_{x,y}\, dx \,  dy.
\]
Let  $\mu_{0,r}^{\#}:=\mu_{0,r}/ |\mu_{0,r}|$ be the normalized probability measure.
\item 
Let $Y(t):=B(t+\sigma_{0,r})$ for all $0\le t\le \tau_r-\sigma_{0,r}$, where $B$ is a standard planar Brownian motion started from $0$,  $\tau_r$ is the first time $B$ hits $\Cc_r$, and $\sigma_{0,r}$ is the last time $B$ visits $\Cc_0$ before $\tau_r$.  The process $(Y(t), 0\le t\le \tau_r-\sigma_{0,r})$ is distributed as $\mu_{0,r}^{\#}$.
We also call it the excursion induced by the Brownian motion $B$.

\item\label{bessel} Let $Y(t):= \exp(S_1(H(t)) + iS_2(H(t)))$, where $S_1$ is a three-dimensional Bessel process started from $0$, $S_2$ is an independent linear Brownian motion started uniformly on $[0,2\pi]$, and $H(t)=\int_0^t |Y(s)|^{-2}\,ds$ is the usual time-change, see the skew-product representation in \cite[Theorem 7.26]{MR2604525}. Let $T_r$ be the first time that $S_1(H(t))
$ hits $r$. The process $(Y(t), 0\le t\le T_r)$ is distributed as $\mu_{0,r}^{\#}$. 
\end{enumerate}

\begin{remark}\label{rem:adv}
	Compared with (ii) where the process $Y$ in (ii) is only defined for a finite time interval which depends on $r$,
	the definition \eqref{bessel} has the advantage that $Y$ is defined for all $t\ge 0$.
\end{remark}

We next use the path decomposition to establish the equivalence of definitions (i) and (ii).
Plugging \eqref{eq:led} into \eqref{eq:int-itb} and setting $D'=\Bc_0$, $D=\Bc_r$, we know that the law of $B[\sigma_{0,r},\tau_r]$ is given by
\begin{equation*}
	\wh\mu_{0,r}=\int_{\Cc_0}\int_{\Cc_r} |\mu_{0,x}^{\Bc_r}| \,  \mu^{\Ac(0,r)}_{x,y} \, dx \, dy=\frac{r}{\pi} 	\int_{\Cc_0}\int_{\Cc_r} \mu^{\Ac(0,r)}_{x,y} \, dx \, dy,
\end{equation*}
where we used \eqref{eq:Green} to get  $|\mu_{0,x}^{\Bc_r}|=G_{\Bc_r}(0,x)=r/\pi$ for all $x\in \Cc_0$. Moreover, since $|\mu_{0,\Cc_r}|=1$ and $|\mu_{0,\Cc_r}|=|\mu_{0,x}^{\Bc_r}| |\mu_{0,r}|$. This implies that $\wh\mu_{0,r}=\mu_{0,r}^{\#}$.

	Now, we collect several lemmas which control the initial part of a Brownian motion.

	\begin{lemma}\label{lem:intermediate-part}
	    Let $W_1=\{z: |\arg(z)|\le 1/20 \}\cap \Ac(0,2)$ and $W_2=\{z: |\arg(z)|\le 1/18 \}\cap \Ac(-1/10,21/10)$ be two wedges with $W_1\subset W_2$. There is a universal constant $C$ such that for all $z_1\in W_1\cap\Cc_0$, $z_2\in W_1\cap \Cc_2$ and $r\ge 3$,
	    \begin{equation}\label{eq:ip-1}
	    \Pb^{z_1}( B[0,\sigma_{2,r}]\in W_2 \mid B(\sigma_{2,r})=z_2)\ge Cr^{-1},
	    \end{equation}
	    where $B$ is a planar Brownian motion started from $z_1$ and $\sigma_{2,r}$ is the last time $B$ visits $\Cc_2$ before hitting $\Cc_r$, and moreover, for all $s>1/10$, 
	    \begin{equation}\label{eq:ip-2}
	    \Pb^{z_1}( B[0,\sigma_{2,r}]\in W_2 \mid B(\sigma_{2,r})=z_2, \tau_{-s}>\tau_r)\ge C(r^{-1}\vee s^{-1}),
	    \end{equation}
    where $\tau_r$ is the first time $B$ hits $\Cc_r$.
	\end{lemma}
\begin{proof}
	Note that the Green's functions $G_{W_2}(z_1,z_2)\ge C_1$ and $G_{\Bc_r}(z_1,z_2)\le C_2 r$ for some universal constants $C_1,C_2$. 
	By \eqref{eq:cled},
	\[
	\Pb^{z_1}( B[0,\sigma_{2,r}]\in W_2 \mid B(\sigma_{2,r})=z_2)
	=\frac{G_{W_2}(z_1,z_2)}{G_{\Bc_r}(z_1,z_2)}
	\ge C_1C_2^{-1}r^{-1},
	\]
	which concludes the proof of \eqref{eq:ip-1}. As for \eqref{eq:ip-2}, we note that $G_{\Ac(-s,r)}(z_1,z_2)\le C_2 (s\wedge r)$ and  
	\[
	\Pb^{z_1}( B[0,\sigma_{2,r}]\in W_2 \mid B(\sigma_{2,r})=z_2, \tau_{-s}>\tau_r)=\frac{G_{W_2}(z_1,z_2)}{G_{\Ac(-s,r)}(z_1,z_2)}
	\ge C_1C_2^{-1}(r^{-1}\vee s^{-1}).
	\]
	This finishes the proof of the lemma.
\end{proof}

\begin{lemma}\label{lem:localization}
    	Let $r\ge 1, 0<\alpha<1$. Let $B$ be a planar Brownian motion started from $1$ stopped upon reaching $\Cc_r$. Let $\sigma_{0,r}$ be the last time $B$ visits $\Cc_0$.
    	There exists a constant $C=C(\alpha)$ such that for all $r\ge 1, z\in\Cc_0$ with $|z-1|\ge 2\alpha$,
    	\begin{equation}\label{eq:even-br}
    	\Pb\Big(B[0,\sigma_{0,r}]\subseteq \Bc(0,1)\cup \Bc(1,\alpha)\cup \Bc(z,\alpha)\mid B(\sigma_{0,r})=z\Big) \ge C r^{-1}.
    	\end{equation}
\end{lemma}
    \begin{proof}
    	By \eqref{eq:cled}, it follows that
    	    \begin{equation} 
 \Pb\Big(B[0,\sigma_{0,r}]\subseteq \Bc(0,1)\cup \Bc(1,\alpha)\cup \Bc(z,\alpha)\mid B(\sigma_{0,r})=z\Big) =\frac{G_{\Bc(0,1)\cup \Bc(1,\alpha)\cup \Bc(z,\alpha)}(1,z)}{G_{\Bc_r}(1,z)}
    		\label{eq:frac1}
\end{equation}   	
        By \eqref{eq:Green}, the numerator in the right-hand side of \eqref{eq:frac1} is greater than $C_1(\alpha)$ and the denominator is less than $C_2(\alpha)r$ for all $z\in\Cc_0$ with $|z-1|\ge 2\alpha$. This finishes the proof.
    \end{proof}

    \begin{lemma}\label{lem:outer-bridge}
	Let $r, r_0\ge 1, 0<\alpha<1$ and $\Bc_{r+1}\subset D\subset \Bc_{r+r_0}$. Let $B$ be a planar Brownian motion started at $e^r$ stopped upon reaching $\Cc_0$. 
	Let $\sigma_{r,0}$ be the last time $B$ visits $\Cc_r$ before hitting $\Cc_0$.
	Let $\tau_0$ be the first time $B$ hits $\Cc_0$ and $\tau_{\partial D}$ be the first time $B$ hits $\partial D$.
	There exists a constant $C=C(\alpha,r_0)$ such that for all $r\ge 1, z\in\Cc_r$ with $|z-e^r|\ge 2\alpha e^r$,
	\begin{equation}\label{eq:odd-br}
		\Pb(B[0,\sigma_{r,0}]\subseteq \Bc_r^c\cup \Bc(e^r,\alpha e^r)\cup  \Bc(z,\alpha e^r)\mid \tau_0<\tau_{\partial D}, B(\sigma_{r,0})=z) \ge C.
	\end{equation}
\end{lemma}
    \begin{proof}
    	Similar to the proof of Lemma~\ref{lem:localization},
    	\begin{align} 
    	\notag
    	&\;\Pb(B[0,\sigma_{r,0}]\subseteq \Bc_r^c\cup \Bc(e^r,\alpha e^r)\cup  \Bc(z,\alpha e^r)\mid \tau_0<\tau_{\partial D}, B(\sigma_{r,0})=z) \\ 
    	\notag
    	=&\; \frac{G_{(D\setminus \Bc_r)\cup\Bc(e^r,\alpha e^r)\cup  \Bc(z,\alpha e^r)}(e^r,z)}{G_{D\setminus \Bc_0}(e^r,z)}\\ 
    	\label{eq:frac}
    	\ge&\; \frac{G_{\Ac(r,r+1)\cup\Bc(e^r,\alpha e^r)\cup  \Bc(z,\alpha e^r)}(e^r,z)}{G_{\Bc_{r+r_0}}(e^r,z)}.
    \end{align}
	By scaling invariance of the Green's function, the numerator in \eqref{eq:frac} is greater than $C_1(\alpha)$ and the denominator is less than $C_2(\alpha,r_0)$ for all $z\in\Cc_r$ with $|z-e^r|\ge 2\alpha e^r$. This finishes the proof.
\end{proof}

    \begin{lemma}\label{lem:BM-prop-1}
For $r>1$, let $B$ and $\sigma_{0,r}$ be defined as in Lemma~\ref{lem:localization}. Let $D_r$ be the event that $B[0,\sigma_{0,r}]$ does not disconnect $\Cc_0$ from infinity. Then, $\Pb(D_r)\lesssim r^{-1}$.
    \end{lemma}
    \begin{proof}
    	Let $N_r$ denote the number of crossings of $B$ from $\Cc_0$ to $\Cc_{1}$. Then for all $k\in\Nb_0$,
    	\[
    	\Pb(N_r= k+1)= \left(\frac{r-1}{r}\right)^k\frac{1}{r}.
    	\]
	During each crossing from $\Cc_0$ to $\Cc_{1}$, with positive probability $p$ (not depending on $r$), $B$ would make a closed loop which disconnects $\Cc_0$ from infinity. Therefore,
    	\[
    	\Pb(D_r)\le\sum_{k\ge 0} \left(\frac{r-1}{r}\right)^k\frac{1}{r} (1-p)^{k+1}\lesssim \frac1r.
    	\]
This finishes the proof.
    \end{proof}

   Combining Lemma~\ref{lem:localization} and Lemma~\ref{lem:BM-prop-1}, we know that $\Pb(D_r)\asymp r^{-1}$. We also have the following conditioned version of Lemma~\ref{lem:BM-prop-1}.
   
\begin{lemma}\label{lem:D1r}
  	Let $r\ge 2$ and $B$ be a planar Brownian motion started at $1$ stopped upon reaching $\Cc_r$. Denote by $\sigma_{1,r}$ the last time $B$ visits $\Cc_1$ before reaching $\Cc_r$. For all $z\in \Cc_1$, let $W^z$ be the Brownian link (see \eqref{eq:cled} and below) from $1$ to $z$ in $\Bc_r$, i.e., $B[0,\sigma_{1,r}]$ conditioned on $B(\sigma_{1,r})=z$. Let $D^z_{1,r}$ be the event that $W^z$ does not disconnect $\Cc_0$ from infinity. Then, there exists a universal constant $C$ such that for all $z\in \Cc_1$ and $r>1$, we have
  	\[
  	\Pb(D^z_{1,r})\le C r^{-1}.
  	\]
  \end{lemma}

\begin{proof}
	Suppose $z\in \Cc_1$. Let $\sigma_{1/2,r}$ (resp.\  $\sigma'_{1/2,r}$) be the last time $B$ (resp.\ $W^z$) visits $\Cc_{1/2}$ before reaching $\Cc_r$. Then, the law of $B[0,\sigma_{1/2,r}]$ and the law of $W^z[0,\sigma'_{1/2,r}]$ are absolutely continuous with densities bounded from above and below by constants which do not depend on $z$ and $r$. Let $E$ (resp.\ $E'$) be the event that $B[0,\sigma_{1/2,r}]$ (resp.\ $W^z[0,\sigma_{1/2,r}]$) does not disconnect $\Cc_0$ from infinity. Then we know that $\Pb(E)\asymp \Pb(E')$. By an argument similar to Lemma~\ref{lem:BM-prop-1}, $\Pb(E)\lesssim r^{-1}$. Therefore, 
	$\Pb(D^z_{1,r})\le \Pb(E')\lesssim \Pb(E)\lesssim r^{-1}$, which finishes the proof.
\end{proof}

	The next lemma relates the law of the last exit point (on the next scale) to a uniform distribution which also decouples the starting point and the last exit point.
	\begin{lemma}\label{lem:last-exit-uniform}
		Let $r\ge 2$. Let $B_0$ be a random point on $\Cc_0$ and $B$ be a Brownian motion started from $B_0$. Let $X$ denote a uniformly distributed random variable on $\Cc_1$. There exists a universal constant $C>1$ such that for any bounded measurable positive functions $f: \Cc_1\rightarrow \Rb$ and $g: \Cc_0\rightarrow \Rb$, 
		\[
		C^{-1} \Eb(f(X))\Eb(g(B_0))\le \Eb(f(B(\sigma_{1,r})) g(B_0)) \le C \Eb(f(X))\Eb(g(B_0)).
		\] 
	\end{lemma}
	In other words, the law of $B(\sigma_{1,r})$ and the uniform distribution are absolutely continuous w.r.t. each other with densities uniformly bounded away from $0$ and $\infty$. 
\begin{proof}
	By using the last exit decomposition \eqref{eq:led}, it suffices to show that for all $x\in \Cc_0$, $z\in \Cc_1$,
	\[
	a_1\le p(x,z)=\int_{\Cc_r} G_{\Bc_r}(x,z)H_{\Ac(1,r)}(z,y) dy \le a_2,
	\]
	where $a_1,a_2$ are some positive constants that do not depend on $x$ and $z$, and $H_{\Ac(1,r)}(x,y)$ is the boundary Poisson kernel. 
	By \eqref{eq:Green}, $G_{\Bc_r}(x,z)\asymp G_{\Bc_r}(0,e)$. By rotation invariance, $\int_{\Cc_r} H_{\Ac(1,r)}(z,y) dy$ takes the same value for all $z\in\Cc_1$. The lemma follows immediately.
\end{proof}
We remark that, from the above proof it can also be deduced that the analogous result is not true if we consider the last exit point on the unit circle instead of the next scale $\Cc_1$ (w.r.t. the starting scale $\Cc_0$). That is why we need to spare an additional scale when we prove Theorem~\ref{thm:up-to-constants-ex}. However, if we require the Brownian motion ever entered deep inside the unit ball, then the last exit on the same scale has the same property by strong Markov property. This is illustrated in the following lemma.

\begin{lemma}\label{lem:leu-SameScale}
	Let $r\ge 1$. Let $B_0$ and $B$ be defined as in Lemma~\ref{lem:last-exit-uniform}. Let $Y$ denote a uniformly distributed random variable on $\Cc_0$. There exists a universal constant $C>1$ such that for any bounded measurable positive functions $f,g : \Cc_0\rightarrow \Rb$,
	\begin{equation*}
	C^{-1} \Eb(f(X))\Eb(g(B_0))\le \Eb(f(B(\sigma_{0,r})) g(B_0)\mid \tau_{-1}<\sigma_{0,r}) \le C \Eb(f(X))\Eb(g(B_0)).
	\end{equation*}
\end{lemma}

In fact, by inverting the Brownian motion, the above two lemmas can be formulated for Brownian motions from outside to inside. For example, if we are in the setup of Lemma~\ref{lem:outer-bridge}, then the law of $B(\sigma_{r,0})$ under the condition $\tau_{r+1}<\sigma_{r,0}$ has a density w.r.t.\ the uniform distribution on $\Cc_r$ which is bounded away from $0$ and $\infty$.

	\subsection{Brownian loop soup}\label{subsec:BLS}
	We review the Brownian loop soup in this section. 
	The \textit{Brownian loop measure} $\mu$ is defined in \cite{MR2045953} by
	\begin{equation}\label{eq:def-mu}
	\mu=\int_{\mathbb{C}} \int_{0}^{\infty} \frac{1}{2 \pi t^{2}} \mu^{\#}(z, z ; t) \,d t\, d A(z),
	\end{equation}
	where $\mu^{\#}(z, z ; t)$ is the usual probability measure of a Brownian bridge from $z$ to $z$ with time length $t$, and $dA$ denotes the Lebesgue measure on $\Cb$. In reality, $\mu$ should be viewed as a measure on unrooted loops where one forgets the root $z$ (see \cite{MR2045953} for more details).
	For a domain $D$, let $\mu_D$ be $\mu$ restricted to the loops in $D$.
	A \textit{Brownian loop soup} in $D$ with intensity $c\ge 0$ is a Poisson point process with
	intensity $c\mu_D$. 
	We denote by $\Gamma$ a Brownian loop soup with intensity $c$ in the whole plane. For each domain $D$, we denote by $\Gamma_{D}$ the collection of loops in $\Gamma$ that are in $D$.  Note that $\Gamma_D$ is also a Brownian loop soup with intensity $c$ in $D$. We also set $\Gamma_0:=\Gamma_{\Ub}$. 
The law of the Brownian loop soup is conformally invariant, namely if $f: D \rightarrow D^{\prime}$ is a conformal transformation, then the image of $\Gamma_{D}$ under $f$ is equal to a Brownian loop soup in $D'$.

It was shown in  \cite{MR2979861} that the Brownian loop soup exhibits a phase transition at the critical intensity $c=1$, so that for $c\in(0,1]$ a Brownian loop soup in $\Ub$ a.s.\ has infinitely many clusters. Throughout, we suppose that $c\in(0,1]$ is fixed and often omit it from the notation (except when we say otherwise).

Let us first record an FKG-Harris inequality for the Brownian loop soup, shown in \cite{MR770197}.
	\begin{lemma}[FKG-Harris inequality, \cite{MR770197}]\label{lem:FKG}
		An event $A$ is said to be \emph{decreasing} if for any $\Gamma\notin A$ and $\Gamma\subseteq\Gamma'$ we have $\Gamma'\notin A$, where $\Gamma, \Gamma'$ are realizations of the Brownian loop soup. Then, for any two decreasing events $A$ and $B$, we have
		\begin{align*}%\label{eq:FKG}
		\Pb(A\cap B)\ge \Pb(A) \Pb(B).
		\end{align*}
	\end{lemma}

Let us now obtain a simple estimate for the size of clusters in a sub-critical loop soup.

\begin{lemma}\label{lem:cluster_small}
Let $D$ be a bounded domain and $\Gamma_D$ be a Brownian loop soup with intensity $c\in(0,1]$ in $D$. For all $\eps>0$, the probability that every cluster in $\Gamma_D$ has diameter at most $\eps$ is positive.
\end{lemma}
\begin{proof}
It suffices to prove the lemma for $D=\Ub$, because one can always find $\delta>0$ such that $\delta D\subset\Ub$.
Arguments in \cite[Lemma 9.6]{MR2979861} show that at subcritical intensities $c\in (0,1]$, for any $z_0\in\Ub$ and $r_1<r_2$ such that the annulus $\{z: r_1< |z-z_0|<r_2\}$ is contained in $\Ub$, there is a positive probability that no cluster in $\Gamma_0$ crosses this annulus. We can find a finite set of points $\{z_i\}_{i\in I}$ such that $\Ub$ can be covered by the annuli $\{z: \eps/2< |z-z_i|<\eps\}$ and that any cluster with diameter at least $\eps$ must cross one of these annuli. For each $i\in I$, the event $E_i$ that there does not exist any cluster crossing the annulus $\{z: \eps/2< |z-z_i|<\eps\}$ is decreasing.
Applying Lemma~\ref{lem:FKG} for the events $E_i$, we obtain that with positive probability, there does not exist any cluster which crosses any of the annuli, and consequently there does not exist any cluster with diameter more than $\eps$.
\end{proof}

It turns out that in many occasions, it will be more convenient to work with the soup of the outer boundaries of the Brownian loops (which are simple loops \cite{MR2350053}) rather than the Brownian loop soup itself. This soup of simple loops has the same property as the original Brownian loop soup in terms of disconnection from infinity, and was used in \cite{MR2802511} to compute the dimension of the CLE carpet.
More precisely, let $\Lambda$ be the collection of outer boundaries of the loops in $\Gamma$. 
If we let $\nu$ be the image of $\mu$ under the map which maps a Brownian loop to its outer boundary, then $\Lambda$ is a Poisson point process with intensity $c\nu$.
For all $r\in\Rb$, denote by $\Lambda_r$ all the loops in $\Lambda$ that are contained in the ball $\Bc_r$. For all $s<r$, let $\Lambda_{s,r}$ be the collection of the loops in $\Lambda$ that are contained in the annulus $\mathcal A(s,r)$.

It was shown in \cite{MR2802511} that the loop soup $\Lambda$ is \emph{thin}, which is a property described in the following lemma. 
\begin{lemma}[Lemmas 2 and 4 \cite{MR2802511}, Thinness]\label{lem:thin}
Let $L_R$ be the set of loops that intersect the unit disk and have diameter at least $R$. For all $R>0$, we have $\nu(L_R)<\infty$.
\end{lemma}
Note that the Brownian loop soup $\Gamma$ is not thin, and we have $\mu(L_R)=\infty$ for all $R>0$. This is the reason why it is more convenient to work with $\Lambda$ rather than $\Gamma$. The thinness of $\Lambda$ allows us to collect the following estimate for the clusters in $\Lambda$.
		\begin{lemma}\label{lem:cluster-thin}
			For $\delta>0$, let $E_{\delta}$ be the event that all the clusters in $\Lambda$ that intersect the unit disk have diameter at most $\delta$. Then, there exists a constant $C=C(\delta)>0$ such that $\Pb(E_{\delta})\ge C$. 
		\end{lemma}
		\begin{proof}
The thinness (Lemma~\ref{lem:thin}) of $\Lambda$ implies that with positive probability $p_1>0$, every loop in $\Lambda$ intersecting $\Cc_{1+\delta}$ has diameter at most $\delta/10$. 
Moreover, Lemma~\ref{lem:cluster_small} implies that with positive probability $p_2>0$, all the clusters in $\Lambda_{1+\delta}$ have diameter at most $\delta/10$. Note that $E_{\delta}$ holds on the intersection of the two previous events which are both decreasing. By Lemma~\ref{lem:FKG}, we know that $\Pb(E_{\delta})\ge p_1p_2$.
		\end{proof}

\subsection{Generalized disconnection exponents}\label{subsec:GDE}

A family of exponents $\eta_\kappa(\alpha, \beta)$ were defined in \cite{MR4221655} using the so-called \emph{general radial restriction measure}
$\Pf^{\alpha, \beta}_\kappa$ with parameters $\kappa$ and $(\alpha, \beta)$.
 Here, we omit the precise definition of $\Pf^{\alpha, \beta}_\kappa$ (we refer the reader to \cite{MR4221655}), but only mention three general facts:
 \begin{itemize}
\item $\Pf^{\alpha, \beta}_\kappa$ is a measure on the space $\Omega$ of simply connected compact sets $K\subset\ol\Ub$ such that $0\in K$ and $K\cap \partial\Ub=\{1\}$. 
\item $\Pf^{\alpha, \beta}_\kappa$ can be explicitly constructed using radial hypergeometric SLE$_\kappa$ with appropriate parameters. 
\item When $\kappa=8/3$, $\Pf^{\alpha, \beta}_{8/3}$ coincide with the standard radial restriction measure with parameters $(\alpha, \beta)$ constructed in \cite{MR3293294}.
\end{itemize}
Recall that for any simply connected domain $D\subset \Ub$ such that $0\in D$, the \emph{conformal radius} of $D$ seen from the origin is defined to be $|f'(0)|^{-1}$ where $f$ is any conformal map from $D$ onto $\Ub$ that leaves the origin fixed. The following theorem is part of \cite[Theorem 1.5]{MR4221655}, and gives a definition of the exponents $\eta_\kappa(\alpha, \beta)$.

\begin{theorem}[\cite{MR4221655} Theorem 1.5]\label{thm:dis_def}
Let $\kappa\in(0,4]$. Let $\alpha \le \eta_\kappa(\beta)$ and $\beta\ge (6-\kappa)/(2\kappa)$ where
\begin{align*}
\eta_\kappa(\beta)= \frac{\left(\sqrt{16\beta\kappa +(4-\kappa)^2} -(4-\kappa)\right)^2 -4(4-\kappa)^2}{32\kappa}.
\end{align*}
Let $K$ be a sample with law $\Pf^{\alpha,\beta}_\kappa$. Let $K_0$ be the connected component containing $0$ of the interior of $K$.
Let $p^R_\kappa(\alpha, \beta)$ be the probability that  the conformal radius of $K_0$ seen from the origin is smaller than $1/R$. 
There exists $\eta_\kappa(\alpha,\beta)>0$ and $C_\kappa(\alpha, \beta)>0$ such that as $R\to\infty$,
\begin{align}\label{p-R-3}
p^R_\kappa(\alpha,\beta)=R^{-\eta_\kappa(\alpha,\beta)} (C_\kappa(\alpha, \beta)+o(1)).
\end{align}
\end{theorem}
The values of the exponents $\eta_\kappa(\alpha, \beta)$ were also computed in \cite[Theorem 1.5]{MR4221655}, and it turned out that $\eta_\kappa(\beta)=\eta_\kappa(0,\beta)$. The  exponents $\eta_\kappa(\beta)$ were then defined in \cite{MR4221655} as the \emph{generalized disconnection exponents}. If we let $\xi_c(\beta):=\eta_{\kappa}(\beta)$ for $c$ and $\kappa$ related by~\eqref{eq:c_kappa}, then the value of $\xi_c(\beta)$ is given by~\eqref{eq:exponent}.

Recently, it was shown in \cite{cg25} that the general radial restriction measure is unique, and hence its law is determined by $\Pf_\kappa^{\alpha, \beta}$. Then, by \cite[Proposition 2.9]{cg25}, we can relate the exponents $\xi_\kappa(\beta)$ to the Brownian loop soup in the following way.
\begin{lemma}[\cite{cg25} Proposition 2.9]\label{lem:obs}
Let $K$ be distributed according to $\Pf^{\alpha, \beta}_{8/3}$. Let $\Cc(K)$ be the union of $K$ with all the clusters that it intersects in an independent loop soup of intensity $c$ in $\Ub$.
Then the filling of $\Cc(K)$ is distributed according to $\Pf_\kappa^{\alpha, \beta}$, where $\kappa$ and $c$ are related by~\eqref{eq:c_kappa}.
\end{lemma}
\begin{figure}
\centering
\includegraphics[scale=.56]{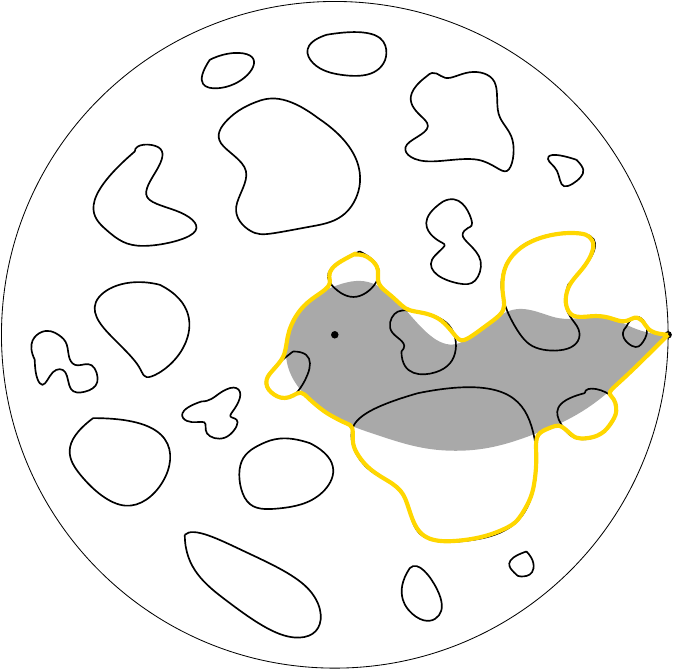}\qquad \qquad
\includegraphics[scale=.56, page=3]{def2}
\caption{\textbf{Left:} Illustration of Lemma~\ref{lem:obs}. The black loops represent the outermost boundaries of the outermost clusters of a loop soup with intensity $c\in(0,1]$ in $\Ub$. The grey region is a sample $K$ with law $\Pf^{\alpha, \beta}_{8/3}$. The yellow line is the outer boundary of the union of $K$ with all the clusters that it intersects. The region encircled by the yellow line has law $\Pf^{\alpha, \beta}_{\kappa(c)}$. \textbf{Right:} A similar picture as the left one where $K$ is replaced by $\beta\in\Nb$ independent Brownian excursions in $\Ub$ between $0$ and $1$. The region encircled by the yellow line has law $\Pf^{0, \beta}_{\kappa(c)}$.} 
\label{fig:def2}
\end{figure}

It was pointed out in \cite{MR1992830} that Brownian excursions satisfy conformal restriction, and in particular the filling of a Brownian excursion from $0$ to $1$ in $\Ub$ satisfies radial restriction with parameters $(0,1)$. If we let $K_\beta$ be the union of $\beta\in \Nb$ independent Brownian excursions in $\Ub$ which start at $0$ and exit $\Ub$ at $1$, then the filling of $K_\beta$ is distributed as $\Pf^{0, \beta}_{8/3}$. 
Let $\Cc(K_\beta)$ be the union of $K_\beta$ with all the clusters that it intersects in an independent loop soup of intensity $c$ in $\Ub$.
Then Lemma~\ref{lem:obs} implies that the filling of $\Cc(K_\beta)$ is distributed according to $\Pf_\kappa^{0, \beta}$, where $\kappa$ and $c$ are related by~\eqref{eq:c_kappa} (see Figure~\ref{fig:def2} right).
\smallskip

Recall that for $k\in\Nb$, the definition of $\xi_c(k,0)$ given by~\eqref{eq:wh-p1} involves Brownian motions and a loop soup in $\Bc(0,r)$ for $r>1$. If we rescale everything by $1/r$, then the definition of $\xi_c(k,0)$ given by~\eqref{eq:wh-p1} is very similar to the definition of $\xi_c(k)$ given just above, with two differences:
\begin{itemize}
\item In the definition of $\xi_c(k,0)$ (where we rescale everything by $1/r$), we consider $k$ Brownian motions started at $k$ uniform points on $\Cc_{1/r}$ and stopped upon exiting $\Ub$; while in the definition of $\xi_c(k)$, we consider $k$ Brownian excursions in $\Ub$ between $0$ and $1$.
\item In the definition of $\xi_c(k,0)$ we consider an event of non-disconnection; while in the definition of $\xi_c(k)$, we consider an event that the conformal radius of the random set generated by the Brownian motions and the loop soup is small.
\end{itemize}
We will show in Proposition~\ref{prop:same-exponents} that $\xi_c(k,0)=\xi_c(k)$.

	\section{Up-to-constants estimates}\label{sec:up-to-constant}
In this section, we prove up-to-constants estimates for the generalized non-intersection probability (Theorem~\ref{thm:xi-k}) and non-disconnection probability (Theorem~\ref{thm:same_exponent}), assuming a crucial ``separation lemma'' (Lemma \ref{lem:separation_lem_2}) whose proof we postpone till  Section~\ref{sec:separation-lemma}. 
Among other things, these results allow us to define the generalized intersection exponents, and relate them to the generalized disconnection exponents defined in \cite{MR4221655} with values \eqref{eq:exponent} (Corollary \ref{cor:xi(k,0)} and Proposition~\ref{prop:same-exponents}).

In Section~\ref{subsec:Statement of the results}, we state the main results. In Section~\ref{subsec:k=2-ee}, we state and prove Theorem \ref{thm:up-to-constants-ex}, which is a variant of  Theorem~\ref{thm:xi-k} formulated in terms of extremal distances. The proof of Theorem \ref{thm:up-to-constants-ex} relies on a separation lemma (Lemma \ref{lem:separation_lem_2}).
In Section~\ref{subsec:k=2}, we derive Theorem \ref{thm:xi-k} in the $k=2$ case as a corollary of Theorem \ref{thm:up-to-constants-ex}. Finally, in Section~\ref{subsec:other-exponents}, we state in Theorem \ref{thm:up-to-constants-ex-k-sep} the corresponding result for general  $k$'s, and then prove Proposition~\ref{prop:same-exponents}.

\subsection{Statement of the results}\label{subsec:Statement of the results}
Recall from Section~\ref{subsec:BLS} that $\Gamma$ is a Brownian loop soup with intensity $c$ in the plane and $\Lambda$ is the collection of outer boundaries of the Brownian loops in $\Gamma$.
	The definition of $\xi_c(k,0)$ given in~\eqref{eq:wh-p1}, as well as the generalized intersection exponents $\xi_c(k, \lambda)$ that we will define in this subsection, remains the same whether we use $\Gamma$ or $\Lambda$. However, later it will be crucial for us to work with $\Lambda$ rather than $\Gamma$, when we switch to the setup of extremal distances, because $\Lambda$ is thin (see the end of Section~\ref{subsec:BLS}). Therefore, we will define the exponents using $\Lambda$ from now on.

	For all $r>0$, let $\Lambda_r$ be the collection of loops in $\Lambda$ which are contained in the ball $\Bc_r$ of radius $e^r$ centered at $0$.
	Let $\Cc_r=\partial \Bc_r$ be the circle of radius $e^r$ centered at $0$.
Let $B^0$, $B^1$, $\cdots$, $B^k$ be $k+1$ independent Brownian motions started from $k+1$ uniformly chosen points on $\Cc_0$. For all $r$, define the hitting times of Brownian motion as
	\[\tau^i_r:=\inf\{ t>0: B^i_t\in \Cc_r\}.\]
	Denote by $\overline B^i_r$ the path $B^i[0,\tau^i_r]$. 
	Let $\Theta_r$ be the union of $\overline B^1_r\cup\cdots\cup \overline B^k_r$ together with the clusters of loops in $\Lambda_r\setminus\Lambda_0$ that it intersects. 
	Let $\Pi_r$ be the $\sigma$-algebra generated by $ \overline B^1_r, \overline B^2_r, \Lambda_r\setminus\Lambda_0$.
	Define 
	\begin{equation}\label{eq:Zr}
	Z_r=Z^k_r:=\Pb( \overline B^0_r\cap \Theta_r=\emptyset \mid \Pi_r).
	\end{equation}
	We define \textit{the generalized non-intersection probabilities} for all $c\in(0,1], k\in\Nb, r\ge 0, \lambda>0$, 
	\begin{equation}\label{eq:Z-r}
	p(c,k,r,\lambda):=\Eb((Z_r)^{\lambda}).
	\end{equation}
Since the Brownian motion and the Brownian loop soup are scale-invariant, by the standard sub-multiplicativity argument, there exists an exponent $\xi_c(k,\lambda)>0$ such that
	\begin{equation}\label{eq:g-intersection}
	p(c,k,r,\lambda)\approx e^{-r\xi_c(k,\lambda)}.
	\end{equation}
	We call $\xi_c(k,\lambda)$ the \textit{generalized intersection exponent}. 
The following results show that the asymptotic estimate \eqref{eq:g-intersection} can be improved to an up-to-constants one. 
	
	\begin{theorem}\label{thm:xi-k}
		Let $c\in (0,1]$, $k\in\Nb$ and $\lambda_0>0$. For all $\lambda\in(0,\lambda_0]$
		\begin{equation}\label{eq:xi-k}
		p(c,k,r,\lambda)\asymp e^{-r\xi_c(k,\lambda)},
		\end{equation}
		where the implied constants only depend on $c,k,\lambda_0$.
	\end{theorem}
	
	Note that $p(c,k,r,\lambda)$ is decreasing in all four parameters.
	If we let $\lambda=0$ and make the convention $0^0=0$, then \eqref{eq:Z-r} can be extended to 
	\begin{equation}\label{eq:p(c,k,r,0)}
	p(c,k,r,0)=\Pb(Z_r>0)=\lim_{\lambda\rightarrow 0+}\Eb((Z_r)^{\lambda})=\lim_{\lambda\rightarrow 0+}p(c,k,r,\lambda),
	\end{equation}
	which we call the \textit{generalized non-disconnection probability}.
	Note that $\{Z_r>0\}$ is equivalent to the event that $\Theta_r$ does not disconnect $\Cc_0$ from $\Cc_r$. 
	Therefore, the generalized non-disconnection probability \eqref{eq:p(c,k,r,0)} defined here is exactly the same as that introduced before \eqref{eq:wh-p1}.
Theorem~\ref{thm:xi-k} implies the following corollary, which gives the first part of Theorem \ref{thm:same_exponent}.
	\begin{corollary}\label{cor:xi(k,0)}
		Let $c\in (0,1]$, $k\in\Nb$. The generalized intersection exponent $\xi_c(k,\lambda)$ is continuous in $\lambda\ge 0$. In particular, $\xi_c(k,0):=\lim_{\lambda\rightarrow 0+}\xi_c(k,\lambda)$ exists. Furthermore, 
		\begin{equation}
		p(c,k,r,0) \asymp e^{-r\xi_c(k,0)},
		\end{equation}
		where the implied constants only depend on $c,k$.
	\end{corollary}

At the end of this section, we will show that $\xi_c(k,0)=\xi_c(k)$ in Proposition~\ref{prop:same-exponents}, where $\xi_c(k)$ (see \eqref{eq:exponent}) is a reformulation of the generalized disconnection exponent $\eta_{\kappa}(k)$ introduced in \cite{MR4221655} (see Theorem~\ref{thm:dis_def} and the paragraph below).
Combined with Corollary \ref{cor:xi(k,0)}, this gives Theorem \ref{thm:same_exponent}.

\begin{figure}
	\centering
	\includegraphics[scale=0.35]{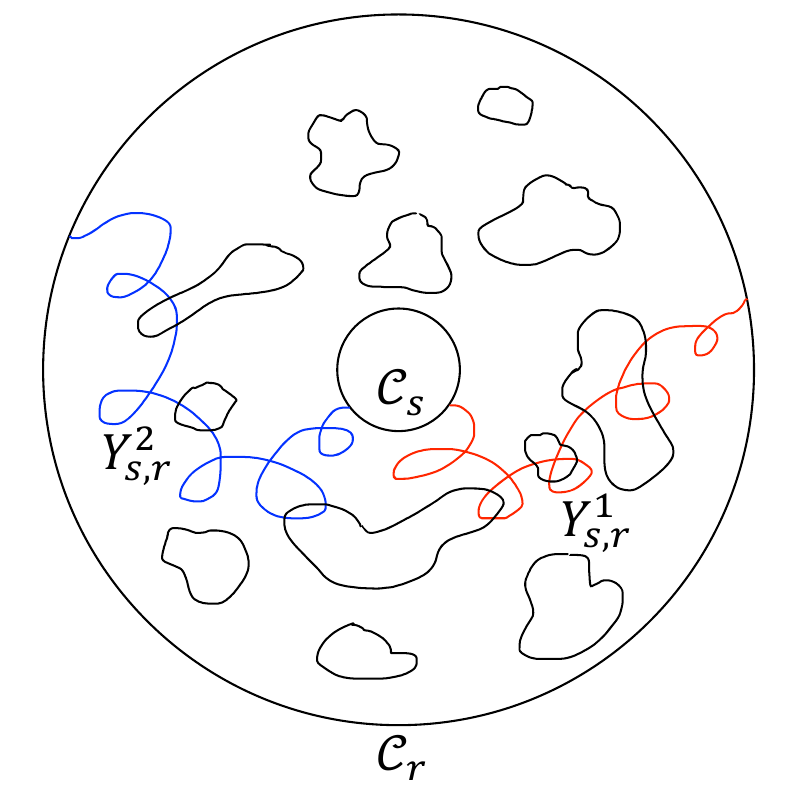}
	\includegraphics[scale=0.27]{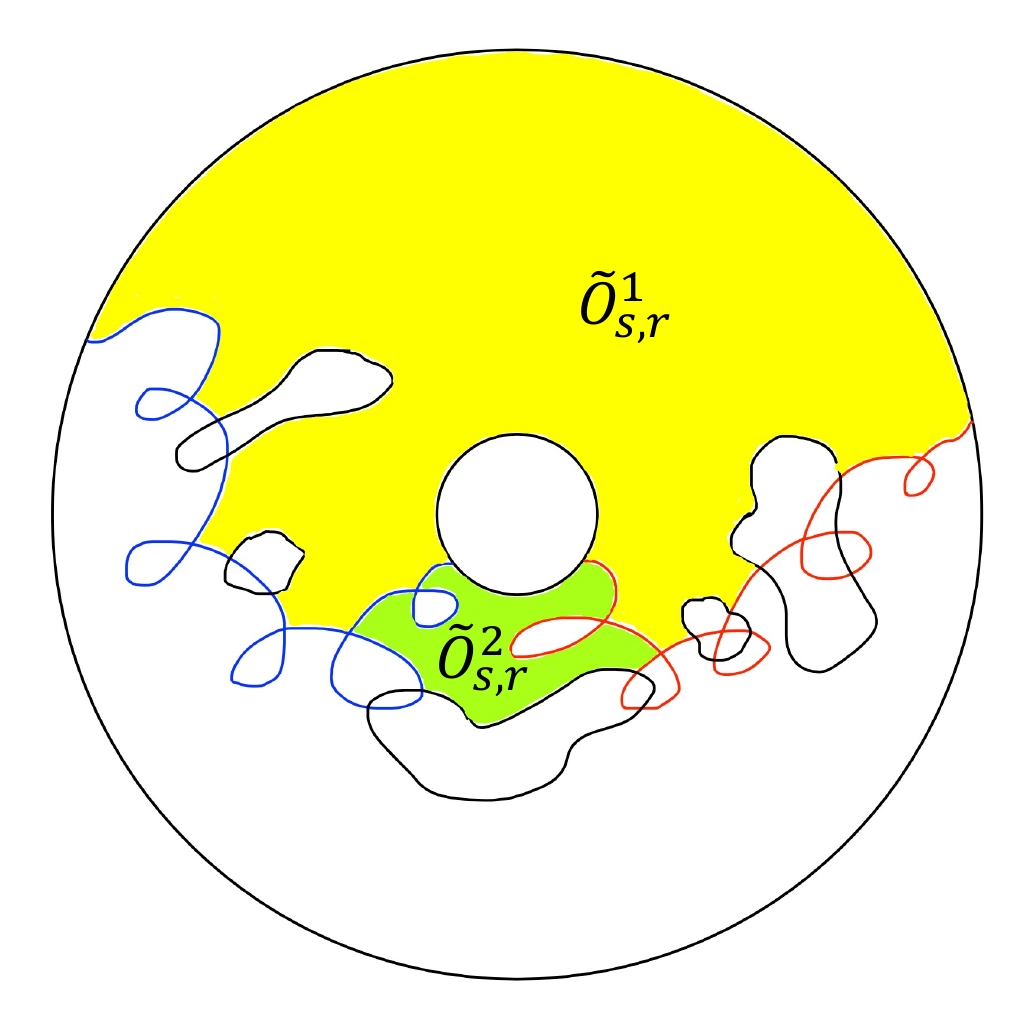}
	\caption{\textbf{Left:} The concentric circles are $\Cc_s$ and $\Cc_r$ respectively. The black loops represent the outermost boundaries of the outermost clusters in $\Lambda_{s,r}$. The red and blue curves are excursions $Y^1_{s,r}$ and $Y^2_{s,r}$ respectively. \textbf{Right:} The yellow and green regions are two connected components $\wt O^1_{s,r}$ and $\wt O^2_{s,r}$ respectively.} 
	\label{fig:br}
\end{figure}

	\subsection{Setup using extremal distances and excursions}\label{subsec:k=2-ee}
	In this subsection, we present a version of Theorem~\ref{thm:xi-k} formulated in terms of extremal distances and excursions.
	We keep the notation introduced in Section~\ref{subsec:Statement of the results}.
	For $i=1,2$ and $0\le s<r$, let $\sigma^i_{s,r}$ be the last time that $B^i$ hits $\Cc_s$ before $\tau^i_r$, i.e., 
	\[
	\sigma^i_{s,r}:=\sup\{ t<\tau^i_r: B^i(t)\in \Cc_s \}.
	\]
	Then for $i=1,2$,
	\begin{equation}\label{eq:Y}
	Y^i_{s,r}(t):=B^i(t+\sigma^i_{s,r}), \quad 0\le t\le \tau^i_r-\sigma^i_{s,r},
	\end{equation}
	are two independent Brownian excursions across the annulus $\mathcal A(s,r)=\Bc_r\setminus \ol\Bc_s$. 
	Define the hitting time of $Y^i_{s,r}$ on $\Cc_u$ as 
	\[
	T^i_u:=\inf\{ t>0: Y^i_{s,r}(t)\in \Cc_u\}.
	\]
	Recall that $\Lambda_{s,r}$ is the collection of the loops in $\Lambda$ that are contained in $\mathcal A(s,r)$.
	We use $Y^i_{s,r}$ to denote the trace $Y^i_{s,r}[0,T_r^i]=B^i[\sigma^i_{s,r},\tau^i_r]$,
and use $\wt Y^i_{s,r}$ to denote the union of $Y^i_{s,r}$ and the collection of clusters in $\Lambda_{s,r}$ that intersect $Y^i_{s,r}$.  Let $\Ic^1_s$ be the counterclockwise arc on $\Cc_s$ from $Y^1_{s,r}(0)$ to $Y^2_{s,r}(0)$. Almost surely, $\mathcal A(s,r)\setminus (\wt Y^1_{s,r}\cup \wt Y^2_{s,r})$ contains two connected components $\wt O^1_{s,r}$ and $\wt O^2_{s,r}$ that are at zero distance from $\Cc_s$, where $\wt O^1_{s,r}$ is the one such that $\partial\wt O^1_{s,r}\cap\Cc_s=\Ic^1_s$. See Figure~\ref{fig:br} for an illustration of these objects. For $i=1,2$, let 
$$
\mbox{$\wt L^i_{s,r}:=L(\wt O^i_{s,r})$ be the $\pi$-extremal distance between $\Cc_s\cap \partial\wt O^i_{s,r}$ and $\Cc_r\cap \partial\wt O^i_{s,r}$ in $\wt O^i_{s,r}$.}
$$
When $\partial\wt O^i_{s,r}\cap \Cc_r=\emptyset$, set $\wt L^i_{s,r}:=\infty$. Let $\wt L_{s,r}=\wt L^1_{s,r} \wedge \wt L^2_{s,r}$. Define 
	\begin{equation}
	\wt b_{s,r}=\wt b_{s,r}(\lambda)=(r-s)^{-2}\mathbb{E}\Big(e^{-\lambda \wt L_{s,r}}\Big) \quad \text{ for all } 0\le s<r.
	\end{equation}
	When $s=0$, we drop the subscript $s$ and write
	\begin{equation}\label{eq:br}
	\wt b_{r}:=\wt b_{0,r}=r^{-2}\mathbb{E}\Big(e^{-\lambda \wt L_{r}}\Big).
	\end{equation}
	By the conformal invariance of Brownian motions, Brownian loop soups and extremal distances, we know that $\wt L_{s,r}$ (resp. $\wt b_{s,r}$) has the same distribution as $\wt L_{r-s}$ (resp. $\wt b_{r-s}$) for all $0\le s<r$. Note that $\wt L_{r}$ is stochastically increasing in $r$ and therefore $\wt b_{r}$ is decreasing in $r$.

We will focus on the proof of the following theorem and later show that it implies Theorem~\ref{thm:xi-k} (see Corollary~\ref{cor:xi-2}).
In particular, we will show that the exponent $\wh\xi_c$ in the theorem below is the same as the exponent $\xi_c$ in Theorem~\ref{thm:xi-k}.
	\begin{theorem}\label{thm:up-to-constants-ex}
		Let $c\in (0,1]$ and $\lambda_0>0$. For all $\lambda\in(0,\lambda_0]$, there exists $\wh\xi_c(2,\lambda)$ such that
		\begin{equation}\label{eq:up-to-constants-ex}
		\wt b_r\asymp  e^{-r \wh\xi_c(2,\lambda)},
		\end{equation}
		where the implied constants only depend on $c,\lambda_0$.
\end{theorem}

	The proof of the lower bound in Theorem \ref{thm:up-to-constants-ex} is relatively easy, as it suffices to show that $\wt b_{r}$ is sub-multiplicative. 
	\begin{proof}[Proof of the lower bound in Theorem \ref{thm:up-to-constants-ex}]
It suffices to show that there exists a universal constant $C>0$ such that for all $c\in(0,1]$, $\lambda>0$ and $r,s\ge 1$,
		\begin{equation*}
		\wt b_{s+r+1}\le C\wt b_{s} \wt b_{r}.
		\end{equation*}
Indeed, this will imply that the limit $\wh\xi_c(2,\lambda):=-\lim_{r\rightarrow\infty}\frac{\log \wt b_r}{r}$ exists and $\wt b_r\gtrsim e^{-r\wh\xi_c(2,\lambda)}$. 

		Let $i=1,2$. We split each Brownian excursion $Y^i_{s+r+1}$ into two non-intersecting  Brownian excursions $Y^i_{s}$ and $Y^i_{s+1,s+r+1}$ with an intermediate part $W^i=Y^i[\tau^i_{s},\sigma^i_{s+1,s+r+1}]$.
		Denote by $H$ the event that neither $W^1$ nor $W^2$ disconnects $\Cc_s$ from infinity. By Lemma~\ref{lem:D1r} and rotation invariance and scaling invariance of the Brownian motion,
		\begin{equation}\label{eq:H}
		\Eb(\one_H\mid \Gc_{s,r})\lesssim r^{-2},
		\end{equation}
	where $\Gc_{s,r}$ is the $\sigma$-algebra generated by the quadruple
\begin{equation}\label{eq:quad}
    \Big(B^1(\tau^1_{s}), B^2(\tau^2_{s}), B^1(\sigma^1_{s+1,s+r+1}),B^2(\sigma^2_{s+1,s+r+1})\Big).
\end{equation}
 Write $\Ub_s$ for the law of the quadruple $(X_1,X_2,X'_1,X'_2)$ with $X_1$ and $X_2$ (resp. $X'_1$ and $X'_2$) independently uniformly distributed on $\Cc_{s}$ (resp. $\Cc_{s+1}$).
By Lemma~\ref{lem:last-exit-uniform}, the law of the quadruple \eqref{eq:quad} has a density with respect to $\Ub_s$ which is uniformly and universally bounded away from $0$ and $\infty$, hence
    \begin{equation}\label{eq:EtoU}
    \Eb\Big( e^{-\lambda \wt L_{s}}\Big)= 
    \Eb\Big(\Eb\Big( e^{-\lambda \wt L_{s}}\Big| \Gc_{s,r} \Big) \Big) \asymp\Ub_s\Big(\Eb\Big( e^{-\lambda \wt L_{s}}\Big| \Gc_{s,r} \Big) \Big).   
    \end{equation}
 Here we slightly abuse $\Ub_s$ as the law of the quadruple in \eqref{eq:quad}.
 In a similar fashion, 
\begin{equation}\label{eq:EtoU2}
 \Eb\Big( e^{-\lambda \wt L_{s+1,s+r+1}}\Big) \asymp \Ub_s\Big(\Eb\Big( e^{-\lambda \wt L_{s+1,r+1}}\Big| \Gc_{s,r} \Big) \Big),
\end{equation}
and
\begin{equation}\label{eq:EtoU3}
        \Eb\Big( e^{-\lambda \wt L_{s}} e^{-\lambda \wt L_{s+1,s+r+1}} \Big)  \asymp \Ub_s\Big(\Eb\Big(e^{-\lambda \wt L_{s}} e^{-\lambda \wt L_{s+1,s+r+1}}\Big|\Gc_{s,r}\Big)\Big). 
\end{equation}
Viewing Brownian excursions $Y^i_{s+r+1}$ $i=1,2$ as normalized path measures, we see that conditioned on $\Gc_{s,r}$, $Y^i_{s}$ and $Y^i_{s+1,s+r+1}$, and $W^i$, $i=1,2$ are independent Brownian excursions and Brownian links, respectively (this follows from the last-exit decomposition of Brownian path measures, see Section \ref{subsec:facts}). Moreover, we know that loops in $\Ac_{0,s}$ and $\Ac_{s+1,r+1}$ are independent. Hence,
	\begin{equation}\label{eq:independence}
		\text{$\wt L_{s}$, $\wt L_{s+1,s+r+1}$ and $\one_H$ are independent conditioned on $\Gc_{s,r}$.} 
	\end{equation}
Therefore,
\begin{equation}\label{eq:ed-compare}
\begin{split}
 &\Eb\Big( e^{-\lambda \wt L_{s}} e^{-\lambda \wt L_{s+1,s+r+1}} \Big)  \overset{\eqref{eq:EtoU3}}{\asymp} \Ub_s\Big(\Eb\Big(e^{-\lambda \wt L_{s}} e^{-\lambda \wt L_{s+1,s+r+1}}\Big|\Gc_{s,r}\Big)\Big) \\
 \overset{\eqref{eq:independence}}{\asymp} &\Ub_s\Big(\Eb\Big(e^{-\lambda \wt L_{s}}\Big|\Gc_{s,r}\Big) \Eb\Big( e^{-\lambda \wt L_{s+1,s+r+1}}|\Gc_{s,r}\Big)\Big)\\
 \asymp\,\;\;&\Ub_s\Big(\Eb\Big( e^{-\lambda \wt L_{s+1,r+1}}\Big| \Gc_{s,r} \Big) \Big)\Ub_s\Big(\Eb\Big( e^{-\lambda \wt L_{s}}\Big| \Gc_{s,r} \Big) \Big)\overset{\eqref{eq:EtoU},\eqref{eq:EtoU2}}{\asymp} 
	\Eb\Big( e^{-\lambda \wt L_{s}}\Big) \Eb\Big( e^{-\lambda \wt L_{s+1,s+r+1}}\Big).
\end{split}
 	\end{equation}
		Furthermore, by Lemmas~\ref{lem:comparison-principle} and \ref{lem:composition-laws}, we have 
		\begin{equation}\label{eq:extremal_distance_ge}
		\wt L_{s+r+1}\ge \wt L_{s}+\wt L_{s+1,s+r+1}.
		\end{equation}
		Combining the above estimates with the fact that $\wt L_{s+r}<\infty$ implies $H$, we obtain
		\begin{align*}
		\wt b_{s+r+1}&\stackrel{\eqref{eq:br}}{=}(s+r+1)^{-2}\Eb\Big( e^{-\lambda \wt L_{s+r+1}} \one_{\wt L_{s+r+1}<\infty} \Big)\\
		&\stackrel{\eqref{eq:extremal_distance_ge}}{\le} s^{-2}\Eb\Big( e^{-\lambda \wt L_{s}} e^{-\lambda \wt L_{s+1,s+r+1}} \one_{H} \Big)\\ 
		&\stackrel{\eqref{eq:independence}}{=} s^{-2}\Eb\left( \Eb\Big( e^{-\lambda \wt L_{s}} e^{-\lambda \wt L_{s+1,s+r+1}} \mid \Gc_{s,r} \Big) \Eb\Big( \one_{H}\mid \Gc_{s,r}  \Big) \right)\\ 
		&\stackrel{\eqref{eq:H}}{\lesssim}  s^{-2} r^{-2}
		\Eb\Big( e^{-\lambda \wt L_{s}} e^{-\lambda \wt L_{s+1,s+r+1}} \Big)\\
		&\stackrel{\eqref{eq:ed-compare}}{\lesssim}  s^{-2}\Eb\Big( e^{-\lambda \wt L_{s}} \Big) 
		r^{-2}\Eb\Big( e^{-\lambda \wt L_{r}} \Big).
		\end{align*}
		This concludes the proof.	
\end{proof}

Next, we will derive the upper bound on $\wt b_r$, which consists of showing that $\wt b_{r}$ is super-multiplicative.  
In order to use Lemma~\ref{lem:reverse-composition} to reverse the inequality \eqref{eq:extremal_distance_ge}, we need the configurations to be well-separated on the intermediate scale where we split the excursions. For this purpose, we introduce some nice events (see Figure~\ref{fig:very-nice0} for an illustration of this event). 

\begin{figure}
	\centering
	\includegraphics[scale=.4]{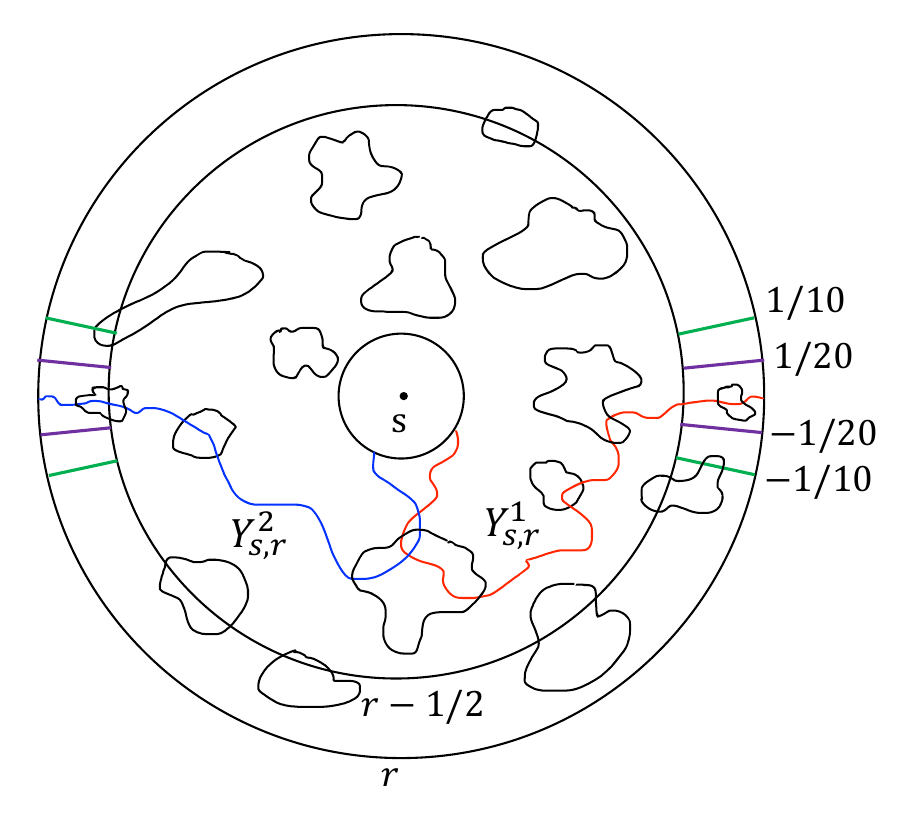}
	\caption{An illustration of the configuration that is very nice at the end (see Definition \ref{def:nice_end}). The three concentric circles are $\Cc_s$, $\Cc_{r-1/2}$ and $\Cc_r$, respectively. The black loops represent the outermost boundaries of the outermost clusters in $\Lambda_{s,r}$, and the red and blue curves are $Y_{s,r}^1$ and $Y_{s,r}^2$ respectively. The angles of the purple and green lines are indicated in the figure. }
	\label{fig:very-nice0}
\end{figure}

\begin{definition}[Very nice at the end and at the beginning]\label{def:nice_end}
For $s\ge 0$ and $r\ge s+1$, let $G^+(s,r)$ be the event that the configuration is \emph{very nice at the end}, i.e., 
	\begin{itemize}
		\item $\partial\wt O^1_{s,r}\cap\Cc_r$ is a counterclockwise (non-empty) arc on $\Cc_r$ from $Y^1_{s,r}(T^1_r)$ to $Y^2_{s,r}(T^2_r)$;
		\item $\wt Y^i_{s,r}\cap\Ac(r-1/2,r)\subseteq \{ z: |\arg(z)+(1-i)\pi| \le 1/20\}$ for $i=1,2$;
		\item $\{ z: |\arg(z)| \le 1/10 \text{ or } |\arg(z)-\pi| \le 1/10 \}\cap\Ac(r-1/2,r)$ does not disconnect $\partial\wt O^1_{s,r}\cap\Cc_s$ from $\partial\wt O^1_{s,r}\cap\Cc_r$.
	\end{itemize}
	
	We can similarly define the event  $G^-(s,r)$ that  the configuration is \emph{very nice at the beginning}:
we impose the same first condition as for $G^+(s,r)$ and replace the two last conditions  for $G^+(s,r)$ by
	\begin{itemize}
		\item $\wt Y^i_{s,r}\cap \Ac(s,s+1/2)\subseteq \{ z: |\arg(z)+(1-i)\pi| \le 1/20\}$ for $i=1,2$;
		\item $\{ z: |\arg(z)| \le 1/10 \text{ or } |\arg(z)-\pi| \le 1/10 \}\cap\Ac(s,s+1/2)$ does not disconnect $\partial\wt O^1_{s,r}\cap\Cc_s$ from $\partial\wt O^1_{s,r}\cap\Cc_r$.
	\end{itemize}
	 We write $G^+(r)$ for $G^+(0,r)$ and $G^-(r)$ for $G^-(0,r)$.	
\end{definition}

	The key ingredient to get the upper bound is the following separation lemma, which states that no matter how bad the initial configuration $\wt O^1_r$ is, a positive fraction of extensions up to scale $e^{r+1}$ is nice at the end under the probability measure weighted by $e^{-\lambda \wt L^1_{r+1}}$.
 Let $\Fc_r$ be the $\sigma$-field generated by $Y^1_{r+1}[0,T^1_r]$ and $Y^2_{r+1}[0,T^2_r]$ (note that $Y^i_{r+1}[0,T^i_r]\neq Y^i_r$ if $B^i$ makes several crossings in $\Ac(0,r)$ before reaching $\Cc_{r+1}$). 
 
	\begin{lemma}[Separation lemma]\label{lem:separation_lem_2}
		For all $r\ge 1$ and $\lambda\in (0,\lambda_0]$,
		\begin{equation}\label{eq:sep-cond}
		\Eb\Big(e^{-\lambda \wt L^1_{r+1}} \one_{G^+(r+1)} \mid\Fc_r\Big)\gtrsim
		\Eb\Big(e^{-\lambda \wt L^1_{r+1}} \mid\Fc_r\Big).
		\end{equation}
	\end{lemma}

	We remark that the proof strategy in \cite{MR1901950} (which addressed the case when $c=0$) can be adapted here to prove Theorem~\ref{thm:up-to-constants-ex}, bypassing the use of the separation lemma. However, the separation lemma is very powerful with numerous applications (see the discussion at the end of Section \ref{subsec:intro_separation}).
	For this reason, we still decide to take the (slightly) longer route and give a proof of the separation lemma in this paper.
	
In the following, we will complete the proof of Theorem \ref{thm:up-to-constants-ex} assuming Lemma~\ref{lem:separation_lem_2}. We stress that Lemma~\ref{lem:separation_lem_2} is stronger than what is necessary for proving Theorem \ref{thm:up-to-constants-ex}. We defer the more technical proof of this lemma to Section~\ref{sec:separation-lemma}.

	\begin{proof}[Proof of the upper bound in Theorem \ref{thm:up-to-constants-ex}] We assume Lemma~\ref{lem:separation_lem_2}.
It suffices to show that for all $r,s\ge 1$ and $\lambda\in(0,\lambda_0]$,
		\begin{equation}\label{eq:super_multi}
		\wt b_{s+r+2}\gtrsim  \wt b_{s} \wt b_r.
		\end{equation}
Since $e^{-\lambda \wt L^1_r}\le e^{-\lambda \wt L_r}\le e^{-\lambda \wt L^1_r}+e^{-\lambda \wt L^2_r}$, we have 
		\begin{equation}
		\Eb\Big(e^{-\lambda \wt L^1_r}\Big)\le \Eb\Big(e^{-\lambda \wt L_r}\Big) \le 2\Eb\Big(e^{-\lambda \wt L^1_r}\Big).
		\end{equation}
Thus it suffices to show
\begin{equation}\label{eq:ub-1}
		\Eb\Big(e^{-\lambda \wt L^1_{s+r+2}}\Big) \gtrsim \left( \frac{s+r}{sr} \right)^2
		\Eb\Big(e^{-\lambda \wt L^1_s}\Big) \Eb\Big(e^{-\lambda \wt L^1_r}\Big).
		\end{equation}
Note that $\wt L^1_r$ and $\wt L^1_{s+2,s+r+2}$ have the same distribution.

On the event $G^+(s)\cap G^-(s+2,s+r+2)$, if we further control  the intermediate parts of the configuration from $\Cc_{s}$ to $\Cc_{s+2}$, then we can relate $ \wt L^1_{s+r+2}$ to  $\wt L^1_s + \wt L^1_{s+2,s+r+2}$ by Lemma~\ref{lem:concatenation}.
In practice, we will first prove
\begin{align}\label{eq:ub-3}
\Eb\Big(e^{-\lambda \wt L_{s+r+2}}\Big)
		\ge C(\lambda_0)\left( \frac{s+r}{sr} \right)^2
		\Eb\Big(e^{-\lambda (\wt L^1_{s} +\wt L^{1}_{s+2,s+r+2}) } \one_{G^+(s)\cap G^-(s+2,s+r+2)} \Big).
\end{align}
Then the separation lemma (Lemma~\ref{lem:separation_lem_2}) comes into play. 
By taking expectations on both sides of \eqref{eq:sep-cond}, we get that for all $s\ge 0, r\ge s+1$, and $\lambda\in (0,\lambda_0]$, 
		\begin{align}\label{eq:sep+}
		\Eb\Big(e^{-\lambda \wt L^1_{s,r}} \one_{G^+(s,r)} \Big)\gtrsim 
		\Eb\Big(e^{-\lambda \wt L^1_{s,r}} \Big).
		\end{align}
By symmetry, we also have
		\begin{align}\label{eq:sep-}
		\Eb\Big(e^{-\lambda \wt L^{1}_{s,r}} \one_{G^-(s,r)} \Big)\gtrsim 
		\Eb\Big(e^{-\lambda \wt L^{1}_{s,r}} \Big).
		\end{align}
Similar to the argument of \eqref{eq:ed-compare}, (and again applying Lemma \ref{lem:last-exit-uniform}), the combination of \eqref{eq:sep+} and \eqref{eq:sep-} gives that
\begin{equation}\label{eq:sep-combination}
	\Eb\Big(e^{-\lambda (\wt L^1_{s} +\wt L^{1}_{s+2,s+r+2}) } \one_{G^+(s)\cap G^-(s+2,s+r+2)} \Big)\gtrsim 
	\Eb\Big(e^{-\lambda \wt L^1_{s} } \Big)
	\Eb\Big(e^{-\lambda\wt L^{1}_{s+2,s+r+2}}\Big).
\end{equation}
Plugging \eqref{eq:sep-combination} into \eqref{eq:ub-3}, we can deduce \eqref{eq:ub-1}, which will complete the proof of the theorem.

\begin{figure}[t]
		\centering
		\includegraphics[scale=0.7]{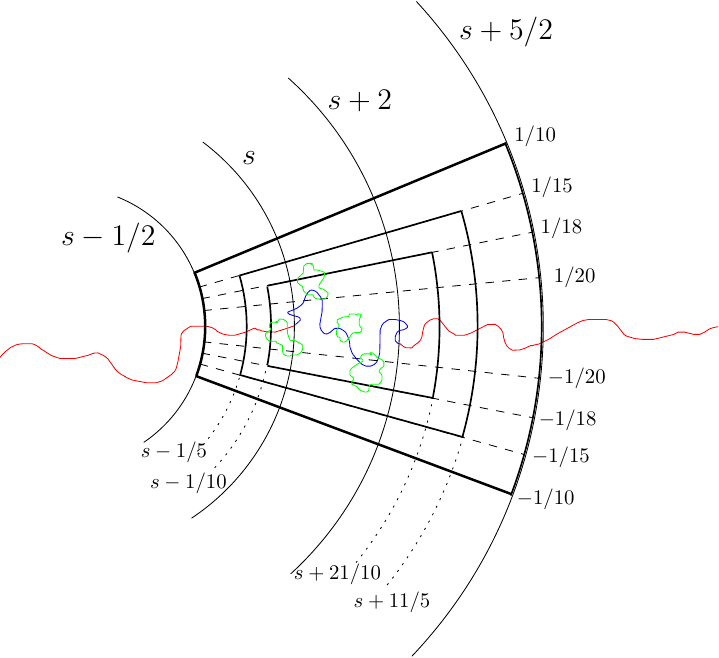}
		\caption{Illustration for wedges defined above Lemma \ref{lem:concatenation} and a typical sample from the event $J$ defined in \eqref{eq:H-def}. For convenience, we only sketch the configuration on the right hand side, corresponding to the excursion $Y^1_{s+r+2}$. 
			First, $Y^1_{s+r+2}$ is decomposed into three parts according to its first hitting on $\Cc_s$ and its last hitting on $\Cc_{s+2}$, and they are in red, blue and red, respectively. 
			The three wedges are $W_1$, $W_1'$ and $W_1^s$ in decreasing order.
			The clusters encountered by $Y^1_{s+r+2}$ and intersect with $\Ac(s,s+2)$ are sketched in green.
			The very nice event $G^+(s)\cap G^-(s+2,s+r+2)$ makes the red curves in $W_1$ stay in the cone with absolute angle $1/20$. 
			The event $U$ requires the blue curve to stay in the smallest wedge $W_1^s$. 
			The event $E$ makes sure the green clusters are sufficiently small.}
		\label{fig:very-nice}
	\end{figure}

We now focus on proving \eqref{eq:ub-3}.
In order to apply Lemma~\ref{lem:concatenation}, we need to condition the intermediate parts of the configuration from $\Cc_{s}$ to $\Cc_{s+2}$ to be good enough. 
More concretely,  let $U$ be the event that for $i=1$ and $2$,
		$B^i[\tau^i_{s},\sigma^i_{s+2,s+r+2}]$ respectively remains in the wedge $W_i^s:=\{ z: |\arg(z)+(1-i)\pi| \le 1/18\}\cap\Ac(s-1/10,s+21/10)$, the smallest wedge in Figure~\ref{fig:very-nice}. By Lemma~\ref{lem:intermediate-part}, there exists $C_1>0$ such that for all $z_1=e^{s+i\theta_1}$, $z_2=e^{s+2+i\theta_2}$, $z_3=e^{s+i\theta_3}$ and $z_4=e^{s+2+i\theta_4}$ with  $|\theta_j|\le 1/20$ for all $1\le j\le 4$,     
	\begin{equation}\label{eq:U}
		\begin{split}
			&\Pb\Big(U\mid B^1(\tau^1_{s})=z_1, B^1(\sigma^1_{s+2,s+r+2})=z_2, B^2(\tau^2_{s})=z_3, B^2(\sigma^2_{s+2,s+r+2})=z_4\Big)   \\
			\ge &C_1 \left(\frac{1}{s}\vee\frac{1}{r}\right)^2\ge \frac{C_1}{4} \frac{(s+r)^2}{s^2r^2}.
		\end{split}
	\end{equation}
Let $E$ be the event that all the clusters in $\Lambda$ that intersect $\Ac(s-1,s+3)$ are of diameter smaller than $e^s/1000$. By scaling invariance and Lemma~\ref{lem:cluster-thin}, 
		\begin{equation}\label{eq:E}
		\Pb(E)\ge C_2 \text{ for some } C_2>0. 
		\end{equation}
Denote
\begin{equation}\label{eq:H-def}
	J:=G^+(s)\cap G^-(s+2,s+r+2)\cap U\cap E
\end{equation} 
for brevity. See Figure~\ref{fig:very-nice} for an illustration.
On the event $J$, we can apply Lemma~\ref{lem:concatenation}. 
More concretely, let $W'_1,W'_2,W_1$ and $W_2$ be the four wedges defined in Lemma~\ref{lem:concatenation}. 
Let $D$, $D'$ and $D''$ denote the good domains $\wt O^1_{s+r+2}$, $\wt O^1_{s}$ and $\wt O^1_{s+2,s+r+2}$, respectively. 
Then, it is straightforward to check that on the event $J$ all the conditions in Lemma~\ref{lem:concatenation} under this setting are satisfied.
It follows from Lemma~\ref{lem:concatenation} that for some $C_3>0$,
            \begin{equation*}%\label{eq:H}
            \wt L^1_{s+r+2}\le \wt L^1_{s}+\wt L^{1}_{s+2,s+r+2}+C_3 \text{ on the event } J.
            \end{equation*}
	Therefore,
		\begin{align*}
		\Eb\Big(e^{-\lambda \wt L_{s+r+2}}\Big)&\ge \Eb\Big(e^{-\lambda (\wt L^1_{s} +\wt L^{1}_{s+2,s+r+2}+C_3)} 
		\one_{J} \Big) \\
		&\ge \frac{C_1}{4} \frac{(s+r)^2}{s^2r^2} e^{-\lambda C_3}
		\Eb\Big(e^{-\lambda (\wt L^1_{s} +\wt L^{1}_{s+2,s+r+2}) } \one_{G^+(s)\cap G^-(s+2,s+r+2)\cap E} \Big)
		\end{align*}
		where the second inequality can be obtained by first conditioning the four endpoints $B^i(\tau^i_{s})$, $B^i(\sigma^i_{s+2,s+r+2})$, $i=1,2$, then using the uniform estimate \eqref{eq:U}.
By  \eqref{eq:E} and FKG inequality (Lemma~\ref{lem:FKG}), we further have
\begin{align*}
\Eb\Big(e^{-\lambda \wt L_{s+r+2}}\Big)&\ge \frac{C_1C_2}{4}  e^{-\lambda C_3} \frac{(s+r)^2}{s^2r^2}
		\Eb\Big(e^{-\lambda (\wt L^1_{s} +\wt L^{1}_{s+2,s+r+2}) } \one_{G^+(s)\cap G^-(s+2,s+r+2)} \Big).
\end{align*}
Choosing $C(\lambda_0)=\frac{C_1C_2}{4} e^{-\lambda_0 C_3}$, we obtain \eqref{eq:ub-3} which concludes the proof of the theorem.
			\end{proof}
		
		\subsection{Up-to-constants estimates for $p(c,2,r,\lambda)$}\label{subsec:k=2}
	In this subsection, we prove Theorem~\ref{thm:xi-k} for $k=2$, which is formulated below as a corollary of Theorem~\ref{thm:up-to-constants-ex}. 
	\begin{corollary}\label{cor:xi-2}
		Let $c\in (0,1]$ and $\lambda_0>0$. For all $\lambda\in(0,\lambda_0]$
		\begin{equation}
		p(c,2,r,\lambda)\asymp e^{-r\wh\xi_c(2,\lambda)},
		\end{equation}
		where the implied constants only depend on $c,\lambda_0$. This implies that $\wh\xi_c(2,\lambda)=\xi_c(2,\lambda)$.
	\end{corollary}

	\begin{proof}
		Recall \eqref{eq:Zr} and \eqref{eq:Z-r} for the definitions of $Z_r$ and $p(c,2,r,\lambda)$.
It suffices to show that 
\begin{equation}\label{eq:ezrlambda}
\Eb(Z_r^{\lambda})\asymp r^{-2}\Eb\Big(e^{-\lambda \wt L_r} \Big).
\end{equation}
Then the corollary follows from Theorem~\ref{thm:up-to-constants-ex} directly.
		We write $Y^0_r=B^0[\sigma^0_{r},\tau^0_r]$. 
		Recall $\Theta_r$ is the union of $\overline B^1_r\cup \overline B^2_r$ together with the clusters of loops in $\Lambda_r\setminus\Lambda_0$ that it intersects. For $i=0,1,2$, let $D_i$ be the event that $B^i[0,\sigma^i_{r}]$ does not disconnect $\Cc_0$ from infinity. Then,
		\begin{equation}\label{eq:zr}
		Z_r=\Pb\Big(\ol B_r^0\cap\Theta_r=\emptyset\mid \Pi_r\Big)\le 
		\Pb\Big(Y^0_r\cap(\wt Y^1_{r}\cup \wt Y^2_{r})=\emptyset \mid \wt Y^1_{r},\wt Y^2_{r}\Big)
		\Pb(D_0) \one_{D_1\cap D_2}.
		\end{equation}
We know that the total mass of excursions across the annulus $\Ac(0,r)$ is $2\pi r^{-1}$, and up to multiplicative constants, $e^{-\wt L_r}$ measures the total mass of the excursion measure in $\wt O^1_r$ and $\wt O^2_r$, see \cite[Example 5.10]{MR2129588}. Thus,
		\begin{equation}\label{eq:ed-asym}
		\Pb\Big(Y^0_r\cap(\wt Y^1_{r}\cup \wt Y^2_{r})=\emptyset \mid \wt Y^1_{r},\wt Y^2_{r}\Big)\asymp r e^{-\wt L_r}. 
		\end{equation}
		By Lemma~\ref{lem:BM-prop-1}, we know that $\Pb(D_0)\lesssim r^{-1}$. We also know that $D_1\cap D_2$ is independent of $\wt L_r$ conditioning on $Y^1_r(0)$ and $Y^2_r(0)$, and $\Pb(D_1\cap D_2\mid Y^1_r(0), Y^2_r(0))\lesssim r^{-2}$ uniformly with respect to the locations of $Y^1_r(0)$ and $Y^2_r(0)$. Plugging all into \eqref{eq:zr}, we deduce $\Eb(Z_r^{\lambda})\lesssim r^{-2}\Eb(e^{-\lambda \wt L_r} )$.
		
		For the lower bound of \eqref{eq:ezrlambda}, consider the following events:
		\begin{itemize}
			\item Let $\delta=1/20$ and $V=\{ z: |\arg(z)-\pi/2|\le \delta \}\cap\Cc_0$. Let $U_0$ be the event that $B^0[0,\sigma^0_{r}]$ remains in the $\delta$-neighborhood of $V$. 
			\item Let $U_1$ be the event that $B^i[0,\sigma^i_{0,r}]$ remains in the wedge $\{ z: |\arg(z)+(1-i)\pi|\le 1/15\}\cap \Ac(-1/4,1/4)$ for both $i=1$ and $2$. Then $\Pb(U_1)\gtrsim r^{-2}$ (similar to \eqref{eq:ip-1}).
			\item Let $E$ be the event that all the clusters in $\Lambda_r\setminus\Lambda_0$ that intersect the unit disk have diameter smaller than $1/1000$. This event has probability bounded from below by some universal constant by Lemma~\ref{lem:cluster-thin}.
		\end{itemize}
On the event $U_0\cap U_1\cap E\cap G^-(r)$, $B^0[0,\sigma^0_{r}]$ does not intersect $\Theta_r$. Therefore,
		\begin{equation}\label{eq:Z_r-1}
		Z_r\ge
		\Pb\Big( \{Y^0_r\cap\Theta_r=\emptyset\}\cap U_0 \mid \Pi_r\Big)
		 \one_{U_1\cap E\cap G^-(r)}.
		\end{equation}
By an argument similar to the proof of \eqref{eq:ip-1}, there exists a constant $C$ such that for all $z\in V$, 
\begin{equation}\label{eq:uu}
	\Pb(U_0\mid B^0(\sigma^0_r)=z)\ge C r^{-1}.
\end{equation}
By using the uniform estimate \eqref{eq:uu}, we have 
\begin{equation}\label{eq:uu-1}
	\Pb\Big( \{Y^0_r\cap\Theta_r=\emptyset\}\cap U_0 \mid \Pi_r\Big)
	\gtrsim r^{-1} \Pb\Big( \{Y^0_r\cap\Theta_r=\emptyset\}\cap \{ B^0(\sigma^0_r)\in V \} \mid \Pi_r\Big).
\end{equation}
Note that the probability on the right hand side of \eqref{eq:uu-1} can be written as the proportion of excursion measure from $V$ to $\Cc_r$ in $\Ac(0,r)\setminus\Theta_r$ in the excursion measure across $\Ac(0,r)$ just as \eqref{eq:ed-asym}. Since the excursion measure can be written as the exponential of the corresponding extremal distance, by Lemma~\ref{lem:segment}, on the event $U_1\cap E\cap G^-(r)$,
\begin{equation}\label{eq:seg-V}
	\Pb\Big( \{Y^0_r\cap\Theta_r=\emptyset\}\cap \{ B^0(\sigma^0_r)\in V \} \mid \Pi_r\Big) \asymp \Pb\Big( Y^0_r\cap\Theta_r=\emptyset \mid \Pi_r\Big).
\end{equation} 
Moreover, applying Lemma~\ref{lem:D'}, on the event $U_1\cap E\cap G^-(r)$,
		\begin{equation}\label{eq:Z_r-2}
		\Pb( Y^0_r\cap\Theta_r=\emptyset \mid \Pi_r) 
		\gtrsim \Pb(Y^0_r\cap(\wt Y^1_{r}\cup \wt Y^2_{r})=\emptyset\mid \wt Y^1_{r},\wt Y^2_{r}).
		\end{equation}
		 Combining \eqref{eq:ed-asym}, \eqref{eq:Z_r-1}, \eqref{eq:uu-1}, \eqref{eq:seg-V} and \eqref{eq:Z_r-2}, we have
		\begin{equation}\label{eq:z_r_lambda}
		Z_r^{\lambda}\gtrsim e^{-\lambda \wt L_r} \one_{U_1\cap E\cap G^-(r)}.
		\end{equation}
Note that $U_1$ is independent of the other quantities on the right hand side of \eqref{eq:z_r_lambda}, and that $e^{-\lambda \wt L_r}$, $\one_{E}$, $\one_{G^-(r)}$ are all decreasing with respect to the loop soup. By the FKG inequality (Lemma~\ref{lem:FKG}) and the separation lemma (Lemma~\ref{lem:separation_lem_2}), taking expectations at both sides of  \eqref{eq:z_r_lambda}, we get
		\begin{equation}
		\Eb(Z_r^{\lambda})\gtrsim r^{-2} \,\Eb(e^{-\lambda \wt L_r} \one_{G^-(r)})\gtrsim r^{-2}\,\Eb(e^{-\lambda \wt L_r}).
		\end{equation}
Thus we obtain  \eqref{eq:ezrlambda} and complete the proof of the corollary.
	\end{proof}

	\subsection{Up-to-constants estimates for $p(c,k,r,\lambda)$}\label{subsec:other-exponents}
	The proof of Theorem~\ref{thm:xi-k} for general $k\in \Nb$ is completely analogous to the $k=2$ case. For the sake of brevity, we will summarize the results we need in Theorem~\ref{thm:up-to-constants-ex-k-sep} for the general case without providing the proofs. At the end of this subsection, we will show that $\xi_c(k,0)=\xi_c(k)$ in Proposition~\ref{prop:same-exponents}, where $\xi_c(k)$ is defined in the paragraph just below Theorem~\ref{thm:dis_def} and its value is given by \eqref{eq:exponent}.
	
Recall the definition of $Y^i_{s,t}$ at the beginning of Section~\ref{subsec:k=2-ee}.
	For $k\ge 2$, let $\{ Y^i_{s,t} \}_{1\le i\le k}$ be $k$ independent Brownian excursions across the annulus $\Ac(s,t)$. Write $\wt Y^i_{s,t}$ for the union of $Y^i_{s,t}$ with all the clusters in $\Lambda_{s,t}$ that intersect $Y^i_{s,t}$. Then, the set $\Ac(s,t)\backslash\bigcup_{i=1}^k \wt Y^i_{s,t}$ contains $k_0 \leq k$ connected components joining $\Cc_0$ and $\Cc_r$. Let $\wt L_{s,t}(k)$ be minimal $\pi$-extremal distance between $\Cc_s$ and $\Cc_t$ in the $j$-th component over $1\le j\le k_0$, and set $\wt L_{s,t}(k)=\infty$ if the union of $\wt Y^i_{s,t}$ disconnects $\Cc_s$ from infinity.  When $s=0$, we drop the subscript $s$.
	
	Let us now define the $\alpha$-separated event (also abbreviated as $\alpha$-sep) which is a generalization of the very nice events $G^+(s,r)$ and $G^-(s,r)$ to $k\in \Nb$ excursions.  See Figure~\ref{fig:sep} for an illustration. 
	In the current section,  this event concerns the configuration $\{ Y^i_{s,t} \}_{1\le i\le k}$ inside an independent loop soup $\Lambda_{s,t}$. However, later in Section~\ref{sec:dim}, we will replace $\{ Y^i_{s,t} \}_{1\le i\le k}$ by $k$ Brownian excursions  across the annulus $\Ac(s,t)$ with slightly different laws. Therefore, we choose to define the $\alpha$-sep event for general curves $\{ \gamma^i\}_{1\le i\le k}$ instead of excursions $\{ Y^i_{s,t} \}_{1\le i\le k}$ with specific laws.

	\begin{figure}
		    \vspace{0.4cm}
		\centering
		\includegraphics[scale=0.15]{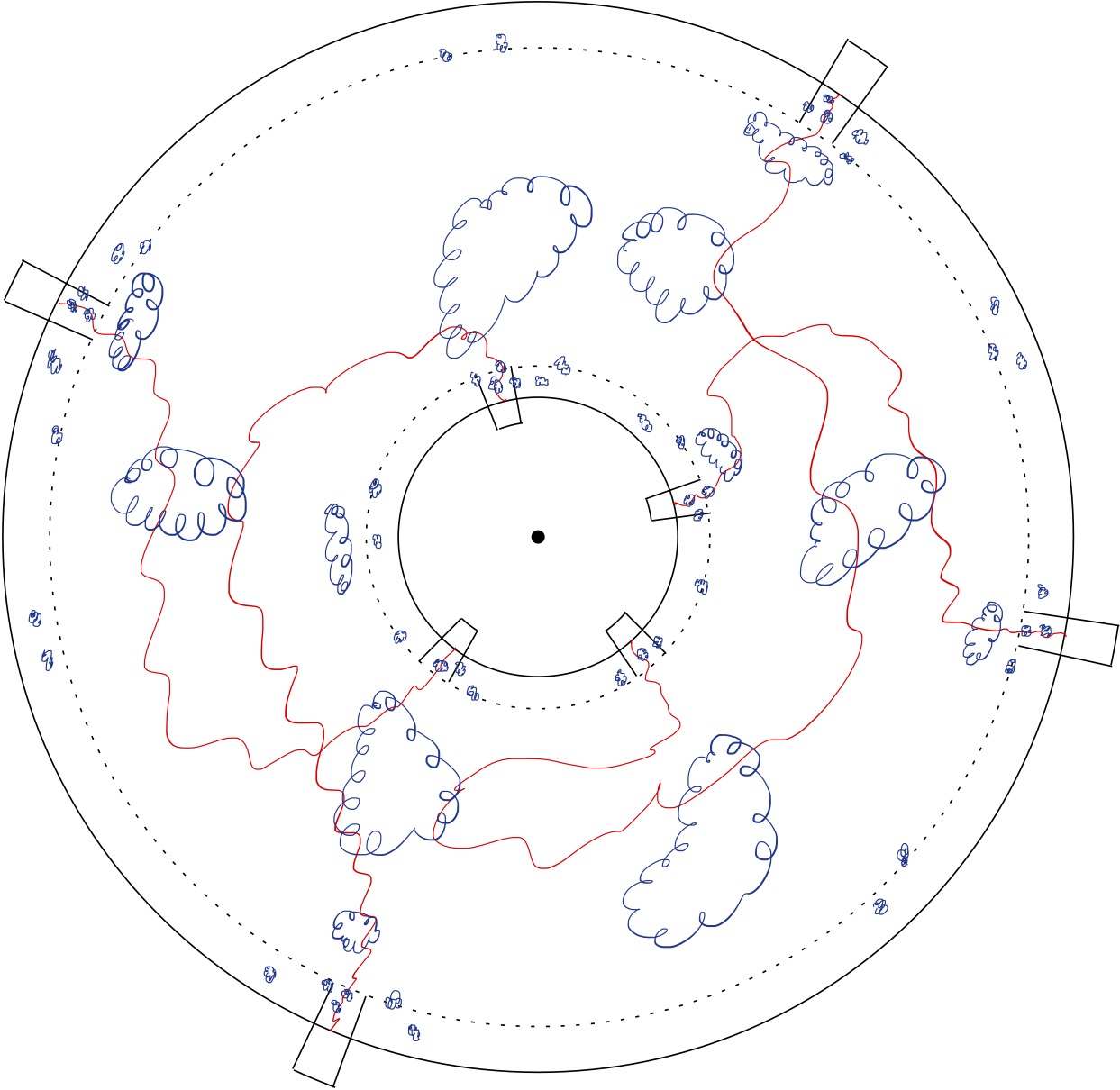}
		\caption{Sketch of the $\alpha\text{-sep}$ configurations for the case $k=4$. The annulus $\Ac(s,r)$ is bounded between solid circles. The clusters in $\Lambda_{s,t}$ are in blue. The curves $\{\gamma^i\}_{1\le i\le k}$ are in red. Their ending parts are contained in the corresponding well-separated wedges.}
		\label{fig:sep}
	\end{figure}
	
	\begin{definition}[$\alpha$-separated event]\label{def:sep}
		
		Suppose $r-s>1$ and $\alpha\in (0,1/2)$. For $1\le i\le k$, let $(\gamma^i(t): 0\le t \le t_i)$ be a curve  contained in the annulus $\Ac(s,r)$ with $\gamma^i(0) \in \Cc_s$ and $\gamma^i(t_i) \in \Cc_r$. Let $\wt\gamma^i$ be the union of $\gamma^i$ with all the clusters of $\Lambda_{s,t}$ that intersect $\gamma^i$. We define the ``landing zones'' of size $\alpha$ in the following bullet points.
		\begin{itemize}
			\item Let $\ell_i$ be the arc on $\Cc_s(x)$ of length $\alpha e^s$  centered at $\gamma^i(0)$, and associated with a wedge 
			\begin{equation*}%\label{eq:w-i}
			w_i:=\{ u z: e^{-\alpha}\le u\le e^{\alpha}, z \in \ell_i \}.
			\end{equation*}
			\item Let $\ell'_i$ be the arc on $\Cc_r(x)$ of length $\alpha e^r$ centered at $\gamma^i(t_i)$, and associated with a wedge 
			\begin{align*}
w'_i:=\{ u z: e^{-\alpha}\le u\le e^{\alpha}, z \in \ell'_i \}.
\end{align*}
		\end{itemize}
		We say the configuration $\Cs=(\{\gamma^i\}_{1\le i\le k}, \Lambda_{s,t})$ is $\alpha\text{-sep}$ if the following conditions hold. 
		\begin{itemize}
			\item[(i)] The landing zones are well-separated:
			\begin{equation*}
			d(\ell_i,\ell_{i+1})\ge \sqrt{\alpha}e^s \text{ and }  d(\ell'_i,\ell'_{i+1})\ge \sqrt{\alpha}e^r \text{ for each } i, 			\end{equation*}
			where indices are cyclic, i.e., $\ell_{k+1}=\ell_1$ and $\ell'_{k+1}=\ell'_1$.
			\item[(ii)] For  each $1\le i\le k$, $\wt\gamma^i$ lands exactly in the respective landing zone:
			\begin{equation*}
			\wt \gamma^i\subseteq \Ac(s+\alpha,r-\alpha) \cup w_i\cup w'_i.
			\end{equation*}
			\item[(iii)] 			The clusters that intersects $\Ac(s,s+\alpha) $ have diameter smaller than $\alpha e^s/100$ and the clusters that intersects $\Ac(r-\alpha,r) $ have diameter smaller than $\alpha e^r/100$.
			\item[(iv)] The union of 
$\wt \gamma^i\cup w_i\cup w'_i $ over $1\le i \le k$
			does not disconnect $\Cc_s$ from infinity.
		\end{itemize}
	\end{definition}
	We now state the following theorem which contains  Theorem~\ref{thm:xi-k}, Corollary~\ref{cor:xi(k,0)} and a generalized version of the separation lemma (Lemma~\ref{lem:separation_lem_2}) to $k\in \Nb$. 
We omit the proof as it only differs from that of the case $k=2$ in small details.

	\begin{theorem}\label{thm:up-to-constants-ex-k-sep}
		Let $c\in (0,1], k\in \Nb$ and $\lambda_0>0$. Write $\Cs$ for the configuration $(\{ Y^i_r\}_{1\le i\le k}, \Lambda_{0,r})$. There exists $\alpha_0>0$ such that for all $\lambda\in (0,\lambda_0]$ and $\alpha\in(0,\alpha_0]$,
		\begin{equation}\label{eq:up-to-constants-ex-k-sep}
		p(c,k,r,\lambda)\asymp r^{-k}\Eb(e^{-\lambda\wt L_r(k)})
		\asymp r^{-k}\Eb(e^{-\lambda\wt L_r(k)} \one_{\{\Cs\text{ is }\alpha\text{-sep} \}}) 
		\asymp e^{-r\xi_c(k,\lambda)}, 
		\end{equation}
		and 
		\begin{equation}\label{eq:k-sep}
		p(c,k,r,0)\asymp r^{-k}\Pb(\{ \wt L_r(k)<\infty \})
		\asymp r^{-k}\Pb(\{ \wt L_r(k)<\infty \}\cap\{ \Cs \text{ is }\alpha\text{-sep}\}) \asymp e^{-r\xi_c(k,0)},
		\end{equation}
		where the implied constants only depend on $c,k,\lambda_0$ and $\alpha_0$.
	\end{theorem}
	
	Finally, we show that $\xi_c(k,0)=\xi_c(k)$, which, combined with Corollary \ref{cor:xi(k,0)}, finishes the proof Theorem \ref{thm:same_exponent}.
	
		\begin{proposition}\label{prop:same-exponents}
			Let $W^1,\cdots,W^k$ be $k$ independent Brownian excursions from $0$ to $1$ in $\Ub$. 
			Let $K$ be the filled union of $W^1 \cup \cdots \cup W^k$ together with all the clusters of $\Lambda_0$ which intersect $W^1 \cup \cdots \cup W^k$. Let $q(c,k, r)$ be the probability that the conformal radius of $K$ w.r.t.\ $0$ in $\Ub$ is at most $e^{-r}$. Then
			\begin{equation}\label{eq:q}
			q(c,k,r) \asymp e^{-r\xi_c(k,0)},
			\end{equation}
			where the implied constants only depend on $c,k$. 
			As a result,
			\begin{equation}\label{eq:same-xi}
			\xi_c(k,0)=\xi_c(k).
			\end{equation}
		\end{proposition}
		\begin{proof}
			We begin with the proof of \eqref{eq:q}.
			Denote the conformal radius of $K$ from $0$ by $r_K(0)$. By Koebe $1/4$ theorem, we know that
			\[
			\frac{r_K(0)}{4}\le d(\partial K, 0) \le r_K(0).
			\]
			Moreover, $d(\partial K, 0)\le e^{-r}$ is equivalent to the condition that $\Cc_{-r}$ is not disconnected from infinity by $K$, and we denote the probability of the latter event by $p'(c,k,r)$. By \eqref{eq:k-sep}, it suffices to show that 
			\[
			p'(c,k,r) \asymp p(c,k,r,0).
			\]
			Note that the law of the family $\{ W^i[0,\tau^i_{-1}] \}_{1\le i\le k}$ and the law of the family of $k$ independent standard planar Brownian motions stopped upon reaching $\Cc_{-1}$ are mutually absolutely continuous with densities bounded by universal constants (depending on $k$ only). By scaling invariance of the Brownian motion and the Brownian loop soup,
			\[
			p'(c,k,r)\lesssim p(c,k,r-1,0) \lesssim p(c,k,r,0).
			\]
			Next, we show the other direction. Conditioned on the event that $\Cc_{-r}$ is not disconnected from infinity by $K$, we can use \eqref{eq:k-sep} to make the configuration induced by $ W^i[\sigma^i_{-r,-1},\tau^i_{-1}]$ for $1\le i\le k$ and $\Lambda_0$ be $\alpha$-sep with probability bounded from below by a universal positive constant. Furthermore, on the event that the configuration is $\alpha$-sep, one can impose mild spatial restrictions to $W^i[0,\sigma^i_{-r,-1}]$ and $W^i[\tau^i_{-1},\tau^i_0]$ for each $i$ (forcing each of them stay in a certain region) such that attaching them to the $\alpha$-sep configuration still does not disconnect $\Cc_{-r}$ from infinity; this will only cost a constant probability (depending only on $\alpha$). From this observation, we know that $p'(c,k,r)\gtrsim p(c,k,r,0)$, which completes the proof of \eqref{eq:q}.
			
				Recall the notation introduced in Section \ref{subsec:GDE}.
				As explained in the paragraph below Lemma~\ref{lem:obs}, $K$ is distributed according to $\Pf_{\kappa}^{0,k}$, where $\kappa$ and $c$ are related by~\eqref{eq:c_kappa}. Then, by Theorem~\ref{thm:dis_def}, there exists $C>0$ depending only on $c$ and $k$ such that as $r\to\infty$, 
			\[
				q(c,k,r)=p_{\kappa}^{e^r}(0,k)=e^{-r\xi_c(k)}(C+o(1)).
			\]
			Combined with \eqref{eq:q}, the above equation implies \eqref{eq:same-xi}. This finishes the proof.	\end{proof}

	\section{Separation lemma}\label{sec:separation-lemma}
	In this section, we will prove Lemma \ref{lem:separation_lem_2}, the separation lemma for the loop-soup setting, which is a generalization of the classical separation lemma \cite[Lemma 6.1]{MR1901950}.
Throughout this section, we will follow the following convention: for any $\delta\in (0,1)$, we write
\begin{equation}\label{eq:Delta}
	\Delta=\Delta(\delta):=-\log \delta
\end{equation}
to simplify notation since we will be working on exponential scales.

\subsection{General proof strategy}\label{subsec:proof_strategy}

We follow a similar general strategy as in \cite{MR1901950}, but need to deal with several additional difficulties, which makes the present proof more technical than the classical one.
In the proof of \cite{MR1901950}, the authors consider a so-called $\delta$-nice event, which roughly says that the paths of the Brownian excursions stopped at $\Cc_r$ are at least $\delta e^r$ ``away'' from each other.
Then they argue that if the configuration of the Brownian excursions up to $\Cc_r$ is $\delta$-nice,  then one can condition on the future parts of the Brownian excursions up to $\Cc_{r+1}$ to stay in well-chosen disjoint wedges with probability $\delta^{\alpha}$ for some $\alpha>0$, so that on this event $L_{r+1}$ is not much bigger than $L_r$, satisfying
\begin{align}\label{eq:sepa_classic}
L_{r+s} \le L_{r+1} \le L_r +C\Delta  \, \text{ for all } \,  s \in [0,1],
\end{align}
where $\Delta=-\log \delta$ by \eqref{eq:Delta} and $C$ is some universal constant.
They also rely on the evident fact that $L_t\ge L_s$ whenever $t\ge s$.
When we have the additional Brownian loop soup, it is still evident that  $\wt L_t\ge \wt L_s$ whenever $t\ge s$. However, when we go from the configuration in $\Bc_r$ to $\Bc_{r+1}$, it costs a lot to control the loops in $\Lambda_{r+1}$ which intersect $\Cc_r$. For example, if the configuration in $\Bc_r$ is ``$\delta$-nice'' (loosely speaking, the natural analogue of the classical $\delta$-nice event is that the union of the Brownian excursions together with the loops in $\Lambda_r$ is $\delta e^{r}$ ``away'' from disconnecting the origin), then in order for $\wt L_{r+1}$ to not increase too much (or at least to stay finite), one would like to condition the loops in $\Lambda_{r+1}\setminus \Lambda_r$ to satisfy some ``good'' event. For example, we can consider the event that all the loops in $\Lambda_{r+1}$ that intersect $\Cc_r$ have diameter smaller than $\delta e^{r}$. 
However, this event has probability of order $\exp(-\alpha/\delta)$ for some $\alpha>0$, which is too small to allow the whole argument to work.
In fact, following this line, we have been unable to find a good event with sufficiently large probability (i.e., polynomial on $\delta$) on which $\wt L_{r+1}$ is comparable to $\wt L_r$.

To overcome this difficulty, we would explore the Brownian excursions and the loop soup at two different speeds. 
More concretely, we define a $\delta$-nice event (see Definition~\ref{def:delta-nice}) which is measurable with respect to  the $\sigma$-algebra generated by the Brownian excursions up to $\Cc_{r+s}$ and the loops in $\Lambda_{r+t}$ for $0\le s < t \le 1$.  
On this $\delta$-nice event,  we also consider an auxiliary $\pi$-extremal distance $L^*(r+s, r+t, \delta)$ (see \eqref{eq:L*}) which is defined using $Y^1_{r+s}$, $Y^2_{r+s}$ and the loop soup $\Lambda_{r+t}$. 
The quantity $L^*(r+s, r+t, \delta)$ should be thought as a replacement for the $\pi$-extremal distance $\wt L_{r+s}$. Note that the former contains more information about the loop soup  than the latter. 
The inequality $\wt L_{r+s} \le \wt L_{r+t}$ would be replaced by the following less evident one, which holds on the aforementioned $\delta$-nice event (see \eqref{eq:bar-le-wt})
\begin{align*}
L^*(r+s, r+t, \delta) \le \wt L_{r+t} + C \cdot (t-s+\Delta+1).
\end{align*}
It states that  $L^*(r+s, r+t, \delta)$ is not much bigger than $\wt L_{r+t}$. 
More importantly, with this definition we are able to show the analogue of \eqref{eq:sepa_classic}, namely on some ``good'' event which can be achieved with a sufficiently large probability (i.e., polynomial on $\delta$), we have (see \eqref{eq:bar-ge-wt})
\begin{align*}
\wt L_{r+t} \le L^*(r+s, r+t, \delta).
\end{align*}
It turns out that these inequalities would allow the whole strategy to work. 

In Section~\ref{subsec:aux_ed}, we define the relevant auxiliary extremal distances and establish a first inequality.
In Section~\ref{subsec:good_events}, we define some good events and prove the aforementioned inequalities about the extremal distances on these good events. In Section~\ref{subsec:proof_sep}, we complete the proof of the separation lemma. 
Section~\ref{subsec:proof_sep} is in the same spirit as the proof in \cite{MR1901950}, but contains some more technicalities, since we need to adjust things to the new setting. In particular, since the auxiliary extremal distances depend on a larger $\sigma$-algebra, we need to regain independence between random variables in some steps of the proof.

\subsection{Auxiliary extremal distances}\label{subsec:aux_ed}
Throughout Section~\ref{sec:separation-lemma}, we will assume the excursions $Y^1(t)$ and $Y^2(t)$ are independently defined using Bessel processes (see definition (\ref{bessel}) in Section~\ref{subsec:facts} above Remark \ref{rem:adv} and the explanation on advantages of this equivalent definition in Remark~\ref{rem:adv}). Still, denote by $T^i_u$ the hitting time of $\Cc_u$ for $Y^i$. Then, $Y^i_r:=Y^i[0,T^i_r]$ has the same distribution as $B^i[\sigma^i_{r},\tau^i_r]$ for all $r>0$. Moreover, $Y^i_{r+s}[0,T^i_r]=Y^i_r$ for all $s\ge 0$. With this definition of $Y^i$, we can now simply define $\Fc_r$ as the $\sigma$-algebra generated by $Y^1_r$ and $Y^2_r$. 

We define afresh the quantities $\wt Y^i_r$ (union of $Y^i_r$ and loop soup clusters), $\wt O_r^i$ (openings between $\wt Y^i_r$) and $\wt L_r^i$ (extremal distances between $\Cc_0$ and $\Cc_r$ in $O_r^i$) in this section in the same fashion as in the previous section (see \eqref{eq:Y} and below), except here we use $Y^i_r$ instead of $B^i[\sigma^i_{r},\tau^i_r]$.
Note that this does not change the law of these quantities (but the definition of the filtration $\Fc_r$ above is slightly different from the filtration defined in Section \ref{sec:up-to-constant}; see above Lemma \ref{lem:separation_lem_2}).

Let us first define some wedges which are used to define the auxiliary extremal distances. Let $z_1$ and $z_2$ be two points on the circle $\Cc_r$. First suppose that the counterclockwise arc from $z_1$ to $z_2$ has angle at most $\pi$.
Recall our convention from \eqref{eq:Delta} that $\delta=e^{-\Delta}\in (0,1)$.
Let the oriented line segment from $z_1$ (resp.\ $z_2$) to the origin rotate\footnote{Rotation by $\pi/4$ prevents the wedges $W_\pm(z_1,\delta)$ and $W_\pm(z_2,\delta)$ defined below from intersecting each other.} a clockwise (resp.\ counterclockwise) angle $\pi/4$ around $z_1$  (resp.\ $z_2$)  to intersect some point $x_1$  (resp.\ $x_2$) on the half circle $\Cc_{r-\Delta}(z_i) \cap \Bc_r$. Let 
\begin{align*}
W(z_i,\delta):=\{z=x_i+ \rho (z_i-x_i) e^{i \theta}: \rho>0, \, \theta \in (-1/200, 1/200)\},\\
W_+(z_i,\delta):=\{z=x_i+ \rho (z_i-x_i) e^{i \theta}: \rho>0, \, \theta \in (-1/100, 1/100)\},\\
W_-(z_i,\delta):=\{z=x_i+ \rho (z_i-x_i) e^{i \theta}: \rho>0, \, \theta \in (-1/400, 1/400)\}.
\end{align*}
If  the counterclockwise arc from $z_1$ to $z_2$ has angle more than $\pi$, then let $y_1:=z_2$ and $y_2:=z_1$. For $i=1,2$, let $W(z_i, \delta):=W(y_{3-i}, \delta)$ and $W_\pm(z_i, \delta):=W_\pm(y_{3-i}, \delta)$. See Figure~\ref{fig:wedges} for an illustration of $W(z_i, \delta)$.

	\begin{figure}
	\centering
	\includegraphics[scale=.8]{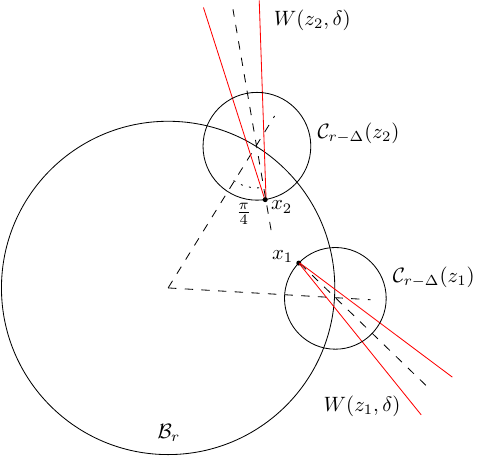}
	\caption{Sketch of $W(z_i,\delta)$ when the counterclockwise arc from $z_1$ to $z_2$ has angle at most $\pi$. The largest ball is $\Bc_r$. The two small circles with centers on $\Cc_r$ are $\Cc_{r-\Delta}(z_1)$ and $\Cc_{r-\Delta}(z_2)$ respectively. For $i=1,2$, $W(z_i,\delta)$ is the wedge in red with vertex $x_i$ and angle $1/100$.}
	\label{fig:wedges}
\end{figure}

	We can now define some auxiliary extremal distances. Let $0\le s<t$. Let $\Ic^1$ be the counterclockwise arc on $\Cc_0$ from $Y^1(0)$ to $Y^2(0)$. 
	Let $\Lambda^{\mathrm{int}}_{s,t}$ be the collection of clusters in $\Lambda_{0,t}$ that intersect both $\Cc_s$ and $Y^1_s\cup Y^2_s$. We call $(K^1,K^2)$ an admissible pair from $s$ to $t$ if for each $i$, $K^i$ is a connected set that joins $Y^i(T^i_{s})$ to $\Cc_{t}$ and $K^i$ does not intersect $\Cc_0$. Then the set 
	\[
	\Ac(0,t)\setminus \left( \wt Y^1_s\cup \wt Y^2_s\cup\Lambda^{\mathrm{int}}_{s,t}\cup K^1\cup K^2  \right)
	\]
	contains a unique connected component $O^*:=O^*(s,t, K^1,K^2)$ that has $\Ic^1$ on its boundary. See Figure~\ref{fig:O-aux}. If $O^*\cap\Cc_{t} \neq \emptyset$ then let
		\[
		L^*(s,t, K^1,K^2)=L(O^*;O^*\cap\Cc_0,O^*\cap\Cc_{t}),
		\] 
		otherwise set $L^*(s,t,K^1,K^2)=\infty$.	
We also define the following abbreviations
	\begin{align}
	\label{eq:W-del}
	&W^i_s(\delta):=W(Y^i(T^i_s),\delta), \quad  W^i_{s,\pm}(\delta):=W_{\pm}(Y^i(T^i_s),\delta),\\
		\label{eq:B-del}
	&\Bc^i_s(\delta):=\Bc_{s-\Delta}(Y^i(T^i_s)), \\
	\notag
	&O^*(s,t,\delta):=O^*(s,t,\Bc^1_s(\delta)\cup W^1_s(\delta),\Bc^2_s(\delta)\cup W^2_s(\delta)),\\
\label{eq:L*}
	&L^{*}(s,t, \delta):=L^*(s,t,\Bc^1_s(\delta)\cup W^1_s(\delta),\Bc^2_s(\delta)\cup W^2_s(\delta)),\\
\label{eq:L+}
	&L^+(s,t,\delta):=L^*(s,t,\Bc^1_s(\delta)\cup W^1_{s,+}(\delta/e),\Bc^2_s(\delta)\cup W^2_{s,+}(\delta/e)).
	\end{align}
Note that $O^*(s,t,\delta)$, $L^{*}(s,t, \delta)$ and $L^+(s,t,\delta)$ are measurable w.r.t.\ the $\sigma$-algebra generated by the loops in $\Lambda_{t}$ and the Brownian excursions up to $\Cc_{s}$.

The following lemma states that on the event $L^+(r+s,r+t,\delta)<\infty$, for any admissible pair $(K^1,K^2)$, no matter how it looks like, $L^*(r+s,r+t, K^1, K^2)$ cannot be too small compared to $L^*(r+s,r+t,\delta/e)$. 
The moral is that we have defined the wedges $W^1_s(\delta)$ and $W^2_s(\delta)$ to be sufficiently away from each other, so that
$L^*(r+s,r+t,\delta/e)$ is not too far from minimizing the $\pi$-extremal distance  $L^*(r+s,r+t, K^1, K^2)$ over all admissible pairs $(K^1, K^2)$.
The proof is similar to that of Lemma~\ref{lem:D'}.
	\begin{lemma}\label{lem:admissible}
		There exists a universal constant $C>0$ such that for all $r, t>0$, $0\le s<t$, and $0<\delta<(e^{t-s}-1)\wedge 1$, if 
		$L^+(r+s,r+t,\delta)<\infty$,
		then for any admissible pair $(K^1,K^2)$ from $r+s$ to $r+t$, 
		\begin{equation}\label{eq:aux-compare}
		L^*(r+s,r+t,\delta/e)\le L^*(r+s,r+t, K^1, K^2)+ C\cdot(t-s+\Delta+1).
		\end{equation}
	\end{lemma}
		\begin{figure}
	\centering
	\includegraphics[scale=.9]{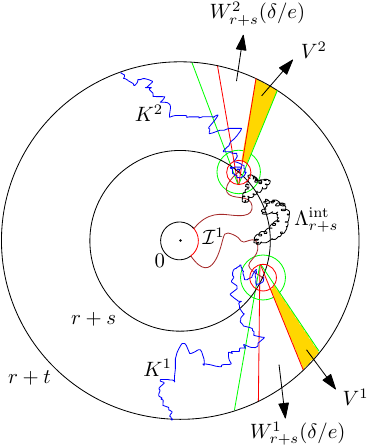}
	\caption{Proof of Lemma~\ref{lem:admissible}. $\Ic^1$ is  the arc on $\Cc_0$ in red, from $Y^1(0)$ to $Y^2(0)$ in counterclockwise order. $\Bc^i_{r+s}(\delta/e)$ (resp.\ $\Bc^i_{r+s}(\delta)$) for $i=1,2$ are the two balls with centers on $\Cc_{r+s}$ in red (resp.\ green). $W^i_{r+s}(\delta/e)$ (resp.\ $W^i_{r+s,+}(\delta/e)$) for $i=1,2$ are the two wedges in red (resp.\ green). $V^1$ and $V^2$ are the yellow parts. The clusters in $\Lambda^{\mathrm{int}}_{r+s}$ are drawn in black (only two possible clusters are sketched here). We give an example of admissible pair $(K^1,K^2)$ in blue.}
	\label{fig:O-aux}
\end{figure}
	\begin{proof}    	
		Abbreviate $V^3:=O^*(r+s, r+t, \delta/e)$ and $V^4:=O^*(r+s, r+t, K^1,K^2)$.
		By Lemma~\ref{lem:comparison-principle} (the comparison principle), it suffices to show \eqref{eq:aux-compare} for the case when $V^3 \subseteq V^4$.
We refer the reader to Figure~\ref{fig:O-aux} for the objects defined below.  For $i=1,2$, let $V^i$ be the connected component of
		\[
		\big(W^i_{r+s,+}(\delta/e)\setminus W^i_{r+s}(\delta/e)\big)\cap \Bc_{r+t}\setminus \Bc^i_{r+s}(\delta/e)
		\]  
		that is within zero distance from $V^3\cap\Cc_{r+t}$. Recall good domains in Definition~\ref{def:good-domain} whose boundaries in counterclockwise order are denoted by $\partial_1, \partial_3, \partial_2, \partial_4$.
For $i=1,2,3,4$, we denote the four parts of $\partial V^i$ by $\partial_1^i,\partial_3^i,\partial_2^i,\partial_4^i$. 
		Let $L^{i}=L(V^{i};\partial^{i}_1,\partial^{i}_2)$. Note that $\partial_1^3=\partial_1^4=\Ic^1$ and
\begin{align}\label{eq:l3l4}
L^3=L^*(r+s,r+t,\delta/e), \qquad  \quad L^4=L^*(r+s,r+t, K^1, K^2).
\end{align}
		Let $\rho^{i}$ be the extremal metric finding the length of the collection $\Gamma^{i}$ of curves in $V^i$ connecting $\partial_3^i$ and $\partial_4^i$ so that 
		\[
		A_{\rho^i}(V^i)=\pi^{-1}L^i.
		\]
		 Define the metric 
		\[
		\rho:=\max \big\{ \rho^4, \rho^1 \one_{V^1}, \rho^2 \one_{V^2}, 2\delta^{-1}e^{-(r+s)} (1_{\Bc^1_{r+s}(\delta)}+1_{\Bc^2_{r+s}(\delta)}) \big\} \text{ in } V^3,
		\]
		Since $L^+(r+s,r+t,\delta)<\infty$,  every curve in $\Gamma^3$ has length at least one in the metric $\rho$. Hence,
		\begin{align*}
&L^3\le \pi A_{\rho}(V^3)\\
		\le& \pi \Big( A_{\rho^4}(V^4)+A_{\rho^1}(V^1)+A_{\rho^2}(V^2)+4\delta^{-2}e^{-2(r+s)}2\pi (\delta^2-\delta^2/4)e^{2(r+s)} \Big)\\
		=&L^4+L^1+L^2+6\pi^2.
		\end{align*}
		By estimates similar to \eqref{eq:wedge-ex}, we know that for $i=1,2$, $L^i\lesssim t-s+\Delta$. Combined with \eqref{eq:l3l4}, this concludes the lemma.
	\end{proof}

\subsection{Good events and some estimates}\label{subsec:good_events}

In this subsection, we define some good events which will be used to control the evolution of the extremal distances in the next subsection. We will also provide some lower bounds for the probabilities of such events.
\begin{definition}[$\delta$-nice event]\label{def:delta-nice}
	For any $0<s<t$, denote by $\Upsilon(s,t,\delta)$ the event that 
	there exists a continuous curve $\gamma: [0,\infty)\rightarrow \mathbb C$ such that $\gamma\cap\Cc_0=\gamma(0)\in\Ic^1$, $\gamma(t)\rightarrow\infty$ as $t\rightarrow\infty$ and $\gamma\cap\Big( Y^1_{s}\cup Y^2_{s}\cup \Bc^1_s(\delta) \cup \Bc^2_s(\delta)\cup \Lambda_{0,t} \Big)=\emptyset$. 
\end{definition}
Note that $\Upsilon(s,t,\delta)$ is measurable w.r.t.\ the $\sigma$-algebra generated by the loops in $\Lambda_{t}$ and the Brownian excursions up to $\Cc_{s}$. Also note that the event $\Upsilon(s,t,\delta)$ is increasing as $\delta$ decreases.

In the following, we assume that $r>0$, $0\le s<t$ and $0<\delta<(e^{t-s}-1)\wedge 1$.
\begin{definition}[Excursions in wedges]\label{def:exc-in-wed}
	Denote by $\Ec(r+s,r+t, \delta)$ the event that 
\begin{align*}
Y^i[T^i_{r+s},T^i_{r+t}]\subseteq W^i_{r+s,-}(\delta)\, \text{ for both }\, i=1,2. 
\end{align*}
	\end{definition}
	Let $\alpha_1>0$ be a constant such that the probability that a complex Brownian motion starting at $(\delta,0)$ reaches the unit circle without leaving the wedge $\{ z: |\arg(z)|\le 1/400 \}$ is at least $\delta^{\alpha_1}$ for all $\delta>0$.
	Then, by scaling invariance, we have 
	\begin{equation}\label{eq:2BM_wedge}
	\Pb(\Ec(r+s,r+t, \delta) \mid \Fc_{r+s})\ge e^{-2\alpha_1(t-s+\Delta)}.
	\end{equation}
\begin{definition}[Loop soup being small]\label{def:loop_delta_small}
For any $z\in \Cc_{r+s}$, denote by $\Lc(r+s,r+t, z,\delta)$ the event that the loop soup is small around $z$ in the following sense:
\begin{itemize}
		\item All clusters in $\Lambda_{r+t}$ that intersect $\Bc_{r+s-\Delta}(z)$ have diameter at most $e^{r+s-\Delta}/100$.
\item All clusters in $\Lambda_{r+t}$ that intersect the annulus $\Bc_{r+s-\Delta+k}(z)\setminus \Bc_{r+s-\Delta+k-1}(z)$ have diameter at most $e^{r+s-\Delta+k-1}/100$ for all $k\in\Nb$ with $k\le k_0:= \left \lfloor t-s+\Delta+1 \right\rfloor$. 
	\end{itemize}
We also define
$$
\Lc(r+s,r+t, \delta):=\Lc(r+s,r+t, Y^1(T^1_{r+s}),\delta)\cap \Lc(r+s,r+t, Y^2(T^2_{r+s}),\delta).$$
\end{definition}

We remark that for all $\delta>0$, we have $\Lc(r+s,r+t, \delta/e) \subseteq \Lc(r+s,r+t, \delta)$.

\begin{lemma}
There exists a constant $\alpha_2>0$ such that for all $z\in \Cc_{r+s}$,
	\begin{equation}\label{eq:loop_small}
	\Pb(\Lc(r+s,r+t, z,\delta))\ge e^{-\alpha_2(t-s+\Delta+3)}.
	\end{equation}
\end{lemma}
\begin{proof}
By Lemma~\ref{lem:cluster-thin} and scaling invariance of Brownian loop soup, we know that for some universal constant $\alpha_2>0$ the first item in Definition~\ref{def:loop_delta_small} happens with probability at least $e^{-\alpha_2}$, which also holds for the event in the second item for each $k\le k_0$. Therefore, 
\[
\Pb(\Lc(r+s,r+t, z,\delta))\ge e^{-\alpha_2(k_0+1)}\ge e^{-\alpha_2(t-s-\Delta+3)}.
\]
This finishes the proof.
\end{proof}
By the FKG inequality (Lemma~\ref{lem:FKG}) and the above lemma, we obtain
	\begin{equation}\label{eq:2loop_small}
	\Pb(\Lc(r+s,r+t, \delta))\ge e^{-2\alpha_2(t-s+\Delta+3)}.
	\end{equation}

We are now ready to prove the key lemma which compares $L^*(r+s,r+t, \delta/e)$ (see \eqref{eq:L*}) with $\wt L^1_{r+t}$ (see the beginning of Section~\ref{subsec:k=2-ee}) on good events.
	\begin{lemma}\label{lem:ed-comparable}
On the event $\Upsilon(r+s,r+t, \delta)\cap\Lc(r+s,r+t, \delta)$, with the same constant $C$ in Lemma~\ref{lem:admissible}, we have
		\begin{equation}\label{eq:bar-le-wt}
		L^*(r+s,r+t, \delta/e)\le \wt L^1_{r+t}+ C\cdot(t-s+\Delta+1).
		\end{equation}
		Furthermore, on the event $\Upsilon(r+s,r+t, \delta)\cap\Lc(r+s,r+t, \delta)\cap\Ec(r+s,r+t,\delta/e)$, we have
		\begin{equation}\label{eq:bar-ge-wt}
		\wt L^1_{r+t}\le L^*(r+s,r+t, \delta).
		\end{equation}
	\end{lemma}
	
	\begin{proof}
		We first show \eqref{eq:bar-le-wt}. The key observation is that 		
		on the event $\Lc(r+s,r+t, \delta)$, for both $i=1,2$, there is no cluster in $\Lambda_{r+t}$ that intersects both $\Cc_{r+s}$ and 
$W^i_{r+s,+}(\delta/e)\setminus\Bc^i_{r+s}(\delta)$ (recall \eqref{eq:W-del} and \eqref{eq:B-del} for the definition), for the following reasons. Let $z_i:=Y^i(T^i_{r+s})$ for $i=1,2$. 
For each $1\le k \le k_0$ (recall $k_0$ in Definition~\ref{def:loop_delta_small}), if a cluster in $\Lambda_{r+t}$ intersects both $\Cc_{r+s}$ and the following set
$$W^i_{r+s,+}(\delta/e) \cap \big( \Bc_{r+s-\Delta+k}(z_i)\setminus \Bc_{r+s-\Delta+k-1}(z_i) \big),$$ 
then this cluster must have diameter at least  $e^{r+s-\Delta+k-1}/10$. However, on the event $\Lc(r+s,r+t, z_i,\delta)$, such clusters do not exist for any $1\le k \le k_0$. It follows that on the event $\Upsilon(r+s,r+t, \delta)\cap\Lc(r+s,r+t, \delta)$,
we have $L^+(r+s,r+t, \delta)<\infty$. We can now apply Lemma~\ref{lem:admissible} and deduce
		\begin{equation*}%\label{eq:bar-le-wt-1}
		L^*(r+s,r+t, \delta/e)\le L^*(r+s,r+t, Y^1[T^1_{r+s},T^1_{r+t}],Y^2[T^2_{r+s},T^2_{r+t}])+ C\cdot(t-s+\Delta+1).
		\end{equation*}
		Moreover, by Lemma \ref{lem:comparison-principle} (the comparison principle),
		\begin{equation*}%\label{eq:bar-le-wt-2}
		L^*(r+s,r+t, Y^1[T^1_{r+s},T^1_{r+t}],Y^2[T^2_{r+s},T^2_{r+t}])\le \wt L^1_{r+t}.
		\end{equation*}
		Combined, we obtain \eqref{eq:bar-le-wt}.
\smallskip
		
		We then turn to the proof of \eqref{eq:bar-ge-wt}. Note that  on the event $\Lc(r+s,r+t, \delta)$, for each $1\le k \le k_0$, every cluster in $\Lambda_{r+t}$ that intersects 
	$$W^i_{r+s,-}(\delta/e) \cap \big( \Bc_{r+s-\Delta+k}(z_i)\setminus \Bc_{r+s-\Delta+k-1}(z_i) \big),$$
has diameter at most $e^{r+s-\Delta+k-1}/100$, hence is contained in $W^i_{r+s}(\delta/e)\cup\Bc^i_{r+s}(\delta)$. 
Also, on $\Lc(r+s,r+t, \delta)$, every cluster in $\Lambda_{r+t}$ that intersect $\Bc_{r+s-\Delta-1}(z_i)$ has diameter at most $ e^{r+s-\Delta}/100$, hence is contained in $\Bc^i_{r+s}(\delta)$.
It follows that on $\Lc(r+s,r+t, \delta)$, for both $i=1,2$, the union of the wedge $W^i_{r+s,-}(\delta/e)$ and the clusters in $\Lambda_{r+t}$ it intersects is contained in $W^i_{r+s}(\delta/e)\cup\Bc^i_{r+s}(\delta)$. 

On the event $\Ec(r+s,r+t, \delta/e)$, we have $Y^i[T^i_{r+s},T^i_{r+t}]\subseteq W^i_{r+s,-}(\delta/e)$. Therefore, on $\Lc(r+s,r+t, \delta)\cap\Ec(r+s,r+t, \delta/e)$, the union of $Y^i[T^i_{r+s},T^i_{r+t}]$ and the clusters in $\Lambda_{r+t}$ it intersects is contained in $W^i_{r+s}(\delta)\cup\Bc^i_{r+s}(\delta)$. 
Consequently, on the event $\Upsilon(r+s,r+t, \delta)\cap\Lc(r+s,r+t, \delta)\cap\Ec(r+s,r+t, \delta/e)$, we have $O^*_{r+s,r+t}(\delta) \subseteq \wt O^1_{r+t}$ and \eqref{eq:bar-ge-wt} holds by Lemma \ref{lem:comparison-principle}.
	\end{proof}

\subsection{Proof of the separation lemma}\label{subsec:proof_sep}
	In this subsection, we finally give the proof of the separation lemma (Lemma~\ref{lem:separation_lem_2}). 
	
	To start with, let us prove a weak version of it in the next lemma. That is, for any initial configuration of Brownian excursions up to $\Cc_r$,  under the measure tilted by $\exp(-\lambda \wt L^1_{r+1})$, with uniformly positive probability, the continuations of configuration will be uniformly $\delta_0$-nice (see Definition~\ref{def:delta-nice}) at some stopping time before reaching $\Cc_{r+1/4}$ for some $\delta_0>0$ defined in Lemma \ref{lem:separation_lemma_1}. Recall that $\Fc_r$ is the $\sigma$-algebra generated by $Y^1_r$ and $Y^2_r$.
	\begin{lemma}\label{lem:separation_lemma_1}
		There exist universal constants $C>0$, $\delta_0\in (0,1/2)$ and a stopping time (w.r.t.\  $\Fc_r$) $\zeta\in [r,r+1/4]$ such that
		\begin{equation}
		\Eb\big(e^{-\lambda \wt L^1_{r+1}} \one_{\Upsilon(\zeta, r+1,\delta_0)} \mid\Fc_{r}\big)
		\ge C \Eb\big(e^{-\lambda \wt L^1_{r+1}}\mid\Fc_{r}\big).
		\end{equation}
	\end{lemma}
	\begin{proof}
In this proof, we make the following abbreviations: for any $\Delta>0$ 
$$\Upsilon(s,\Delta):=\Upsilon(s,r+1, \delta), \,\, \Ec_\Delta:=\Ec(r,r+1, \delta), \,\, \Lc_\Delta:=\Lc (r,r+1, \delta),  \,\, L^{*}_\Delta:=L^{*}(r,r+1,\delta).$$
Recall our convention in \eqref{eq:Delta} that $\delta=-\exp(-\Delta)$) and Definitions~\ref{def:delta-nice},~\ref{def:exc-in-wed},~\ref{def:loop_delta_small} and \eqref{eq:L*} for definitions of the above quantities respectively.
For any positive integer $m$, let 
		\[
		\tau_m:=\inf\{ 0\le s\le 1: \Upsilon(r+s,m)\mbox{ occurs}\}\wedge 1.
		\]
		For each $1\le k\le m^2$, let $D_{k,m}$ be the event that  $Y^i\big[T^i_{r+(k-1)e^{-m/2}}, T^i_{r+ke^{-m/2}}\big]$ does not disconnect the disk of radius $e^{-m} e^{r+(k-1)e^{-m/2}}$ centered at $Y^i\big(T^i_{r+(k-1)e^{-m/2}}\big)$ from infinity for both $i=1$ and $2$. Then there exists $m_0\ge 1$ such that for all $m\ge m_0$ and $1\le k\le m^2$, we have $\Pb(D_{k,m})\le (1-\beta)^2$, where 
		$\beta$ is the probability that a planar Brownian motion started from the origin and stopped upon reaching the circle of radius $2$ centered at the origin disconnects the unit disk from infinity. Let $D_{m}:=\bigcap_{k=1}^{m^2} D_{k,m}$. From the strong Markov property of Brownian motion, we get that for some universal constant $a>0$,
		\begin{equation}\label{eq:ineq-for-D}
		\Pb(D_m)\le e^{-am^{2}}.
		\end{equation} 
Note that 
		\begin{equation}\label{eq:subset_D}
		\{ \wt L^1_{r+1}<\infty \}\cap \{ \tau_m\ge m^{2}e^{-m/2} \}\subseteq D_m.
		\end{equation}
		It then follows that
		\begin{align*}
		\Eb(e^{-\lambda \wt L_{r+1}^{1}} \one_{\tau_{m} \geq m^{2}e^{-m/2} } \mid \mathcal{F}_{r}) e^{-2\alpha_2(m+6)}
		&\le \Eb(e^{-\lambda \wt L_{r+1}^{1}} \one_{D_m} \mid \mathcal{F}_{r}) \Pb(\Lc_{m+2})\\
		&\le \Eb(e^{-\lambda \wt L_{r+1}^{1}} \one_{D_m} \one_{\Lc_{m+2}} \mid \mathcal{F}_{r}),
		\end{align*}
		where we use \eqref{eq:subset_D} and \eqref{eq:2loop_small} in the first inequality, and use Lemma~\ref{lem:FKG} (FKG inequality) in the last inequality since $e^{-\lambda \wt L_{r+1}^{1}}$ and $\one_{\Lc_{m+2}}$ are both decreasing with respect to the loop soup and $D_m$ and $\mathcal{F}_{r}$ are both independent of the loop soup. (Note that $\{ \tau_m\ge m^{2}e^{-m/2} \}$ is increasing with respect to the loop soup.) In order to use Lemma~\ref{lem:FKG}, we have replaced it with $D_m$ to gain independence.
		Moreover, 
		\begin{align*}
		&\Eb(e^{-\lambda \wt L_{r+1}^{1}} \one_{D_m} \one_{\Lc_{m+2}} \mid \mathcal{F}_{r})
		\one_{\Upsilon(r,m+1)}\\
		\le\; &e^{\lambda C (m+2)}\Eb(e^{-\lambda L^*_{m+2}} \one_{D_m} \one_{\Lc_{m+2}} \mid \mathcal{F}_{r})
		\one_{\Upsilon(r,m+1)}\\
		\le\;& e^{\lambda C (m+2)}e^{-am^2}\Eb(e^{-\lambda L^*_{m+2}} \one_{\Lc_{m+2}} \mid \mathcal{F}_{r})
		\one_{\Upsilon(r,m+1)}.
		\end{align*}
In the first inequality above, we use \eqref{eq:bar-le-wt} and the fact that $\Lc_{m+1}$ holds on the event $\Lc_{m+2}$.
In the second inequality we combine \eqref{eq:ineq-for-D} and the fact that $D_m$ is independent of $L^{*}_{m+2}$, $\Lc_{m+2}$ and $\mathcal{F}_{r}$.
We also have
		\begin{align*}
		&e^{-2\alpha_1(m+4)}\Eb(e^{-\lambda L^*_{m+2}} \one_{\Lc_{m+2}} \mid \mathcal{F}_{r})
		\one_{\Upsilon(r,m+1)}\\
		\le\; &\Eb(e^{-\lambda L^*_{m+2}} \one_{\Lc_{m+2}} \one_{\Ec_{m+3}}\mid \mathcal{F}_{r})
		\one_{\Upsilon(r,m+1)}\\
		\le\; &\Eb(e^{-\lambda \wt L_{r+1}^{1}} \mid \mathcal{F}_{r})
		\one_{\Upsilon(r,m+1)}.
		\end{align*}
In the first inequality above, we use \eqref{eq:2BM_wedge} and the fact that $\Ec_{m+3}$ is independent of $\Lc_{m+2}$ and $L^{*}_{m+2}$ given $\Fc_r$. In the second inequality above, we use \eqref{eq:bar-ge-wt} and the fact that $\Upsilon(r,m+2)$ holds on the event $\Upsilon(r,m+1)$.
		Combining all the inequalities above, we get 
		\begin{equation}
		\Eb(e^{-\lambda \wt L_{r+1}^{1}} \one_{\tau_{m} \geq m^{2}e^{-m/2} } \mid \mathcal{F}_{r}) \one_{\Upsilon(r,m+1)}
		\le \eps_m\Eb(e^{-\lambda \wt L_{r+1}^{1}} \mid \mathcal{F}_{r})
		\one_{\Upsilon(r,m+1)}
		\end{equation}
		for $m\ge m_0$, where $\eps_m=e^{-am^2+\lambda C (m+2)+2\alpha_1(m+4)+2\alpha_2(m+6)}$ is a summable sequence. Similarly, by replacing $r$ with $r+\tau_{m+1}$, we see that for all $m\ge m_0$,
		\begin{equation}\label{eq:recursion-tau_m}
		\Eb(e^{-\lambda \wt L_{r+1}^{1}} \one_{\tau_{m} \leq r_m } \mid \mathcal{F}_{r+\tau_{m+1}}) 
		\ge (1-\eps_m) \Eb(e^{-\lambda \wt L_{r+1}^{1}} \mid \mathcal{F}_{r+\tau_{m+1}}) \one_{\tau_{m+1}\le r_{m+1}},
		\end{equation}
		where $$r_{m}=\sum_{l=m}^{\infty}l^2e^{-l/2}.$$ Let $m_1$ be the smallest integer such that $m_1\ge m_0, r_{m_1}<1/4$, and $\eps_l<1$ for all $l\ge m_1$. 
		Iterating \eqref{eq:recursion-tau_m} from $m=m_1$ to infinity using the towering property of conditional expectations, we obtain 
		\begin{equation}
		\Eb(e^{-\lambda \wt L_{r+1}^{1}} \one_{\tau_{m_1} \leq r_{m_1} } \mid \mathcal{F}_{r})
		\ge \prod_{m=m_1}^{\infty} (1-\eps_m) \Eb(e^{-\lambda \wt L_{r+1}^{1}} \mid \mathcal{F}_{r}),
		\end{equation}
		where we used the fact that if $\wt L_{r+1}^{1}<\infty$, then $\tau_m=0$ for all sufficiently large $m$. We conclude the lemma with $\delta_0=e^{-m_1}$, $\zeta=r+(\tau_{m_1}\wedge 1/4)$ and $C=\prod_{m=m_1}^{\infty} (1-\eps_m)$, where we know $C>0$ since $\eps_m$ is a summable sequence.
	\end{proof}

We are now ready to prove Lemma~\ref{lem:separation_lem_2} by using Lemma~\ref{lem:separation_lemma_1}.
		\begin{proof}[Proof of Lemma~\ref{lem:separation_lem_2}]
			Let $\zeta\in [r,r+1/4]$ be the stopping time and $\delta_0$ be the constant  in Lemma~\ref{lem:separation_lemma_1}. By Lemma~\ref{lem:separation_lemma_1}, it suffices to prove that 
			\begin{equation}\label{eq:G-delta}
			\Eb\big(e^{-\lambda \wt L^1_{r+1}} \one_{G^+(r+1)} \mid\Fc_{r}\big)\ge 
			C \Eb\big(e^{-\lambda \wt L^1_{r+1}} \one_{\Upsilon(\zeta, r+1, \delta_0)}\mid\Fc_{r}\big),
			\end{equation}
where $C=C(\delta_0,\lambda_0)>0$ (recall that $\lambda_0$ is the constant from the statement of Lemma~\ref{lem:separation_lem_2}). Throughout the proof, all the constants that we introduce depend only on $\delta_0$ and $\lambda_0$.

The idea of the proof for \eqref{eq:G-delta} is that, on the $\delta_0$-nice event $\Upsilon(\zeta, r+1, \delta_0)$, since the configuration is $\delta_0$-nice at scale $\zeta$, we have got enough space to let the future parts of the Brownian excursions (after hitting $\Cc_\zeta$) evolve in some well-chosen ``tubes'' before reaching $\Cc_{r+1}$, with a positive probability (depending on $\delta_0$). If we further condition the loop soup to be sufficiently small (which costs another positive probability depending on $\delta_0$), then the configuration in $\Ac_{0,r+1}$ will be very nice at the end. 
In practice, in order to gain independence when conditioning on such events, we will look at an additional intermediate layer $\Cc_{r+2/5}$ and replace $\wt L^1_{r+1}$ in \eqref{eq:G-delta} by the auxiliary extremal distance $L^*(\zeta,r+2/5, \delta_0/e)$ (see \eqref{eq:L*}) at an intermediate step.
\smallskip

	\begin{figure}[t]
	\centering
	\includegraphics[scale=0.55]{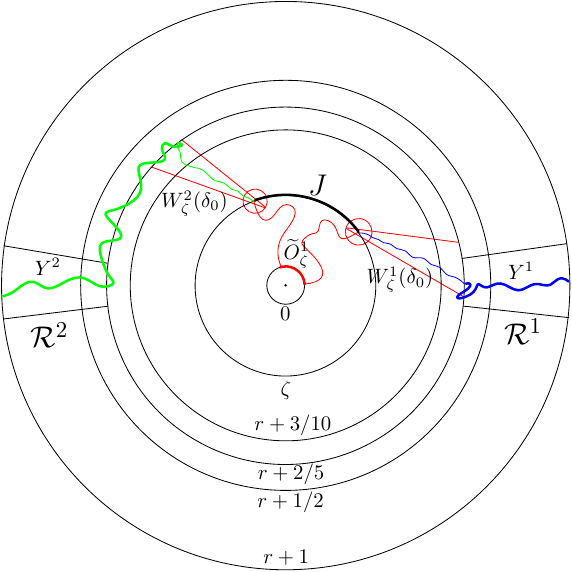}\;
	\includegraphics[scale=0.55]{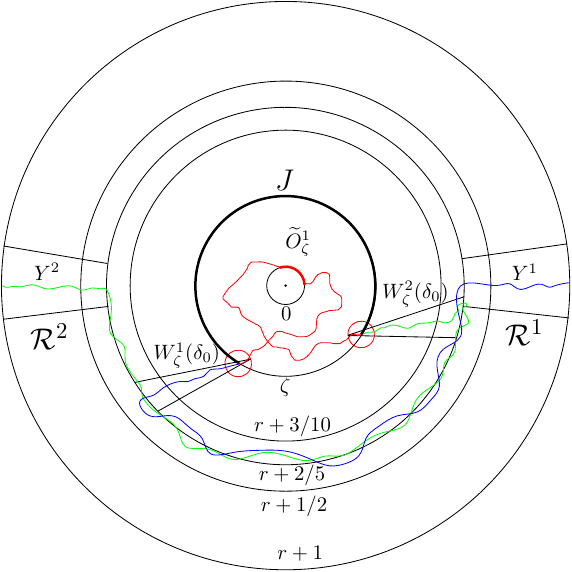}
	\caption{Proof of Lemma~\ref{lem:separation_lem_2}. \textbf{Left:} Illustration of the event $E_1\cap E_1'$. The six concentric circles from the inside to the outside are $\Cc_0$, $\Cc_{\zeta}$, $\Cc_{r+3/10}$, $\Cc_{r+2/5}$, $\Cc_{r+1/2}$ and $\Cc_{r+1}$. 
			The arc $J$ on $\Cc_{\zeta}$ is in bold black. The excursion $Y^1$ (resp. $Y^2$) from (the hitting time of) $\Cc_\zeta$ to $\Cc_{r+2/5}$ is in blue (resp. green), and from $\Cc_{r+2/5}$ to $\Cc_{r+1}$ is in bold blue (resp. green). The ball $\Bc^i_{\zeta}(\delta_0)$ and the wedge $W^i_{\zeta}(\delta_0)$ are in red. We let $Y^i$ first stay in the respective wedge and then reach $\Rc^i$ in the annulus $\Ac(r+3/10,r+1/2)$. \textbf{Right:} Illustration of the event $E_2\cap E_2'$.}
	\label{fig:sepa-1}
\end{figure}

			Let $J:=\partial\wt O_{\zeta}^1\cap\Cc_{\zeta}$. There are two possible orientations for $J$ and we need to treat them separately. Let $E_1$ (resp.\ $E_2$) be the event that $J$ is a counterclockwise (resp.\ clockwise) arc on $\Cc_{\zeta}$ from $Y^1(T^1_\zeta)$  to $Y^2(T^2_\zeta)$. See Figures~\ref{fig:sepa-1} for an illustration. 
			Note that 
			\begin{equation}\label{eq:deltaE1E2}
			 \Upsilon(\zeta, r+1, \delta_0)\subseteq E_1\cup E_2.
			\end{equation}

			We first focus on the case of $E_1$. The case of $E_2$ is similar and we will discuss it at the end of the proof. We aim to prove that
			\begin{equation}\label{eq:E-1}
			\Eb\big(e^{-\lambda \wt L^1_{r+1}} \one_{G^+(r+1)} \mid\Fc_{r}\big)\gtrsim 
			 \Eb\big(e^{-\lambda \wt L^1_{r+1}} \one_{E_1\cap\Upsilon(\zeta, r+1, \delta_0)}\mid\Fc_{r}\big).
			\end{equation}
			We define the following two opposite wedges in $\Ac(r+2/5,r+1)$:
			\begin{align*}
			&\Rc^1:=\{ z: |\arg(z)| \le 1/40 \}\cap\Ac(r+2/5,r+1) ,\\
			&\Rc^2:=\{ z: |\arg(z)-\pi| \le 1/40 \}\cap\Ac(r+2/5,r+1).
			\end{align*}
			We also define 
			$$E_1' := \{W^1_{\zeta}(\delta_0)\cap\Cc_{r+2/5}\cap\ol{\Rc^1}\neq\emptyset\}.$$ 
			Note that $E_1'$ only depends on the points $Y^1(T^1_\zeta)$ and $Y^2(T^2_\zeta)$, and by rotational invariance, we have
			\begin{equation}\label{eq:E'_1}
				\Pb( E_1' \mid E_1\cap\Upsilon(\zeta, r+1, \delta_0)) \ge C_1
			\end{equation}
			for some $C_1>0$. 
We further define the event $G$:
			\begin{itemize}
				\item For $i=1,2$, $Y^i[T^i_{r+2/5},T^i_{r+1}]\subseteq \Ac(r+3/10,r+1/2)\cup\Rc^i$.
				\item The distance between $Y^1[T^1_{r+2/5},T^1_{r+1}]$ and $Y^2[T^2_{r+2/5},T^2_{r+1}]$ is at least $e^r/1000$.
			\end{itemize}		
			Note that $G$ only depends on $Y^i[T^i_{r+2/5},T^i_{r+1}]$, $i=1,2$. Moreover, there exists a constant $C_0>0$ such that 
			\begin{align}\label{eq:G_lower_bound}
			\Pb \big(G \mid Y^1(T^1_{r+2/5}), Y^2(T^2_{r+2/5}) \big) \ge \, C_0 \quad \text{ on } \cap_{i=1}^2\big\{Y^i(T^i_{r+2/5})\in W^i_{\zeta}(\delta_0) \big\}.
			\end{align}	
Let $F$ be the event that all the clusters in $\Lambda_{0,r+1}$ have diameter at most $e^r/1000$. Then $\Pb(F)$ is equal to a universal constant. Recall Definitions~\ref{def:delta-nice},~\ref{def:exc-in-wed} and~\ref{def:loop_delta_small}, and define
			\[
			H:=E_1'\cap E_1 \cap \Upsilon(\zeta, r+2/5,\delta_0) \cap G\cap F\cap\Lc(\zeta,r+2/5, \delta_0/e) \cap\Ec(\zeta,r+2/5, \delta_0 e^{-2}).
			\]
Note that $G^+(r+1)$ holds on the event $H$.
By using a reasoning similar to the proof of Lemma~\ref{lem:concatenation}, on the event $H$,
			we have 
			\begin{equation}\label{eq:r+1tor+1/2}
			\wt L^1_{r+1}\le  L^*(\zeta,r+2/5, \delta_0/e)+C_2
			\end{equation}
			for some constant $C_2>0$.
	Therefore		 
			\begin{equation}
			\Eb\big(e^{-\lambda \wt L^1_{r+1}} \one_{G^+(r+1)} \mid\Fc_{r}\big)\ge 
			e^{-C_2\lambda_0} \Eb\big(e^{-\lambda L^*(\zeta,r+2/5, \delta_0/e)} \one_{H}\mid\Fc_{r}\big).
			\end{equation} 
			By the FKG inequality (Lemma~\ref{lem:FKG}) and the strong Markov Property of the Brownian motion,  we have
			\begin{align*}
			& \Eb\big(e^{-\lambda L^*(\zeta,r+2/5, \delta_0/e)} \one_{H}\mid\Fc_{r}\big) \ge \Pb(F) \times  \Eb\Big[ \Pb( \Ec(\zeta,r+2/5, \delta_0 e^{-2}) \cap G \mid \Fc_\zeta)\\
		&\qquad \qquad \qquad \qquad	\times\Eb\big(e^{-\lambda L^*(\zeta,r+2/5, \delta_0/e)} \one_{E_1'\cap E_1\cap\Upsilon(\zeta, r+1, \delta_0)\cap\Lc(\zeta,r+2/5, \delta_0/e)} \mid\Fc_{\zeta}\big) \mid \Fc_r \Big].
			\end{align*}
Note that $\Pb(\Ec(\zeta,r+2/5, \delta_0 e^{-2}) \mid \Fc_\zeta) \ge C_3$ for some constant $C_3>0$. Combined with \eqref{eq:G_lower_bound}, we get
\begin{align*}
\Pb(\Ec(\zeta,r+2/5, \delta_0 e^{-2})\cap G \mid \Fc_\zeta) \ge C_4
\end{align*}
for some constant $C_4>0$.
This leads to
\begin{align*}
\Eb\big(e^{-\lambda L^*(\zeta,r+2/5, \delta_0/e)} \one_{H}\mid\Fc_{r}\big) \ge C_5 \Eb\big(e^{-\lambda L^*(\zeta,r+2/5, \delta_0/e)} \one_{E_1\cap\Upsilon(\zeta, r+1, \delta_0)\cap\Lc(\zeta,r+2/5, \delta_0/e)}\mid\Fc_{r}\big)
\end{align*}
for some constant $C_5>0$, where we used rotational invariance to get rid of $E_1'$ (see \eqref{eq:E'_1}).
			Furthermore, by \eqref{eq:bar-le-wt} and noting that $\Upsilon(\zeta, r+1, \delta_0)$ implies $\Upsilon(\zeta,r+2/5,\delta_0)$, we have
			\[
			L^*(\zeta,r+2/5, \delta_0/e)\le \wt L^{1}_{r+2/5}+C_6 (2/5+\Delta_0) 
			\]  
on the event $\Upsilon(\zeta, r+1, \delta_0)\cap\Lc(\zeta,r+2/5, \delta_0/e)$ for some constant $C_6>0$, where $\Delta_0:=-\log \delta_0$.
			Noting that $\wt L^1_{r+2/5}\leq \wt L^1_{r+1}$ and using the FKG inequality (Lemma~\ref{lem:FKG}) again, we obtain 
			\begin{align*}
			&\Eb\big(e^{-\lambda L^*(\zeta,r+2/5, \delta_0/e)} \one_{E_1\cap\Upsilon(\zeta, r+1, \delta_0)\cap\Lc(\zeta,r+2/5, \delta_0/e)}\mid\Fc_{r}\big)\\
			\ge \;& e^{-\lambda_0C_6(2/5+\Delta_0)} \Pb(\Lc(\zeta,r+2/5, \delta_0/e)) 
			\Eb\big(e^{-\lambda \wt L^{1}_{r+1}} \one_{E_1\cap\Upsilon(\zeta, r+1, \delta_0)}\mid\Fc_{r}\big).
			\end{align*}
			By \eqref{eq:2loop_small}, we have $ \Pb(\Lc(\zeta,r+2/5, \delta_0/e))  \ge C_7$ for some constant $C_7>0$. Combining all the inequalities above, we obtain \eqref{eq:E-1}.
			
			On the event $E_2$, the component $\wt O^1_{\zeta}$ has on its boundary the clockwise arc on $\Cc_{\zeta}$ from $Y^1(T^1_\zeta)$  to $Y^2(T^2_\zeta)$. We can let the Brownian motions move like that in Figure~\ref{fig:sepa-1} (right) and repeat the argument with $E_2, E_2'$ in place of $E_1, E_1'$, where
\begin{align*}
E_2':=\{W^2_{\zeta}(\delta_0)\cap\Cc_{r+2/5}\cap\ol{\Rc^1}\neq\emptyset\}.
\end{align*}
This will allow us to deduce that
			\begin{equation}\label{eq:E-2}
			\Eb\big(e^{-\lambda \wt L^1_{r+1}} \one_{G^+(r+1)} \mid\Fc_{r}\big)\gtrsim 
			\Eb\big(e^{-\lambda \wt L^1_{r+1}} \one_{E_2\cap\Upsilon(\zeta, r+1, \delta_0)}\mid\Fc_{r}\big).
			\end{equation}
			Combining \eqref{eq:E-1} with \eqref{eq:E-2} and noting \eqref{eq:deltaE1E2}, we obtain \eqref{eq:G-delta}. This completes the proof.
		\end{proof}

	\section{Dimension of multiple points on the outer boundaries of outermost clusters in a Brownian loop soup}\label{sec:dim}
	
	In this section, we will focus on the proof of the following Proposition~\ref{prop:dim-upper}.
	Let $\Gamma_0$ be a Brownian loop soup in the unit disk with intensity $c\in(0,1]$. Let $\Sc_\mathrm{b}$ (resp.\ $\Dc_\mathrm{b}$, $\Tc_\mathrm{b}$) be the set of  simple (resp.\  double,  triple) points of $\Gamma_0$ which are on the outer boundaries of  outermost clusters in $\Gamma_0$. Recall from Theorem~\ref{thm:same_exponent} that the generalized disconnection exponents satisfy $\xi_c(k,0)=\xi_c(k)$ and their values are given by \eqref{eq:exponent}.
	
	\begin{proposition}\label{prop:dim-upper}
		For any $c\in(0,1]$, the following holds almost surely
		\begin{equation}\label{eq:up-bound}
		\dimh(\Sc_\mathrm{b})\le 2-\xi_c(2), \quad
		\dimh(\Dc_\mathrm{b})\le 2-\xi_c(4), \quad \Tc_\mathrm{b}=\emptyset,
		\end{equation}
		and the following holds with positive probability
		\begin{equation}\label{eq:lo-bound}
		\dimh(\Sc_\mathrm{b})\ge 2-\xi_c(2) , \quad
		\dimh(\Dc_\mathrm{b})\ge 2-\xi_c(4).
		\end{equation}
	\end{proposition}
	
	In order to prove Proposition~\ref{prop:dim-upper}, we need to find an adapted setting which allows us to use the generalized disconnection exponents and the up-to-constants estimates derived in the previous section.
	In Section~\ref{subsec:single_loop}, we will define such a setting  and state a result (Proposition~\ref{prop:single_loop}) for multiple points on the frontier of a single Brownian loop thrown into an independent Brownian loop soup.  In Sections~\ref{subsec:first_moment} and~\ref{subsec:second_moment}, we will prove the first and second moment estimates, which are the key ingredients for the proof of Proposition~\ref{prop:single_loop}. 
	Finally in Section~\ref{subsec:dim_out_bdy}, we will relate the result for a single Brownian loop to the initial setting of  Proposition~\ref{prop:dim-upper} and complete its proof.

	\subsection{Multiple points on a single Brownian loop}\label{subsec:single_loop}
	We will first consider multiple points on a single Brownian loop inside an independent loop soup.
	Let us now describe the  set-up. 
	If $D$ is a domain and $\partial D$ is smooth near $z\in \partial D$, the Brownian bubble measure $\mub_{D}(z)$ is just the boundary-to-boundary measure $\mu^D_{z,z}$ introduced in Section~\ref{subsec:facts} with a different normalizing constant. We refer the reader to Section 3.4 in \cite{MR2045953} for details.
	Let $\Gamma_0$ be a Brownian loop soup in the unit disk $\Ub$ with intensity $c\in (0,1]$.
	Let $S_0=[-1/16,1/16]^2$.
	We fix $r\in[1/4,1/2]$, $\theta\in[0,2\pi)$ and let $\gamma$ be a rooted Brownian loop sampled according to the Brownian bubble measure $\mub_{\Bc(0,r)}(r e^{i\theta})$ conditioned on $\{ \gamma\cap S_0\neq\emptyset \}$. 
	Let $\wt\gamma$ be the cluster of $\gamma\cup\Gamma_0$ which contains $\gamma$. 
	Let $\fr(\wt\gamma)$ be the set of points $z$ on the outer boundary of $\wt\gamma$ such that there exists a continuous path from any neighborhood of $z$ to $\Cc_0$ that does not intersect $\wt\gamma$ (i.e., any neighborhood of $z$ is not disconnected from infinity by $\wt\gamma$).
	For $k\in\Nb$, let $\Es_k(\gamma)=\Es_k(\gamma,\Gamma_0)$ be the set of points on $\fr(\wt\gamma)$ visited at least $k$ times by $\gamma$.
	
We aim to prove the following proposition.
	\begin{proposition}\label{prop:single_loop}
		For all $c\in(0,1]$, the followings hold almost surely
		\begin{equation}\label{eq:S-up}
		\dimh(\Es_1\cap S_0) \le 2-\xi_c(2), \quad
		\dimh(\Es_2\cap S_0) \le 2-\xi_c(4), \quad
		\Es_3 \cap S_0=\emptyset,
		\end{equation}
		and the followings hold with positive probability
		\begin{equation}\label{eq:S-lo}
		\dimh(\Es_1\cap S_0) \ge 2-\xi_c(2), \quad
		\dimh(\Es_2\cap S_0) \ge 2-\xi_c(4).
		\end{equation}
	\end{proposition}

	We follow the usual strategy for computing Hausdorff dimensions, and adopt almost the same setting as in  \cite{MR2644878}.
	Let $t_{\gamma}$ be the time length of $\gamma$.  For any open or closed sets $A_1,A_2,\cdots$,  we define the following stopping times for the rooted Brownian loop $\gamma$
	\begin{align*}
	\tau\left(A_{1}\right) &:=\inf \left\{t \geqslant 0: \gamma(t) \in A_{1}\right\} \\
	\tau\left(A_{1}, \ldots, A_{n}\right) &:=\inf \left\{t \geqslant \tau\left(A_{1}, \ldots, A_{n-1}\right): \gamma(t) \in A_{n}\right\}, \quad \text { for } n \geqslant 2,
	\end{align*}
	where, as usual, the infimum of the empty set is understood as infinity.
For all $n\ge3$, we divide $S_0$ into $2^{2(n-3)}$ non-overlapping dyadic compact sub-squares of side length $2^{-n}$, which will be called an $n$-square in the following.
	Let $j\ge 4$ and $n\ge j+4$. Suppose $S$ is an $n$-square centered at $v$ and write
	\begin{equation}\label{eq:Sj}
	S^{j}:=\Bc(v,2^{-j}-2^{-n}/\sqrt{2})
	\end{equation}
	for the large ball of radius $2^{-j}-2^{-n}/\sqrt{2}$ concentric with $S$.
	For any $n\ge j+4$ and any $n$-square $S$, let $V(S)$ be the event that
	\begin{equation}\label{eq:visits}
	\tau(S,\partial S^{j},S,\partial S^{j},\cdots,S)<t_{\gamma},
	\end{equation}
	where $S$ appears $k$ times above. In particular, if $k=1$, the above condition reduces to $\tau(S)<\infty$ regardless of $j$. By standard estimates on the hitting probability of Brownian motion (see \cite{MR2677157}), we have 
	\begin{equation}\label{eq:hitting-BM}
	\Pb(V(S))\asymp n^{-k},
	\end{equation}
	where each visit contributes to the probability $n^{-1}$, and the implied constants can be chosen to depend only on $j$. Furthermore, we say an $n$-square $S$ is a $(k,n)$-square if $V(S)$ occurs and $S$ is not disconnected from infinity by $\wt\gamma$ (the cluster of $\gamma\cup\Gamma_0$ which contains $\gamma$). Denote by $\Sfr_{k,n}=\Sfr_{k,n}(j)$ the set of all $(k,n)$-squares. 
	The proof of Proposition~\ref{prop:single_loop} relies on the first and second moment estimates stated in the following lemma.
	
	\begin{lemma}\label{prop:(k,n)}
		For all $k\ge 1$ and $j\ge 4$, there exist constants $C_1,C_2,C_3>0$ (depending only on $r,k,j$ and $c$) such that for any $n\ge m\ge j+4$ and any $n$-square $S$, we have
		\begin{equation}\label{eq:1-first-moment}
		C_1 2^{-n \xi_c(2k)}\le \Pb(S \in\Sfr_{k,n} )\le C_2 2^{-n \xi_c(2k)},
		\end{equation}
		and for any pair of $n$-squares $S,T$ with distance in $[2^{-m-1},2^{-m}]$, we have
		\begin{equation}\label{eq:1-second-moment}
		\Pb(S,T \in\Sfr_{k,n} )\le C_3 2^{-2n \xi_c(2k)}2^{m\xi_c(2k)}.
		\end{equation}
	\end{lemma}
	Note that when $k=1$, the definition of $(1,n)$-squares does not depend on $j$ so one can take $j=4$ in the statement of Lemma~\ref{prop:(k,n)}. 
	
	Let us immediately prove Proposition~\ref{prop:single_loop} using Lemma~\ref{prop:(k,n)}.  We will complete the proof of Lemma~\ref{prop:(k,n)} in Sections~\ref{subsec:first_moment} and~\ref{subsec:second_moment}.
	
	\begin{proof}[Proof of Proposition~\ref{prop:single_loop}, assuming Lemma~\ref{prop:(k,n)}]
		For any $k\ge 1, j\ge 4$, let $\Es_{k,j}$ be the set of points $x\in\Es_k$ such that 
		\begin{equation}\label{eq:k-times}
		\tau(x,\partial \Bc(x,2^{-j}),\cdots,x)<t_{\gamma},
		\end{equation}
		where $x$ appears $k$ times above. Then $\Es_k=\cup_{j\ge 1}\Es_{k,j}$. Note that when $k=1$, \eqref{eq:k-times} reduces to $\tau(x)<\infty$ regardless of $j$. 
		Let us first prove that the following holds a.s.\ for all $k\ge 1$
		\begin{align}\label{eq:Kj}
		\Es_{k,j}\cap S_0=\bigcap_{n=j+4}^{\infty}\bigcup_{S\in\Sfr_{k,n}}S \text{ for all } j\ge 4.
		\end{align}
		Note that the definition of $\Es_{1,j}$ does not depend on $j$, and $\Es_{1,j} = \Es_1$ for all $j\ge 4$.

		Let us prove \eqref{eq:Kj}. Note that
		$
		S^j\subseteq \Bc(x,2^{-j}).
		$
		Thus, if $x\in \Es_{k,j}\cap S_0$, then every $n$-square containing $x$ is a $(k,n)$-square, hence $x$ is in the set on the right hand side in \eqref{eq:Kj}. Conversely, suppose that $x$ is in the set on the right hand side in \eqref{eq:Kj}. Let us show that \eqref{eq:k-times} holds. Let us prove it for $k=2$, because the proof for the general case is similar. There exists a sequence of $(2,n)$-squares $(S_n )_{n\ge j+4}$ such that for all $n\ge j+4$ we have
		\begin{equation}\label{eq:S_n-cond}
		x\in S_n, S_{n+1}\subseteq S_n \text{ and } S^{j}_{n}\subseteq S^{j}_{n+1}.
		\end{equation}
		By \eqref{eq:visits}, we have 
		\begin{equation}
		0<\tau(S_n)<\tau(S_n,\partial S^{j}_n)<\tau(S_n,\partial S^{j}_n,S_n)<t_{\gamma}.
		\end{equation}
		By \eqref{eq:S_n-cond} and monotone convergence theorem, we know that the stopping times $\tau(S_n),\tau(S_n,\partial S^{j}_n)$ and $\tau(S_n,\partial S^{j}_n,S_n)$ are all increasing in $n$ and they converge to random times $t_1, t_2$ and $t_3$, respectively. By continuity of $\gamma$, we have $\gamma(t_1)=\gamma(t_3)=x$, $\gamma(t_2)\in\partial \Bc(x,2^{-j})$ and $0<t_1<t_2<t_3<t_{\gamma}$. Combining it with the fact that $x$ is not disconnected from infinity by a path not intersecting $\wt\gamma$, we conclude $x\in\Es_{k,j}\cap S_0$.

		By Lemma~\ref{prop:(k,n)} and \eqref{eq:Kj}, using the usual methods for computing Hausdorff dimensions (see \cite{MR2604525}), we have  that for $k=1,2$
		$$\dimh(\Es_{k,j}\cap S_0) \le 2-\xi_c(2k) \text{ a.s.} \quad \text{and}\quad \Pb\left(\dimh(\Es_{k,j}\cap S_0) \ge 2-\xi_c(2k) \right)>0.$$
		Since $\Es_k=\cup_{j\ge 1}\Es_{k,j}$ and Hausdorff dimension is stable under countable union, the following also holds for $k=1,2$
		$$\dimh(\Es_{k}\cap S_0) \le 2-\xi_c(2k) \text{ a.s.} \quad \text{and}\quad \Pb\left(\dimh(\Es_{k}\cap S_0) \ge 2-\xi_c(2k) \right)>0.$$
		
		For all $c\in (0,1]$ and $k\ge 3$, we have $\xi_c(2k) >2$ by \eqref{eq:exponent}. We deduce from \eqref{eq:1-first-moment} that 
		$$\Pb(|\Sfr_{k, n}| \ge 1)\le \Eb(|\Sfr_{k, n}|) \le C_2 2^{-n (\xi_c(2k)-2)}.$$
		By~\eqref{eq:Kj}, the following holds for all $j\ge 4$ and $n\ge j+4$
		\begin{align*}
		\Pb(\Es_{k,j} \cap S_0 \not=\emptyset) \le C_2 2^{-n (\xi_c(2k)-2)}.
		\end{align*}
		Letting $n\to \infty$, we have that $\Es_{k,j} \cap S_0 =\emptyset$ a.s. Therefore $\Es_k \cap S_0=\emptyset$ a.s.
	\end{proof}

	\subsection{First moment estimate}\label{subsec:first_moment}
	The goal of this subsection is to establish the first moment estimate  \eqref{eq:1-first-moment}.
	Throughout, fix $k\ge 1$, $j\ge 4$ and $n \ge j+4$.   
	
	Intuitively, standing on a $(k,n)$-square, one can see $2k$ crossings of Brownian paths starting from it reaching a macroscopic distance without disconnecting it. 
	In the following, we decompose the Brownian loop $\gamma$ to extract these $2k$ crossings, see Figure~\ref{fig:one-square}.
	\begin{figure}
		\centering
		\includegraphics[scale=.7]{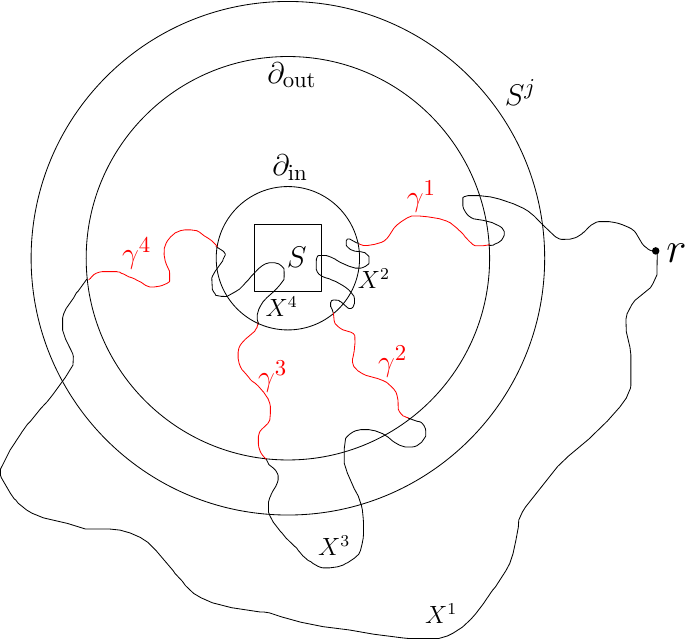}
		\caption{Decomposition of case $k=2$. The square is $S$ and the three concentric balls are $\wh S, \check S^j$ and $S^j$ respectively. $\Ac_S$ is the annulus $\check S^j\setminus \wh S$. The boundaries of $\Ac_S$ are $\partial_{\rm in}=\partial \wh S$ and $\partial_{\rm out}=\partial \check S^j$, respectively. The decomposed crossings $\gamma^i$ between $\partial_{\rm in}$ and $\partial_{\rm out}$ are in red and the segments in between are Brownian links $X^i$ which are in black.}
		\label{fig:one-square}
	\end{figure}
	Suppose $V(S)$ holds (see \eqref{eq:visits}). 
   Let $\wh S:=\Bc(S,2^{-n})$ be the ball of radius $2^{-n}$ concentric with $S$. Let $\check S^j:=\frac12 S^j$ be the ball concentric with $S^j$ but with half of its radius. Then, we have the following relations 
    \[
    S\subseteq \wh S \subseteq \check S^j \subseteq S^j.
    \] 
	Write $\Ac_S$ for the annulus $\check S^j\setminus \wh S$ and $\partial_{\rm in}$ and $\partial_{\rm out}$ for its inner and outer boundaries, respectively. We define
	$
	u_1=\tau(S^j)
	$
	and for all $1\le i\le k$,
	\begin{align*}
	v_i=\inf\{ t>u_i: \gamma(t)\in S \},&\qquad\qquad
	u_{i+1}=\inf\{ t>v_i: \gamma(t)\in \partial S^j \}, \\
	s_{2i-1}=\sup\{ t<v_{2i-1} : \gamma(t)\in \partial_{\rm out}\}, &\qquad\qquad
	t_{2i-1}=\inf\{ t>s_{2i-1} : \gamma(t)\in \partial_{\rm in}\}, \\
	s_{2i}=\sup\{ t<u_{2i+1} : \gamma(t)\in \partial_{\rm in}\}, &\qquad\qquad
	t_{2i}=\inf\{ t>s_{2i} : \gamma(t)\in \partial_{\rm out}\}.
	\end{align*} 
    We also define $\gamma^1=\gamma([s_1, t_1])$ and for all $1\le i\le 2k-1$,
    \begin{align*}
    	X^{i+1}&=\gamma([t_i, s_{i+1}]), \\
    	\gamma^{i+1}&=\gamma([s_{i+1}, t_{i+1}]).
    \end{align*} 
Finally, let $X^1$ be the concatenation of $\gamma([t_{2k}, t_\gamma])$ and $\gamma([0,s_1])$.
This decomposition induces $2k$ crossings $\gamma^i$ across $\Ac_S$, and $2k$ links $X^i$ which connect the crossings back into the loop $\gamma$. 
	For each $i$, denote by $\wt\gamma^i$ the cluster of $\gamma^i\cup\Lambda_{\Ac_S}$ that contains $\gamma^i$.
	
	However, the crossings $\gamma^i$'s are not exactly the same as excursions we considered in previous sections (see \eqref{eq:Y}), and they are not independent of each other. Nevertheless, in what follows, we are going to show that in an ``up-to-constants'' sense, it is still  ``almost'' the case, and then estimates similar to Theorem~\ref{thm:up-to-constants-ex-k-sep} follows naturally.
	
	\begin{lemma}\label{lem:one_loop_decomp}
		For given $k$ and $S$, on the event $V(S)$, the followings hold.
		\begin{enumerate}[(1)]
			\item\label{it:crossing} Let $(Y^1, \ldots, Y^{2k})$ be a family of $2k$ independent Brownian excursions across the annulus $\Ac_S$. Then the law of the family of crossings $(\gamma^1, \ldots, \gamma^{2k})$ has a density w.r.t. that of $(Y^1, \ldots, Y^{2k})$ which is bounded away from $0$ and $\infty$ (the bounds only depend on $j,k$).
			
			\item\label{it:br_exc} Conditionally on the endpoints $\gamma(s_i)$ and $\gamma(t_i)$ for $1\le i\le 2k$, the Brownian links $X^i$ are conditionally independent of each other and further independent of the crossings $\gamma^i$'s. Moreover, if $i$ is odd (resp.\ even), then $X^i$ is distributed as a Brownian link between its endpoints in $\Bc(0,r)$ conditioned on first hitting $\partial S^j$ (resp.\ $S$) and then going to the end $\gamma(s_i)$ in $S^c$ (resp.\ $S^j$). 
		\end{enumerate}
	\end{lemma}
	\begin{proof}
	    The lemma follows by using the strong Markov property of the Brownian motion and Lemma~\ref{lem:leu-SameScale} (also the outside-to-inside version in the paragraph following this lemma) repetitively.
	\end{proof}
	
	For $\alpha\in (0,1/2)$, we say a set of points on $\Cc_s$ is $\alpha$-isolated if the minimum distance between all pairs of these points is greater than $2\alpha e^s$.
	
	\begin{lemma}\label{lem:odd-even}
		Let $S$ be an $n$-square. Suppose $V(S)$ occurs and the endpoints of even and odd links are $\alpha$-isolated on $\partial_{\rm in}$ and $\partial_{\rm out}$ respectively. Let $w_i$ and $w'_i$ be the landing wedges of scale $\alpha$ associated with $2k$ crossings $\gamma^i$'s on $\partial_{\rm in}$ and $\partial_{\rm out}$ respectively (see Definition~\ref{def:sep}). Recall $\Ac_S=\check S^j\setminus \wh S$. Let $E$ be the following event (with the convention $w'_0=w'_{2k}$) 
		\begin{align}\label{eq:loc-2}
		X^i\cap\Ac_S\subseteq w_{i-1}\cup w_{i} \text{ for }  i \text{ even}, \quad X^i\cap\Ac_S\subseteq w'_{i-1}\cup w'_{i} \text{ for }  i \text{ odd}.
		\end{align}
		Then there exists a constant $C=C(\alpha,j,r,k)>0$ such that conditioned on $\alpha$-isolated endpoints of $X^i$'s,
		\begin{equation}\label{eq:E-even}
		\Pb(E\mid V(S))\ge C n^{-k}.
		\end{equation}  
	\end{lemma}
	
	\begin{proof}
		This lemma follows from Lemmas~\ref{lem:localization}, ~\ref{lem:outer-bridge} and~\ref{lem:one_loop_decomp}. In fact, every even link corresponds to the estimate \eqref{eq:even-br} and every odd link corresponds to the estimate \eqref{eq:odd-br}. Only even links will contribute an additional factor $n^{-1}$ and there are $k$ of them.
	\end{proof}

	\begin{proof}[Proof of \eqref{eq:1-first-moment}  in Lemma~\ref{prop:(k,n)}] Let us now prove the first moment estimate.
		Throughout the proof, all the constants can depend on $c,\alpha,j,r,k$ but they are independent of $n$.
		Fix $\alpha\in(0,\alpha_0]$.  
		Recall Definition~\ref{def:sep} for the notion of $\alpha$-sep configuration. In the proof the configuration refers to the collection of crossings and the loop soup $(\{\gamma^i\}_{1\le i\le 2k}, \Lambda_0)$.
		For simplicity, in the following, we say that $A$ does not disconnect $B$ if $A$ does not disconnect $B$ from infinity. 
		Define the following events: 
		\begin{align*}
		E_1&:=\{ \wt\gamma^1 \cup \cdots \cup \wt\gamma^{2k} \text{ does not disconnect } S \},\\ %\label{eq:E-1}\\
		E_2&:=E_1\cap \{ \text{the configuration is $\alpha$-sep} \},\\
		E_3&:=\{ X^{2i} \text{ does not disconnect $S$ for all } 1\le i\le k\}.
		\end{align*}
		
		Applying \eqref{it:crossing} of Lemma~\ref{lem:one_loop_decomp} and Theorem~\ref{thm:up-to-constants-ex-k-sep} (note that $\xi_c(k,0)=\xi_c(k)$), we have
		\begin{align}
			\label{eq:P-E-1}
			&\Pb(E_1 \mid V(S))\asymp n^{2k} 2^{-n\xi_c(2k)}, \\
			\label{eq:P-E'-2}
			&\Pb(E_2 \mid V(S))\asymp n^{2k} 2^{-n\xi_c(2k)}.
		\end{align}
		By strong Markov property of Brownian motion, \eqref{it:br_exc} of Lemma~\ref{lem:one_loop_decomp} and Lemma~\ref{lem:D1r}, we have
		\begin{equation}\label{eq:P-E-2}
		\Pb(E_3 \mid V(S)\cap E_1)\lesssim n^{-k}
		\end{equation}
		To get the upper bound in \eqref{eq:1-first-moment}, note that $\{S \in \Sfr_{k,n} \}$ is contained in $V(S)\cap E_1\cap E_3$. 
		By \eqref{eq:P-E-1}, \eqref{eq:P-E-2} and \eqref{eq:hitting-BM}, we have 
		\begin{equation}\label{eq:k-upper-bound}
		\Pb(S \in \Sfr_{k,n} )\le \Pb(V(S)\cap E_1\cap E_3) \lesssim 2^{-n \xi_c(2k)}.
		\end{equation}
	
		To get the lower bound in \eqref{eq:1-first-moment}, note that 
		\begin{equation}\label{eq:add-restriction}
		\Pb(S \in \Sfr_{k,n} )\gtrsim \Pb(V(S)\cap E_2\cap E \cap E_4),
		\end{equation}
		where $E$ is the event defined in Lemma~\ref{lem:odd-even}, and $E_4$ is the event that the followings hold
		\begin{itemize}
			\item all the clusters of $\Lambda_0$ intersecting $\partial_{\rm out}$ have diameter smaller than $\alpha2^{-j}/100$; \item all the clusters of $\Lambda_0$ intersecting $\partial_{\rm in}$ have diameter smaller than $\alpha2^{-n}/100$;
			\item there is no cluster in $\Lambda_0$ which encircles $S^j$. 
		\end{itemize}
		This is because on the good event $V(S)\cap E_2\cap E \cap E_4$, to make $\{ S \in \Sfr_{k,n} \}$ happen, we only need to impose mild spatial conditions on the links outside $\Ac_S$ such that they stay in well-chosen ``tubes'' and attaching them to the configuration will not cause disconnection. Such further conditions only cost a constant probability which only depends on $\alpha$ and $k$. Furthermore, note that $E_4$ and $V(S)\cap E_2\cap E$ are decreasing events for the loop soup, so by Lemma~\ref{lem:FKG}, we have
		\begin{equation}\label{eq:E4}
		\Pb(V(S)\cap E_2\cap E \cap E_4) \ge \Pb(E_4) \Pb(V(S) \cap E_2 \cap E) 
		\gtrsim \Pb(V(S) \cap E_2 \cap E),
		\end{equation}
	where we used Lemmas~\ref{lem:FKG} and~\ref{lem:cluster-thin} to get $\Pb(E_4)\ge C$ for some constant $C>0$.
        By Lemma~\ref{lem:odd-even},
        \begin{equation}\label{eq:oe}
        	 \Pb(E\mid V(S) \cap E_2) \gtrsim n^{-k}.
        \end{equation}
         Combining estimates \eqref{eq:add-restriction}, \eqref{eq:E4} and \eqref{eq:oe}, we have 
         \[
         \Pb(S\in \Sfr_{k,n} )\gtrsim \Pb(V(S) \cap E_2) \, n^{-k}.
         \] 
         Finally, plugging estimates \eqref{eq:P-E'-2} and \eqref{eq:hitting-BM} into the  RHS above, we obtain 
		\begin{equation*}%\label{eq:k-lower-bound}
		\Pb(S\in \Sfr_{k,n} )\gtrsim 2^{-n \xi_c(2k)}.
		\end{equation*}
		This completes the proof of \eqref{eq:1-first-moment}.
	\end{proof}

	\subsection{Second moment estimate}\label{subsec:second_moment}
	Let us now address the second-moment estimates \eqref{eq:1-second-moment}. Throughout, fix $k\ge 1$, $j\ge 4$ and $n \ge m\ge j+4$.   
	Let $S$ and $T$ be  two $n$-squares with distance in $[2^{-m-1}, 2^{-m}]$. 
	We will work on the event $V(S) \cap V(T)$ where $\gamma$ visits $k$ times both $S$ and $T$.
	We decompose $\gamma$ by treating the following two cases separately.
	
	\smallskip
	\noindent\textbf{Case 1.}
	Suppose $n\ge m+4$. By considering the annuli around $S,T$ and $S\cup T$ respectively, we can extract $3$ disjoint families of $2k$ crossings. This decomposition is similar to the one we did for the first-moment estimate (see above Lemma~\ref{lem:one_loop_decomp}).
	\begin{enumerate}[(i)]
		\item  Let $( \gamma^i(S) )_{1\le i\le 2k}$ be the $2k$ crossings induced by $\gamma$ in this annulus $\Bc(S,2^{-m-3})\setminus S$. For each $i$, let $\wt\gamma^i(S)$ be the union of $\gamma^i(S)$ and the clusters of the loop soup $\Lambda_0$ restricted on this annulus that intersect it.
		\item For the annulus $\Bc(T,2^{-m-3})\setminus T$, we define $(\gamma^i(T) )_{1\le i\le 2k}$ and $(\wt\gamma^i(T) )_{1\le i\le 2k}$ analogously.
		\item\label{it:ST} Let $z$ be the middle point between the centers of $S$ and $T$. For the annulus $\Bc(z,2^{-j-1})\backslash$ $\Bc(z,2^{-m+2})$, we define $( \gamma^i(S,T) )_{1\le i\le 2k}$ and $( \wt\gamma^i(S,T) )_{1\le i\le 2k}$ analogously. 
	\end{enumerate}  
	If we cut off the above $6k$ crossings from $\gamma$,  the remaining part of $\gamma$ consists of $3$ families of $2k$ Brownian links
	$( X^i(S) )_{1\le i\le 2k}$, $( X^i(T) )_{1\le i\le 2k}$ and $( X^i(S,T) )_{1\le i\le 2k}$. Each link is labeled according to the crossing following it. 
	
	\smallskip
	\noindent\textbf{Case 2.} 
	Suppose $n<m+4$. Let $( \gamma^i(S,T) )_{1\le i\le 2k}$, $(\wt\gamma^i(S,T) )_{1\le i\le 2k}$ be defined as in \eqref{it:ST} of Case 1. 
	The complement of the $2k$ crossings in $\gamma$ is composed of $2k$ links $( \wh X^i(S,T) )_{1\le i\le 2k}$.

	\begin{proof}[Proof of \eqref{eq:1-second-moment} in Lemma~\ref{prop:(k,n)}]
		The proof is very similar to that of \eqref{eq:k-upper-bound}, so we will omit some details. We condition on $V(S)\cap V(T)$ and consider the two cases defined above separately.
		
		\smallskip		
		\noindent\textbf{Case 1.} Suppose $n\ge m+4$.
		Let $F_1$ be the event that the following three points hold
		\begin{itemize}
			\item $ \wt\gamma^1(S) \cup \cdots \cup \wt\gamma^{2k}(S)$ does not disconnect $S$;
			\item  $ \wt\gamma^1(T) \cup \cdots \cup \wt\gamma^{2k}(T)$  does not disconnect $T$;
			\item  $ \wt\gamma^1(S, T) \cup \cdots \cup \wt\gamma^{2k}(S, T)$  does not disconnect $\Bc(z,2^{-m+2})$.
		\end{itemize}
		By Lemma~\ref{lem:one_loop_decomp}, the three families $(\wt\gamma^i(S) )_{1\le i\le 2k}$, $(\wt\gamma^i(T) )_{1\le i\le 2k}$, $(\wt\gamma^i(S,T) )_{1\le i\le 2k}$ are defined using crossings that are mutually absolutely continuous w.r.t. independent Brownian excursions, and using loops in $\Lambda_0$ restricted to three disjoint annuli. 
		By Theorem~\ref{thm:up-to-constants-ex-k-sep}, we have
		\begin{equation}\label{eq:second-1}
		\Pb(F_1 \mid V(S) \cap V(T) )\asymp (n-m)^{4k} m^{2k} 2^{-2(n-m)  \xi_c(2k)}  2^{-m\xi_c(2k)}.
		\end{equation}
		Let $F_2$ be the event that the following three points hold
		\begin{itemize}
			\item $ X^{2i}(S)$ does not disconnect $S$ for all $1\le i\le k$;
			\item $X^{2i}(T)$ does not disconnect $T$ for all $1\le i\le k$;
			\item $X^{2i}(S,T)$ does not disconnect $\Bc(z,2^{-m+2})$ for all $1\le i\le k$.
		\end{itemize}
		Then, by Lemma~\ref{lem:D1r}, we have
		\begin{equation}\label{eq:second-2}
		\Pb(F_2 \mid V(S) \cap V(T)\cap F_1)\lesssim (n-m)^{-2k} m^{-k}
		\end{equation}
		By repetitively using \eqref{eq:hitting-BM}, we get 
		\begin{equation}\label{eq:second-3}
		\Pb(V(S)\cap V(T))\asymp m^{-k} (n-m)^{-2k}.
		\end{equation}
		Combining \eqref{eq:second-1}, \eqref{eq:second-2} and \eqref{eq:second-3}, we have 
		\begin{equation}
		\Pb(S,T \in \Sfr_{k,n})
		\le \Pb(V(S)\cap V(T)\cap F_1\cap F_2)
		\lesssim 2^{-2(n-m)  \xi_c(2k)}  2^{-m\xi_c(2k)}.
		\end{equation} 
		
		\smallskip
		\noindent\textbf{Case 2.} 
		Suppose $n<m+4$. Let $G_1$ be the event $ \wt\gamma^1(S,T) \cup \cdots \cup \wt\gamma^{2k}(S,T)$ does not disconnect $\Bc(z,2^{-m+2})$. Let $G_2$ be the event that $\wh X^{2i}(S,T) $ does not disconnect $\Bc(z,2^{-m+2})$ for all $1\le i \le k$. In the same way, we have 
		\begin{equation*}
		\Pb(G_1 \mid V(S) \cap V(T))\asymp m^{2k}2^{-m\xi_c(2k )}, \quad  \Pb(G_2 \mid V(S) \cap V(T)\cap G_1)\lesssim m^{-k}, 
		\end{equation*}
		and 
		\begin{equation*}
		\Pb(V(S)\cap V(T))\asymp m^{-k}.
		\end{equation*}
		It follows that
		\begin{equation*}
		\Pb(S,T \in \Sfr_{k,n})
		\le \Pb(V(S)\cap V(T)\cap G_1\cap G_2)
		\lesssim 2^{m\xi_c(2k)}.
		\end{equation*}
		Since $n-m\in \{0,1,2,3\}$, this implies \eqref{eq:1-second-moment} in this case.
	\end{proof}

	\subsection{Proof of Proposition~\ref{prop:dim-upper}}\label{subsec:dim_out_bdy}
	We will now prove  Proposition~\ref{prop:dim-upper}, using the result (Proposition~\ref{prop:single_loop}) on dimensions of multiple points on the frontier of a single loop inside a Brownian loop soup. 
	
	The main idea is that, for any loop $\eta$ in a Brownian  loop soup $\Gamma_0$, the law of $(\eta, \Gamma_0\setminus \{\eta\})$ is mutually absolutely continuous with respect to the law of $(\gamma, \Gamma_0)$ where $\gamma$ is a Brownian loop sampled independently from $\Gamma_0$. Then Proposition~\ref{prop:single_loop} allows one to deduce the dimension of multiple points on the frontier of $\eta$ inside $\Gamma_0$. We then use the fact that there are a.s.\ countably many loops in a Brownian loop soup and Hausdorff dimension is stable under countable union. 
	
	However, we need to be careful that visits to a $k$-tuple point inside a loop soup for $k\ge 2$ can be made by different loops.
	In this respect, it is simpler to establish the dimension lower bound \eqref{eq:lo-bound},  because we only need to consider points which are visited $k$ times by the same loop, since they form a subset of all $k$-tuple points. Let us prove the lower bound first.
	\begin{proof}[Proof of \eqref{eq:lo-bound} in Proposition~\ref{prop:dim-upper}]
		Let $\mu_\Ub$ be the (infinite) measure on Brownian loops in the unit disk $\Ub$, introduced in Section~\ref{subsec:BLS}.
		For each $\eps>0$, let $\mu_\eps$ denote the measure $\mu_\Ub$ restricted to loops with diameter at least $\eps$. 
		Let $\Gamma^\eps$ be the set of loops in $\Gamma_0$ with diameter at least $\eps$.  
		We can sample $\Gamma^\eps$ by first sampling  a Poisson random variable $\Nc_\eps$ with parameter $c |\mu_\eps|$, and then sampling the loops $\gamma_i$ for $1\le i \le \Nc_\eps$ independently according to $\mu_\eps$.
		Conditioned on $\{ \Nc_\eps\ge 1 \}$ (which occurs with positive probability), the law of 
		$(\gamma_1, \Gamma_0\setminus\{\gamma_1\})$
		is mutually absolutely continuous with respect to the law of $(\xi, \Gamma_0)$, where $\xi$ is sampled independently according to $\mu_\eps$. 
		
		The measure $\mu_\Ub$ admits the following decomposition according to the point on $\xi$ which is farthest to the origin (the proof is very similar to \cite[Proposition 8]{MR2045953})
		\begin{align}\label{eq:mu_decomp}
		\mu_\Ub=\frac{1}{\pi} \int_{0}^{1} \int_0^{2\pi} \mub_{\Bc(0, r)} (r e^{i\theta}) d\theta r dr.
		\end{align}
		If $\xi$ is sampled according to $\mu_\eps$, then with probability at least $C>0$ (where $C$ does not depend on $\eps$), we have $\max\{|z| : z\in \xi\}\in[1/4, 1/2]$ and $\xi\cap S_0\neq\emptyset$. On this event, the conditional law of $\xi$ is given by
		\begin{align*}
		\frac{1}{\pi} \int_{1/4}^{1/2} \int_0^{2\pi} \mub_{\Bc(0, r)} (r e^{i\theta})[\cdot\mid \xi\cap S_0\neq\emptyset] d\theta r dr.
		\end{align*}
		Then Proposition~\ref{prop:single_loop} implies that with positive probability, the dimension of simple (resp.\ double) points on the {frontier} of $\xi$ inside $\Gamma_0 \cup \{\xi\}$ is at least $2-\xi_c(2)$ (resp.\ $2-\xi_c(4)$).
		By the previous paragraph, we further have that the dimension of simple (resp.\ double) points on the frontier of $\gamma_1$ inside $\Gamma_0$ is at least $2-\xi_c(2)$ (resp.\ $2-\xi_c(4)$).
		Since the latter set is a subset of $\Sc_\mathrm{b}$ (resp.\ $\Dc_\mathrm{b}$), this proves  \eqref{eq:lo-bound}.
	\end{proof}
	
	Let us now prove the remaining statements of Proposition~\ref{prop:dim-upper}. The proof of the dimension upper bound is slightly more complicated, especially for double and triple points, because they can be visited by more than one loop. 
	
	\begin{proof}[End of proof of Proposition~\ref{prop:dim-upper}]
		It suffices to prove the dimension upper bound~\eqref{eq:up-bound}. We need to consider, for simple, double and triple points, all the possible configurations according to the number of different loops that visit these points. For the sake of brevity, as well as to keep notation simple, we will only give a detailed proof for the case of double points that are visited by two different loops. The proofs of the other cases are similar (or even simpler).
		
		Let $\mu_\Ub$, $\mu_\eps$, $\Gamma^\eps$, $\Nc_\eps$ and $(\gamma_i)_{1\le i \le \Nc_\eps}$ be defined as in the previous proof of \eqref{eq:lo-bound}. Conditionally on $\{\Nc_\eps\ge 2\}$  (which occurs with positive probability), the law of $(\gamma_1, \gamma_2, \Gamma_0\setminus\{\gamma_1, \gamma_2\})$
		is mutually absolutely continuous with respect to the law of $(\xi_1, \xi_2, \Gamma_0)$, where $\xi_1$ and $\xi_2$ are sampled independently according to $\mu_\eps$. Let $\Dc_0$ be the set of points which are both on $\xi_1 \cap \xi_2$ and on the outer boundary of the cluster in $\Gamma_0\cup \{\xi_1, \xi_2\}$ containing $\xi_1$ and $\xi_2$. If $\xi_1\cap \xi_2=\emptyset$ or if the cluster containing $\xi_1, \xi_2$ is not an outermost cluster in $\Gamma_0\cup \{\xi_1, \xi_2\}$, let $\Dc_0=\emptyset$.
		It suffices to prove that 
		\begin{equation}\label{eq:2-loops}
		\dimh(\Dc_0 \cap \Bc(0, 1-\eps) )\le 2-\xi_c(4) \, \text{ a.s. for all } c\in(0,1].
		\end{equation}
		Indeed, thanks to the mutual absolute continuity, this will imply that \eqref{eq:2-loops} also hold for $\Dc_1$ in the place of $\Dc_0$, where $\Dc_1$ is the set of double points on the outer boundary of the outermost cluster in $\Gamma_0$ containing $\gamma_1$ and $\gamma_2$ (when there exists such a cluster) visited by both $\gamma_1$ and $\gamma_2$. Note that this is true for any $\gamma_1, \gamma_2 \in \Gamma^\eps$ and for any $\eps>0$. Since there are countably many loops in the loop soup $\Gamma_0$, this implies that the  set of double points on the outer boundaries of outermost clusters in $\Gamma_0$ visited by at least two different loops has dimension at most $2-\xi_c(4)$ a.s.\ for all $c\in(0,1]$, completing the proof for this case.
		
		Let us now prove~\eqref{eq:2-loops}. Fix $N\in \Nb$ such that $2^{-N} \le \eps/8$. Let $Q$ be a dyadic square with side length $2^{-N}$ and center $z_0$ such that $\Bc(z_0, \eps) \subset\Ub$. The union of all such dyadic squares covers $\Bc(0, 1-\eps)$. Therefore, to prove \eqref{eq:2-loops}, it is enough to prove that 
		\begin{align}\label{eq:2-loops2}
		\dimh(\Dc_0 \cap Q)\le 2-\xi_c(4) \, \text{ a.s. for all } c\in(0,1].
		\end{align}
		To prove~\eqref{eq:2-loops2}, it is enough to work on the event where both $\xi_1$ and $\xi_2$ intersect $Q$. Let $\mu$ be the measure of Brownian loops in the whole plane (so that $\mu_\eps$ is equal to $\mu$ restricted to the loops with diameter at least $\eps$ in $\Ub$). We can decompose $\mu$ according to the point on the loop which is farthest to $z_0$
		\begin{align*}%\label{eq:decomp}
		\mu=\frac{1}{\pi} \int_0^\infty \int_0^{2\pi} \mub_{\Bc(z_0, r)} (z_0+r e^{i\theta}) d\theta r dr.
		\end{align*}
		On the event $\{\xi \cap Q \not=\emptyset\}$, the measure $\mu_\eps$ is absolutely continuous with respect to
		\begin{align}\label{eq:wt-mu-eps}
		\wt\mu_\eps=\frac{1}{\pi} \int_{\eps/4}^2 \int_0^{2\pi} \mub_{\Bc(z_0, r)} (z_0+r e^{i\theta}) d\theta r dr.
		\end{align}
		Indeed, a loop $\xi$ sampled according to $\mu_\eps(\cdot \mid \xi \cap Q\not=\emptyset)$ reaches distance at least $\eps/4$ from $z_0$ and is contained in $\Ub \subset \Bc(z_0, 2)$. It is therefore enough to prove \eqref{eq:2-loops2} for $\xi_1$ and $\xi_2$ sampled independently according to $\wt\mu_\eps$.
		We can now use the set-up in Section~\ref{subsec:single_loop}, with the following notion of good squares in place of $(k,n)$-squares. For $j\ge N$ and $n\ge j+4$, we subdivide $Q$ into $2^{2(n-N)}$ squares of side length $2^{-n}$ as what we have done before for $S_0$. With slight abuse of terminology, we also call these sub-squares $n$-squares.
		Then, for each $n$-square $S$, we say that it is a good square if $\xi_1$ and $\xi_2$ both intersect $S$, and $S$ is not disconnected from infinity by $\xi_1\cup\xi_2\cup\Gamma_0$. Let $\Sfr_n$ be the set of all such good squares.
		As in the proof of Lemma~\ref{prop:(k,n)}, for each $n$-square $S$, we can decompose $\xi_1$ and $\xi_2$ to extract four crossings going in and out of $S$, and four links connecting the crossings. We can similarly deduce that there exists $C>0$ such that for any $n$-square $S$, 
		\begin{align*}
		\Pb(S \in\Sfr_n )\le C 2^{-n \xi_c(4)}.
		\end{align*}
		Arguing like in the proof of Proposition~\ref{prop:single_loop}, we can deduce \eqref{eq:2-loops2}. This completes the proof.
	\end{proof}

		\section{Zero-one law}\label{sec:0-1}
	The goal of this section is to prove a zero-one law for multiple points on the outer boundary of every outermost cluster in a loop soup $\Gamma_0$. 
	As a consequence, we obtain the following proposition, which contains the  result of Theorem~\ref{main-thm} for outer boundaries of outermost clusters in $\Gamma_0$.

	\begin{proposition}%[Zero-one law for a cluster]
		\label{prop:0-1}
Let $\Gamma_0$ be a Brownian loop soup with intensity $c\in(0,1]$ in the unit disk. Let $\Sc$ and $\Dc$ be the simple and double points of $\Gamma_0$. The following holds almost surely. For every outermost cluster in $\Gamma_0$ with outer boundary $\gamma$, for every portion $\ell_0$ of $\gamma$, we have
		\begin{align*}
		\dimh(\Sc \cap \ell_0) = 2-\xi_c(2), \quad \dimh(\Dc \cap \ell_0) = 2-\xi_c(4).
		\end{align*}
	\end{proposition}
	The proof of Proposition~\ref{prop:0-1} relies on a result from \cite{MR3901648} on the decomposition of the Brownian loop soup when we explore outer boundaries of its clusters in a Markovian way. 
The Markovian exploration is carried out using conformal maps, and we will repeatedly and implicitly use the following lemma.
\begin{lemma}\label{lem:conformal_dim}
Suppose $A \subset \ol \Ub$ and moreover $A \cap \partial \Ub$ only contains finitely many points. Suppose that $f$ is a continuous map defined on $\ol\Ub$ which is conformal on $\Ub$, then we have $\dimh(A)=\dimh (f(A))$.
\end{lemma}
\begin{proof}
For $\eps\in(0,1)$, let $A_\eps: =A \cap \ol B(0, 1-\eps)$, where $\ol B(0, 1-\eps)$ is the ball of radius $(1-\eps)$ centered at the origin. Since $A \cap \partial \Ub$ only contains finitely many points, we have 
$\dimh(A)=\lim_{\eps\to 0} \dimh(A_\eps)$. Since $f(A\cap \partial\Ub)$ also only contains finitely many points, we have $\dimh(f(A))=\lim_{\eps\to 0} \dimh(f(A_\eps)).$
For any fixed $\eps\in(0,1)$, $f'$ restricted to $\ol B(0, 1-\eps/2)$ is bounded and bounded away from $0$, so $\dimh(A_\eps) = \dimh(f(A_\eps))$. This completes the proof.
\end{proof}

In Section~\ref{subsec:exploration}, we will recall the partial exploration of the Brownian loop soup, and prove some first  results. In Section~\ref{subsec:complete_proof}, we will complete the proof of Proposition~\ref{prop:0-1}.
	
	\subsection{Decomposition of the Brownian loop soup via partial exploration of cluster boundaries}\label{subsec:exploration}
	In this subsection, we will recall a result from \cite{MR3901648} on the decomposition of the Brownian loop soup, which relies on a partial exploration process of the  outer boundaries of the outermost clusters.
	The latter collection is distributed as a CLE, and this exploration process is a version of the CLE exploration process defined in \cite{MR2979861}.
	
	Let us now describe one possible version of the CLE exploration process (see \cite{MR2979861} for more details). We place ourselves in the upper half-plane $\Hb$, and let $\Omega$ be a CLE in $\Hb$.
	Suppose that $\chi$ is a deterministic continuous curve from $0$ to $\infty$ parametrized by $\Rb^+$, such that for all $t>0$, $\Hb\setminus \chi([0,t])$ is simply connected.
	According to \cite{MR2979861}, for each $t>0$, conditionally on all the loops (in this paragraph, we only talk about CLE loops, not Brownian loops) that intersect $\chi([0, t])$, in each connected component of the complement of the union of the discovered loops and $\chi([0,t])$, there is again an independent CLE. 
	This allows us to explore the loops of $\Omega$ in their order of appearance when we move along the curve $\chi$.
	As we encounter a loop, instead of discovering the entire loop at once, we can also discover it progressively by tracing the loop in the clockwise direction until we close the loop. After that, we continue moving along $\chi$ and trace the next (infinitely many) loops as we encounter them. Let $\xi$ be the piece-wise right-continuous curve which is the concatenation of all the loops that we have traced in this process.
	We can view $\xi$ as a Loewner curve in $\Hb$ and parametrize it according to the half-plane capacity. For each $t>0$, let $K_t$ be the hull of $\xi([0,t])$. 
	Let $T$ be a stopping time with respect to the filtration generated by $\xi$. The results of  \cite{MR2979861} imply the following fact. If at time $T$, $\xi(T)$ belongs to some loop $\gamma_0$ in $\Omega$, then we denote by $\sigma(T)$ the time at which we start tracing $\gamma_0$. If $T>\sigma(T)$, then the conditional law of the rest of that loop (i.e., $\gamma_0\setminus \xi([\sigma(T), T])$) is that of an SLE$_\kappa$ from $\xi_T$ to $\xi_{\sigma(T)}$ in $\Hb\setminus K_T$.
	
	\begin{figure}
		\centering
		\includegraphics[scale=0.12]{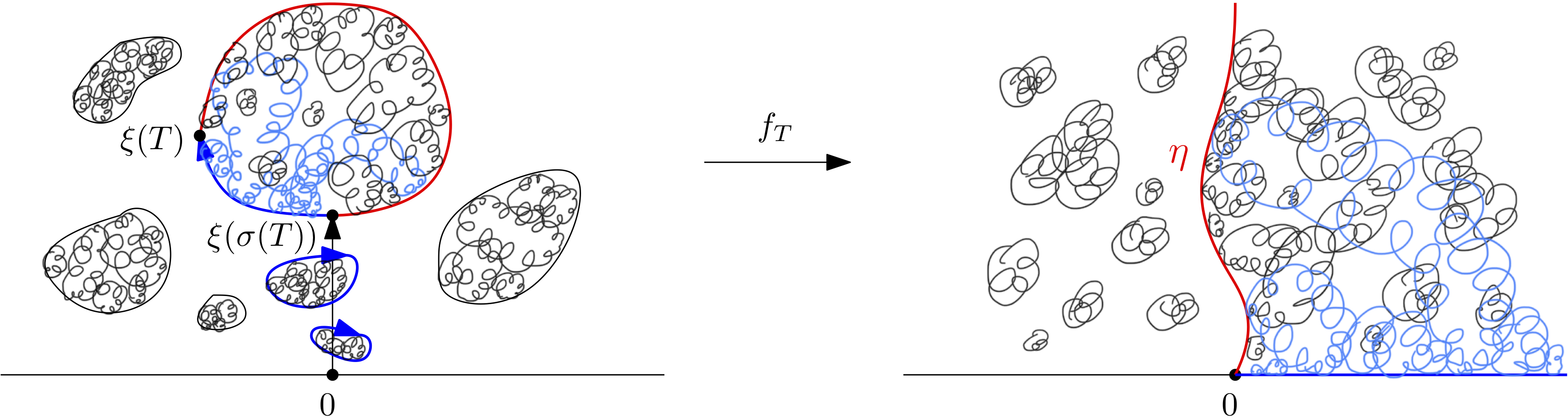}
		\caption{We explore the Brownian loop soup using the Markovian exploration of a CLE (which is the collection of outer boundaries of the outermost clusters). We illustrate the case where $\chi$ is the vertical half line starting from $0$. On the right, the loops in $\wt\Xi^1$ are in blue and the loops in $\wt\Xi^0$ are in black. The curve $\eta$ is in red.}
		\label{fig:partial_exploration}
	\end{figure}

	In \cite{MR3901648}, this exploration is applied to the Brownian loop soup. Let $\Xi$ be a Brownian loop soup in $\Hb$ such that the outer boundaries of its outermost clusters form a CLE $\Omega$. 
	Let $\xi$ be the Markovian exploration of $\Omega$ defined above, and let $T$ be a stopping time such that one is a.s.\ in the middle of tracing a CLE loop at time $T$ (i.e., $\xi_T\not=\xi_{\sigma(T)}$). Let $f_T$ be a conformal map from $\Hb\setminus K_T$ onto $\Hb$ that maps $\xi_T$ and $\xi_{\sigma(T)}$ to $0$ and $\infty$.
	Let $\Xi^1_T$ be the collection of loops in the loop soup that intersect $\xi([\sigma(T), T])$. Let $\Xi^0_T$ be the collection of loops in $\Hb\setminus K_T$ that are not in $\Xi^1_T$. See Figure \ref{fig:partial_exploration} for an illustration.
	The following is a reformulation of \cite[Theorem 1.6]{MR3901648} for the upper half-plane setting.
	
	\begin{lemma}[\cite{MR3901648} Theorem 1.6]\label{lem:qian}
		The sets $f_T(\Xi^1_T)$ and $f_T(\Xi^0_T)$ are independent and they are further independent from $\xi([0,T])$. Moreover, they satisfy the following properties:
		\begin{itemize}
			\item $f_T(\Xi^0_T)$ is a Brownian loop soup of intensity $c$ in $\Hb$.
			\item The union of loops in $f_T(\Xi^1_T)$ satisfies the one-sided chordal restriction. 
		\end{itemize}
	\end{lemma}
	We will not need the one-sided chordal restriction property of  $f_T(\Xi^1_T)$, but will need the fact that the law of $f_T(\Xi^1_T)$ is invariant under scaling, which is implied by the one-sided chordal restriction property (see \cite{MR3901648} for more details). To simplify notation, let $\wt\Xi^0:= f_T(\Xi^0_T)$ and $\wt\Xi^1:= f_T(\Xi^1_T)$.
	
	As we have mentioned earlier, the conditional law of $\xi$ from time $T$ up to the time that it completes the loop that it is tracing is an SLE$_\kappa$ in $\Hb\setminus K_T$. We denote by $\eta$ the image under $f_T$ of this part of $\xi$, which is an SLE$_\kappa$ in $\Hb$ from $0$ to $\infty$. We parametrize $\eta$ according to its half-plane capacity, and let $g_t$ be the conformal map from $\Hb\setminus\eta([0,t])$ onto $\Hb$ which sends $\eta(t)$ to $0$ and fixes $\infty$. Note that $\eta$ is also the left boundary of the union of $\wt\Xi^1$ with all the clusters in $\wt\Xi^0$ that $\wt\Xi^1$ intersects.
	We can continue to explore the collection of loops $\wt\Xi^1 \cup \wt\Xi^0$ along the curve $\eta$. 
	We parametrize $\eta$ according to its half-plane capacity. 
	For $t>0$, let $H^1_t$ be the union of $\wt\Xi^1$ with the collection of loops in $\wt\Xi^0$ that intersect $\eta([0,t])$. Let $H^0_t$ be the collection of loops in $\wt\Xi^0$ that do not intersect $\eta([0,t])$. Then $H^1_t\cup H^0_t =\wt\Xi^1 \cup \wt\Xi^0$ for all $t>0$.
	Lemma~\ref{lem:qian} implies the following result.
	
	\begin{lemma}\label{lem:eta_indep}
		The collections $g_t(H^0_t)$ and $g_t(H^1_t)$ are independent from each other and further independent from $\eta([0,t])$. Moreover, $g_t(H^0_t)$ is distributed as a Brownian loop soup in $\Hb$.
	\end{lemma}
	\begin{proof}
		This lemma directly follows from Lemma~\ref{lem:qian}. Indeed, the collections  $g_t(H^0_t)$ and $g_t(H^1_t)$ are equal to $f_{T'}(\Xi^1_{T'})$ and $f_{T'}(\Xi^0_{T'})$ in Lemma~\ref{lem:qian} for some stopping time $T'>T$.
	\end{proof}
	
	\begin{lemma}\label{lem:0-1-first}
Let $\Sc$ and $\Dc$ respectively be the set of simple and double points in the collection of loops $\wt\Xi^0 \cup \wt\Xi^1$. There exist two random variables $D_1$ and $D_2$, such that almost surely
$\dimh(\Sc\cap \eta([0,t])) = D_1$ and $\dimh(\Dc\cap \eta([0,t]))= D_2$ for all $t>0$. Moreover, $D_1$ and $D_2$ are measurable with respect to the $\sigma$-algebra generated by $\wt\Xi^1$.  
	\end{lemma}
	\begin{proof}
Let us prove the lemma for simple points, then the same proof applies to double points.
		The law of $\dimh(\Sc\cap \eta([0,t]))$ is the same for all $t>0$, because the law of $\wt\Xi^1 \cup \wt\Xi^0$ is invariant under scaling. The value of $\dimh(\Sc\cap \eta([0,t]))$ is increasing in $t$, so they should all be equal, namely there exists a random variable $D_1$ such that almost surely $\dimh(\Sc\cap \eta([0,t])) = D_1$ for all $t>0$.
		
		For $t>0$, let $\Fc_t$ be the $\sigma$-algebra generated by $\eta([0,t])$ and $g_t(H^1_t)$.
		Note that $\dimh(\Sc\cap \eta([0,t]))$ is measurable with respect to $\Fc_t$. It follows that $D_1$ is measurable with respect to $\cap_{t>0}\Fc_t$. Note that the germ $\sigma$-algebra generated by $\eta([0,t])$ is trivial, for it is the same as the germ $\sigma$-algebra of the Brownian motion which generates the driving function of $\eta([0,t])$. Moreover, if a loop in $\wt\Xi^0$ intersects $\eta([0,t])$ for all $t>0$, then it must also intersect $\eta(0)=0$, since it is a continuous curve. This implies that the germ $\sigma$-algebra generated by $g_t(H^1_t)$ is independent from $\wt\Xi^0$, and is the same as the $\sigma$-algebra generated by $\wt\Xi^1$. Therefore, $\cap_{t>0}\Fc_t$ is also equal to the $\sigma$-algebra generated by $\wt\Xi^1$.
	\end{proof}

	\subsection{Proof of Proposition~\ref{prop:0-1}} \label{subsec:complete_proof}
	To complete the proof of Proposition~\ref{prop:0-1}, we need to use the measure on \emph{one-point pinned complete clusters} defined in \cite{MR3901648}, which is closely related to the measure on one-point pinned CLE loops defined in \cite{MR2979861}. Given a collection of Brownian loops, we say that $\theta$ is a \emph{complete cluster}, if it is equal to the union of an outermost cluster and all the loops that are encircled by the outer boundary of this cluster.
	
	In the Markovian exploration of CLE defined in the previous subsection, let $\tau$ be the first time that $\xi$ reaches the loop $\gamma(i)$ in $\Omega$ encircling $i$. Let $h_{\tau}$ be the unique conformal map from $\Hb\setminus K_\tau$ onto $\Hb$ which sends $i$ and $\xi_\tau$ to $i$ and $0$, respectively. Then $h_\tau (\gamma(i))$ is a loop in $\Hb$ pinned at $0$ and encircling $i$. Let $\mu(i)$ be the probability measure of the loop $h_\tau (\gamma(i))$.
	The probability measure $\mu(i)$ can be extended to an infinite measure $\mu$ on simple loops pinned at $0$. Moreover, there exists $\beta>0$ such that $\mu$ can be characterized by the following properties (see  \cite{MR2979861} for more details): 
	\begin{itemize}
		\item The restriction of $\mu$ to loops encircling $i$ is equal to $\mu(i)$.
		\item (Conformal covariance) For any conformal transformation $\psi$ from $\Hb$ onto itself with $\psi(0)=0$, we have
		$
		\psi \circ \mu =|\psi'(0)|^{-\beta} \mu.
		$
	\end{itemize}

	We can explore a Brownian loop soup $\Xi$ in $\Hb$ using the same exploration process. Let $\theta(i)$ be the complete cluster in $\Xi$ which encircles $i$. Let $\nu(i)$ be the probability measure of $h_\tau(\theta(i))$. Then $\nu(i)$ can also be extended to an infinite measure $\nu$ on one-point pinned complete clusters. More precisely, $\nu$ can be reconstructed from $\mu\otimes \Qb$ as follows.
	\begin{itemize}
		\item Let $\Qb$ be the probability law of $f(\theta(i))$ where $f$ is the conformal map from the domain encircled by the outer boundary of $\theta(i)$ onto $\Ub$ which sends $i$ to $0$ with $f'(i)>0$.
		By \cite[Theorem 1]{MR3994105}, we know that $f(\theta(i))$ is independent from $\gamma$, and its law is invariant under all conformal maps from $\Ub$ onto itself.
		
		\item Suppose that $\gamma$ is a one-point pinned loop distributed according to $\mu$. Suppose that $\Theta$ is a collection of loops in $\overline \Ub$ distributed according to $\Qb$.
		Let $f_\gamma$ be a conformal map from the domain encircled by $\gamma$ onto $\Ub$ (we can choose an arbitrary normalization for $f$). Then $f_\gamma^{-1}(\Theta)$ is distributed according to $\nu$.
	\end{itemize}
	This construction ensures that $\nu$ also satisfies the conformal covariance property. More precisely, there exists $\beta>0$ such that for any conformal transformation $\psi$ from $\Hb$ onto itself with $\psi(0)=0$, we have
	\begin{align}\label{eq:conf_cov}
	\psi \circ \nu =|\psi'(0)|^{-\beta} \nu.
	\end{align}

Let us also record the following lemma which follows from \cite[Lemma 4.17]{MR3901648}.
\begin{lemma}[\cite{MR3901648} Lemma 4.17]\label{lem:loop_origin}
Let $\theta$ be a complete cluster sampled according to $\nu$. Then there is $\nu$ a.e.\ no Brownian loop in $\theta$ which touches the origin.
\end{lemma}

	\begin{figure}
		\centering
		\includegraphics[scale=.12]{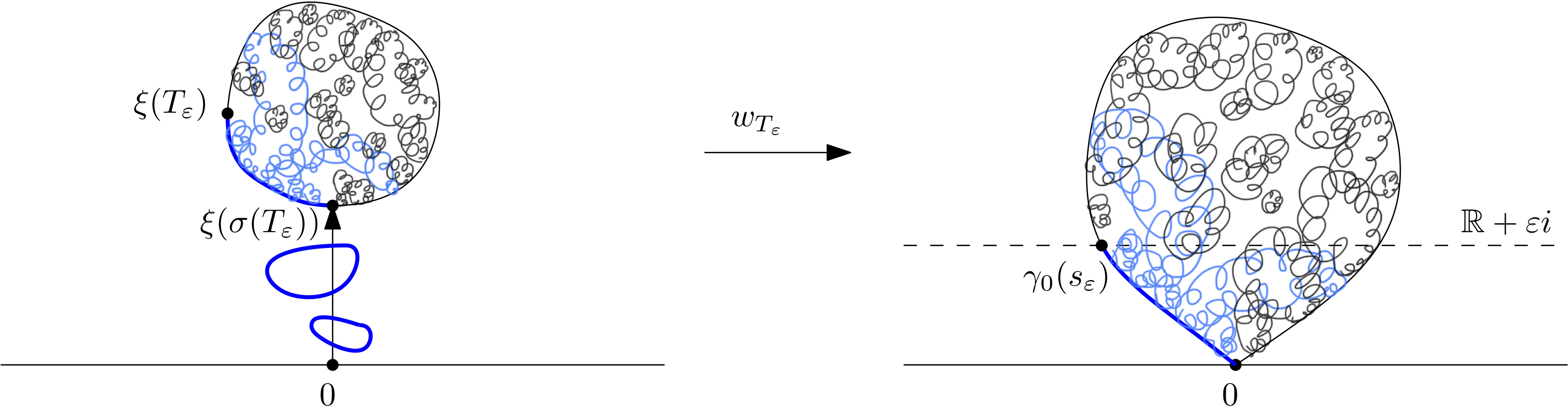}
		\caption{Proof of Lemma~\ref{lem:pinned-0-1}. On the right, we have $\theta_\eps$ which is a one-point pinned complete cluster conditioned to reach $\Rb+\eps i$. The $\sigma$-algebra $\Fc_\eps(\theta_\eps)$ is generated by $\gamma_0([0,s_\eps])$ (in blue) and all the Brownian loops that intersect $\gamma_0([0,s_\eps])$ (also colored in blue).}
		\label{fig:pinned_loop}
	\end{figure}	
	
	In order to prove Proposition~\ref{prop:0-1}, we first prove the following lemma.
	\begin{lemma}\label{lem:pinned-0-1}
		For $\nu$ a.e.\ one-point pinned complete cluster $\theta$, the following holds. Let $\gamma$ be the outer boundary of $\theta$. We parametrize $\gamma$ injectively and continuously by $[0,1]$ in the clockwise direction starting from $0$. For all $t\in(0,1]$, we have 
\begin{align*}
\dimh(\Sc \cap \gamma([0,t])) =2-\xi_c(2), \quad \dimh(\Dc \cap \gamma([0,t])) =2-\xi_c(4).
\end{align*}
	\end{lemma}
	
	\begin{proof}
		For each $t>0$, we define $w_t$ to be the unique conformal map from $\Hb\setminus K_{\sigma(t)}$ onto $\Hb$ which sends $\xi(\sigma(t))$ to $0$ and such that $w_t(z)=z+ O(1)$ as $z\to\infty$ (recall that $\sigma(t)$ is the first time that $\xi$ reaches the cluster that it is tracing at time $t$).
		For each $\eps>0$, let $T_\eps$ be the first time $t>0$ that $w_t(\xi ([\sigma(t), t]))$ reaches the horizontal line $\Rb+ \eps i$. 
		Let $\theta_\eps$ be the complete cluster that $\xi$ is tracing at time $T_\eps$.
		Then $w_{T\eps}(\theta_\eps)$ is distributed according to a probability measure $\Pf_\eps$ obtained from $\nu$ conditioned on the complete cluster intersecting $\Rb+\eps i$ (see  \cite{MR3901648,MR2979861} for more details). 
		
		Let $\gamma_0$ be the outer boundary of the one-point pinned complete cluster $w_{T\eps}(\theta_\eps)$, viewed as a simple loop rooted at $0$ and parametrized in a clockwise manner.  Then $w_{T_\eps}(\xi([\sigma(T_\eps), T_\eps]))$ is equal to the part of $\gamma_0$ from $0$ to the first time $s_\eps$ that it reaches $\Rb+\eps i$.  
		Let $S$ be the total time length of $\gamma_0$ (according to the half-plane capacity parametrization).
For $s_\eps< t \le S$, let $D_\eps (\theta_\eps, t)$ be the dimension of simple points on $\gamma_0([s_\eps, t])$.
		Let $\Fc_\eps(\theta_\eps)$ be the $\sigma$-algebra generated by $\gamma_0([0,s_\eps])$ and the collection of loops in $\theta_\eps$ that intersect $\gamma_0([0,s_\eps])$. See Figure \ref{fig:pinned_loop} for an illustration.
		Lemma~\ref{lem:0-1-first} implies that for all $s_\eps< t \le S$, $D_\eps(\theta_\eps,t)$ is equal to some $D_\eps(\theta_\eps)$ which is measurable with respect to $\Fc_\eps(\theta_\eps)$.

		Now, let $\Pf_1$ be the probability measure on one-point pinned clusters $\theta$ obtained from conditioning $\nu$ on the event that $\theta$ intersects $\Rb + i$. Let $\theta$ be a complete cluster with distribution $\Pf_1$. Let $\gamma$ be the outer boundary of $\theta$, viewed as a parametrized loop starting and ending at $0$, oriented clockwise.  Let $S$ be the total time length of $\gamma$. Fix $\eps\in(0,1)$. Let $s_\eps$ be the first time that $\gamma$ reaches $\Rb + \eps i$. 
		Let $\Fc_\eps(\theta)$ be the $\sigma$-algebra generated by $\gamma([0,s_\eps])$ and the collection of loops in $\theta$ that intersect $\gamma([0,s_\eps])$. We have the following facts.
 For $s_\eps< t \le S$, let $D_\eps(\theta, t)$ be the dimension of simple points on $\gamma([s_\eps, t])$.
			Since $\Pf_1$ can be obtained from $\Pf_\eps$ by conditioning on the complete cluster to reach $\Rb +i$, the previous paragraph implies that for all  $s_\eps< t \le S$, $D_\eps(\theta, t)$ is equal to some $D_\eps(\theta)$ which is measurable with respect to $\Fc_\eps(\theta)$. 
			Note that for each $s_\eps< t \le S$, $D_\eps(\theta, t)$ is increasing as $\eps\to 0$, and $\dimh(\Sc\cap \gamma([0,t])) = \lim_{\eps\to 0}D_\eps(\theta, t) =\lim_{\eps\to 0}D_\eps(\theta)$. This implies that $\dimh(\Sc\cap \gamma([0,t]))$ is the same for all $0<t\le S$, and is measurable with respect to $\cap_{\eps> 0}\Fc_\eps(\theta)$. Since $\gamma_0$ is an SLE curve, the germ $\sigma$-algebra generated by $\gamma_0([0,s_\eps])$ is the same as the germ $\sigma$-algebra of the Brownian motion which generates the driving function of $\gamma_0$, hence is trivial. Furthermore, thanks to Lemma~\ref{lem:loop_origin}, there is no Brownian loop in $\theta$ which intersects the origin. Combined, we can deduce that $\cap_{\eps> 0}\Fc_\eps(\theta)$ is the trivial $\sigma$-algebra. We have therefore proved that $\dimh(\Sc\cap \gamma([0,t]))$ is $\Pf_1$ a.s.\ equal to the same deterministic number, which must be $2-\xi_c(2)$ by Proposition~\ref{prop:dim-upper}, because a complete cluster sampled according to $\Pf_1$ can be obtained from exploring the loop soup in a Markovian way.		
		By scaling and \eqref{eq:conf_cov},  the conclusions of the previous items also hold for $\Pf_r$ instead of $\Pf_1$ for any $r>0$. This implies that for $\nu$ a.e.\ complete cluster $\theta$ with outer boundary $\gamma$, if $\gamma$ has total time length $S$, then for all $0<t\le S$, we have $\dimh(\Sc\cap \gamma([0,t])) = \dimh(\Sc\cap \gamma) = 2-\xi_c(2)$. Here, we have parametrized $\gamma$ according to its half-plane capacity, but the same result obviously holds for any parametrization of $\gamma$ which is continuous and injective from $[0,1]$ in the clockwise direction starting from the origin. The claim for double points follows from the same proof.
\end{proof}
	
We are now ready to complete the proof of Proposition~\ref{prop:0-1}. 
\begin{proof}[Proof of Proposition~\ref{prop:0-1}]
Let us now work in the unit disk $\Ub$. 
Let $\Gamma_0$ be a Brownian loop soup in $\Ub$. We remark that it suffices to prove the result for the outermost cluster in $\Gamma_0$ which encircles the origin, thanks to the conformal invariance of the loop soup.

Let $\Omega_0$ be the collection of outer boundaries of the outermost clusters in $\Gamma_0$, so that $\Omega_0$ is distributed as a CLE in $\Ub$. Let $\gamma_0$ be the loop in $\Omega_0$ which encircles the origin.
Let $\chi$ be any given deterministic curve from a point on $\partial \Ub$ to $0$ parametrized by $[0,1]$ such that for all $0<t\le 1$, $\Ub\setminus \chi([0,t])$ is simply connected. Let $o \in [0,1]$ be the first time that $\chi$ intersects the loop $\gamma_0$. We parametrize $\gamma_0$ from $\chi(o)$ back to itself in the clockwise direction injectively and continuously by $[0,1]$, so that $\gamma_0(0) =\gamma_0(1) = \chi(o)$. Then Lemmas~\ref{lem:pinned-0-1} and~\ref{lem:conformal_dim} imply that almost surely, for any $0< s\le 1$, we have 
$$\dimh(\Sc \cap \gamma_0([0,s])) =2-\xi_c(2), \quad \dimh(\Dc \cap \gamma_0([0,s])) =2-\xi_c(4),$$
where $\Sc$ and $\Dc$ are respectively the sets of simple and double points of $\Gamma_0$. 

Suppose now that the statement of Proposition~\ref{prop:0-1} is not true for $\gamma_0$, Without loss of generality, suppose that with positive probability, there exists a portion $\ell_0\subset \gamma_0$, such that
$$\dimh(\Sc \cap \ell_0) < 2-\xi_c(2).$$
Note that for any portion $\ell$ of $\gamma$, we always have $\dimh(\Sc\cap \ell) \le \dimh(\Sc \cap \gamma)$. 
The proof works the same way if we suppose $\dimh(\Dc \cap \gamma_0([0,s])) < 2-\xi_c(4)$ instead.
From this assumption, we deduce that there exists $n \in \Nb$ and a curve $\chi$ which follows the dyadic grid $2^{-n} \Zb^2$, such that with positive probability, $\chi(o)$ is in the interior of some portion $\ell_0$ with the property that $\dimh(\Sc \cap \ell_0) < 2-\xi_c(2)$. This contradicts the statement of the previous paragraph. Therefore, we have proved Proposition~\ref{prop:0-1}.
\end{proof}

\section{Proof of Theorem~\ref{main-thm}}\label{sec:main}
Theorem~\ref{main-thm} states that a set of properties hold almost surely for every boundary $\ell$ of every cluster in $\Gamma_0$.
In Proposition~\ref{prop:0-1}, we have proved that these properties hold when $\ell$ is the outer boundary of an outermost cluster. We will now extend this result to all boundaries of all clusters in the Brownian loop soup.

Let us first extend this result to inner boundaries of an outermost cluster.
\begin{lemma}\label{lem:inner_boundary}
Let $K$ be the outermost cluster encircling $0$ in $\Gamma_0$. Let $\ell$ be the boundary of the connected component of the complement of $K$ containing $0$. Then almost surely, for any portion $\ell_0$ of $\ell$, we have
\begin{align*}
\dimh(\ell_0 \cap \Sc)= 2- \xi_c(2), \quad \dimh(\ell_0 \cap \Dc)= 2- \xi_c(4), \quad \ell\cap \Tc=\emptyset.
\end{align*}
\end{lemma}
\begin{proof}
For each $r\in(0,1)$, let $E_r$ be the event that the following holds
\begin{itemize}
\item $K \cap B(0,r) =\emptyset$;
\item for the loop soup $\Gamma_0$ restricted to the annulus $\Ub \setminus B(0,r)$,  $K$ is the only cluster going around the annulus.
\end{itemize}
For each $r\in (0,1)$, $E_r$ has positive probability. Indeed, with positive probability, $\Gamma_0$ contains one big loop contained in $\Ub \setminus B(0, (1+r)/2)$ and going around the annulus, and all other loops of $\Gamma_0$ form clusters with diameter at most $r/2$. This event is contained in $E_r$.
We also have that 
\begin{align}\label{eq:union_E}
\Pb\left(\cup_{r\in (0,1)} \, E_r\right)=1.
\end{align}
To see \eqref{eq:union_E}, we define $$
\mbox{$R_1:=\sup\{r>0: K\cap B(0,r)=\emptyset\}$ and $R_2:=\sup\{r>0: B(0,r)\in {\rm hull}(K_2)\}$}
$$
where $K_2$ is the outermost cluster within the connected component of the complement of $K$ containing $0$ and $${\rm hull}(K_2):=\{x\in\mathbb{C}; \mbox{ $x$ is disconnected from infinity by $K_2$}\}$$ stands for the hull of $K_2.$
Then $\{R_2 < r < R_1\} \subseteq E_r$. Since $R_1> R_2$ a.s., the union of $\{R_2 < r < R_1\}$ over $r\in(0,1)$ covers the whole event space.

Let $\ell_1$ be the outer boundary of $K$. Then $\ell_1$ satisfies all the almost sure properties of Proposition~\ref{prop:0-1}. These properties remain a.s.\ true even after we condition on $E_r$.

On the event $E_r$, the loop soup in $\Ub \setminus B(0,r)$ contains a unique cluster $K$ which goes around the annulus, and the outer boundary of $K$ a.s.\ satisfy the properties of Proposition~\ref{prop:0-1}.
However, the law of the loop soup in $\Ub \setminus B(0,r)$ is invariant under the conformal map $z\mapsto r z^{-1}$. This implies that the inner boundary of $K$ (i.e., the boundary of the connected component of the complement of $K$ containing $0$) also satisfy the properties of Proposition~\ref{prop:0-1} a.s.

Thus we have obtained the statements of the lemma, conditionally on $E_r$. Then the lemma follows from \eqref{eq:union_E}. 
\end{proof}

We can now complete the proof of Theorem~\ref{main-thm}.
\begin{proof}[Proof of Theorem~\ref{main-thm}]
For each $z\in\Ub$, let $K_z$ be the outermost cluster encircling $z$ in $\Gamma_0$. 
By Lemma~\ref{lem:inner_boundary} and conformal invariance of the loop soup, we know that the properties of Proposition~\ref{prop:0-1} a.s.\ hold for the boundary of the connected component of the complement of $K_z$ containing $z$. Since this is true for all $z\in\Ub$,  the properties of Proposition~\ref{prop:0-1} a.s.\ hold for all the boundaries of all the outermost clusters in $\Gamma_0$.

We need to extend this result to those clusters in $\Gamma_0$ which are not outermost.
Suppose that $K$ is an outermost cluster in $\Gamma_0$. Conditionally on $K$, in each bounded connected component $O$ of the complement of $K$, $\Gamma_0$ restricted to $O$ is distributed as a loop soup in $O$. Therefore  the properties of Proposition~\ref{prop:0-1} also hold for the boundaries of the outermost clusters of $\Gamma_O$. By iterating the above argument, we  complete the proof of Theorem~\ref{main-thm}.
\end{proof}

\bibliographystyle{plain} 
\bibliography{cr}

\end{document}